\documentclass[reqno, a4paper,12pt]{amsart}

\usepackage[capposition=bottom]{floatrow}
\usepackage{euscript,verbatim}
\usepackage{graphicx}
\usepackage[usenames]{color}
\usepackage{xurl}
\usepackage[colorlinks,linkcolor=red,anchorcolor=blue,citecolor=blue]{hyperref}
\usepackage{amsmath, mathtools,scalerel,xparse}
\usepackage{mathtools}
\usepackage{amsthm}
\usepackage[all]{xy}
\usepackage{amssymb} 
\usepackage{bm}
\usepackage{booktabs}

\usepackage[all]{xy}

\usepackage{mathrsfs}
\usepackage{amscd}

\usepackage{enumitem}
\usepackage[numbers]{natbib}

\usepackage{mathptmx}
\usepackage[T1]{fontenc}

\usepackage{microtype,needspace}
\usepackage{titlesec}

\usepackage{placeins}

\DeclareMathAlphabet{\mathcal}{OMS}{cmsy}{m}{n}

\makeatletter
\numberwithin{equation}{section}

\theoremstyle{plain}
\newtheorem{main theorem}{Main Theorem}
\newtheorem{theorem}{Theorem}[section]
\newtheorem{lemma}[theorem]{Lemma}

\newtheorem{corollary}[theorem]{Corollary}
\newtheorem{proposition}[theorem]{Proposition}
\newtheorem{claim}[theorem]{Claim}

\newtheorem*{theorem*}{``Theorem''}
\theoremstyle{definition}
\newtheorem{definition}[theorem]{Definition}
\newtheorem{remark}[theorem]{Remark}
\newtheorem{example}[theorem]{Example}

\newtheorem*{example*}{Example}
\newtheorem*{remark*}{Remark}

\newtheoremstyle{break}
{\topsep}{\topsep}%
{\itshape}{}%
{\bfseries}{}%
{\newline}{}%
\theoremstyle{break}

\newtheoremstyle{break}
{\topsep}{\topsep}%
{\normalshape}{}%
{\bfseries}{}%
{\newline}{}%

\numberwithin{equation}{section}

\newcommand{\diam}{\mathrm{diam}}
\newcommand{\Span}{\mathrm{Span}}

\newcommand{\setcond}{\hspace{2pt} \middle| \hspace{2pt}}

\DeclareMathAlphabet{\pazocal}{OMS}{zplm}{m}{n}
\DeclareMathOperator{\supp}{supp}

\DeclareMathOperator{\dist}{dist}

\newcommand{\one}{\mathbf1}
\newcommand{\eps}{\varepsilon}
\newcommand{\Scal}{\mathcal S}

\newcommand{\Kcal}{\mathcal K}
\newcommand{\Ncal}{\mathcal N}

\newcommand{\Lcal}{\mathcal L}
\newcommand{\Ecal}{\mathcal E}
\newcommand{\Ucal}{\mathcal U}
\newcommand{\Cscr}{\mathscr C}
\newcommand{\Escr}{\mathscr E}

\newcommand{\Hscr}{\mathscr H}

\newcommand{\Mscr}{\mathscr M}
\newcommand{\Uscr}{\mathscr U}

\newcommand{\norm}[1]{\left\lVert#1\right\rVert}
\newcommand{\abs}[1]{\left\lvert#1\right\rvert}
\newcommand{\ip}[2]{\left\langle#1,#2\right\rangle}
\newcommand{\ld}{\underline{d}_{\,*}}
\newcommand{\ud}{\overline{d}_{*}}
\newcommand{\lud}{d_*}

\NewDocumentCommand{\Int}{e{_^}m}{\ThisStyle{%
  \mathsurround=0pt
  \setbox4=\hbox{$\SavedStyle\int$}%
  \setbox6=\hbox{$\SavedStyle\int\nolimits_{}$}%
  \setbox8=\hbox{$\SavedStyle#3$}%
  \setbox0=\hbox{$\SavedStyle
    \scalerel*[1.1\wd4]{\copy4}
      {\vphantom{\int}\scalebox{.9}{\copy8}}$}%
  \setlength{\dimen0}
    {(\wd4-\wd6+\scriptspace)*\ratio{\wd0}{\wd4}}%
  \mathop{\vcenter{\hbox{\copy0}}}\nolimits
  \IfValueT{#1}{_{\kern-\dimen0 #1}}%
  \IfValueT{#2}{^{#2}}#3%
}}

\newcommand{\RN}[1]{%
\textup{\uppercase\expandafter{\romannumeral#1}}%
}

\titleformat{\section}
{\normalfont\bfseries\Large} % font for section titles
{\thesection} % section number
{1em} % space between number and title
{\bfseries} % makes the title bold
\titleformat{\subsection}
{\normalfont\bfseries\normalsize}
{\thesubsection} 
{1em} 
{\bfseries}

\hypersetup{pdftitle={Constellations for Lyapunov exponents},pdfauthor={Nima Alibabaei}}
\begin{document}

\newcommand\titlelowercase[1]{\texorpdfstring{\lowercase{#1}}{#1}}

\font\mathptmx=cmr12 at 12pt

\title[\fontsize{13}{12}\mathptmx {\it{C\titlelowercase{onstellations for} L\titlelowercase{yapunov exponents}}}]{\LARGE C\titlelowercase{onstellations for} L\titlelowercase{yapunov exponents}}

\author[\fontsize{13}{12}\mathptmx {\it{N\titlelowercase{ima} A\titlelowercase{libabaei}}}]{\fontsize{13}{12}\mathptmx N\titlelowercase{ima} A\titlelowercase{libabaei}}



\begin{abstract}
We study the maximal Lyapunov exponent of random products of positive matrices, and prove that it can be calculated to arbitrary precision with arithmetic work polynomial in the number of requested digits. In nondegenerate settings, the complexity of our method depends on the intrinsic complexity of the dynamics---which we characterize---and not on the Euclidean dimension of the ambient space in which the dynamics live. As numerical applications, we determine the Hausdorff dimension of random intersections of Sierpi\'nski gaskets, and Viswanath's constant---the growth rate of the random Fibonacci sequence.

We first establish a general inequality for nonnegative matrix families showing that, if a signed measure with finite support is sufficiently close to being stationary, then it predicts the Lyapunov exponent accurately. This framework also encompasses Pollicott's Fredholm determinant approximation. We then show that, just like drawing a constellation by picking bright stars, one can construct a sequence of signed measures approximating stationarity by picking informative orbits of the projective dynamics. Furthermore, we prove that our construction uses an asymptotically optimal number of atoms, which is governed by an intrinsic quantity we call \emph{constellation dimension}.
\end{abstract}
\maketitle

\section{Introduction} \label{section:introduction}

Let $A_0,\ldots,A_{m-1}$ be $(d+1)\times(d+1)$ matrices, where $d\geq 1$ is an integer. Consider multiplying them randomly, choosing the indices independently with positive probabilities $\pi_0,\ldots,\pi_{m-1}$ which are fixed throughout. The asymptotic growth-rate,
\begin{equation} \label{eq: definition of lambda}
\lambda \;=\; \lim_{n\to\infty}\frac1n\log\norm{A_{i_1}\cdots A_{i_n}}
\end{equation}
exists and is constant almost surely, independent of the matrix norm \cite{Furstenberg--Kesten}. It is called the \emph{(maximal) Lyapunov exponent}. In this article, we develop a theory on the approximation of the concept of \emph{stationarity}. For (entrywise) positive matrices, we obtain a method to explicitly calculate $\lambda$ with a polynomial number of arithmetic operations in the number of required digits.

As an application, we determine the Hausdorff dimension of the typical random intersection of Sierpinski gaskets $\Scal\subset (\mathbb R/\mathbb Z)^2$. Kenyon--Peres \cite{Kenyon--Peres: translated Cantor sets} report a bound of \mbox{$0.97<\dim_H\big(\Scal\cap(\Scal+t)\big)<1.166$} for Lebesgue a.e.\ $t\in(\mathbb R/\mathbb Z)^2$. Our method yields, for a.e.\ $t$,
\[
\dim_H\!\big(\Scal\cap(\Scal+t)\big)
=
1.057167333781180286955509\cdots.
%\quad (\text{Lebesgue a.e.\ $t\in(\mathbb R/\mathbb Z)^2$})
\]

\subsection{Preliminaries}

Explicit calculation of Lyapunov exponents is known to be notoriously difficult. Kingman \cite[p.897]{Kingman} writes: “Pride of place among the unsolved problems of subadditive ergodic theory must go to the calculation of [the Lyapunov exponent]”. Peres posed the challenge: “Devise reasonably general and effective algorithms for explicit calculation (or at least approximation) of Lyapunov exponents.'' \cite[Question 3]{Peres92}. The problem has remained challenging to this date, even for seemingly simple cases of random matrix products.

An important vein of research was developed by Pollicott \cite{Pollicott10} for i.i.d.\ products of positive invertible matrices. It expresses $\lambda$ through Fredholm determinants of transfer operators, whose $n$-th coefficient as an entire function is determined by all matrix products of length at most $n$. The Taylor series converge rapidly, but their evaluation requires exponentially many words, resulting in \emph{subexponential} complexity in the number of requested digits; to achieve error $\eps>0$, one needs
\[
O\left(\exp\left(\alpha\big(\log(1/\eps)\big)^{\frac{d}{d+1}}\right)\right)
\]
arithmetic operations for some $\alpha>0$. For $2\times 2$ positive matrices, this direction was later developed further by Jurga--Morris \cite{Jurga--Morris} and again by Pollicott \cite{Pollicott20}, with the same arithmetic complexity. We note that the error bound given in \cite{Pollicott10} in higher dimensions ($d\geq2$) is not explicit. See the discussion in \cite[Section 6]{Jurga--Morris} for the difficulty of this approach in higher dimensions.

In a different direction, Bandtlow--Slipantschuk \cite{Bandtlow--Slipantschuk} develop a Lagrange-interpolation approximation of the transfer operator for positive invertible $2\times 2$ matrices with i.i.d.\ law. Their convergence estimate suggests polynomial arithmetic complexity, yet their Lyapunov-exponent error bound remains asymptotic. The method obtained by Protasov--Jungers \cite{Protasov--Jungers} treats nonnegative families, but could have doubly exponential work in $\log(1/\eps)$.

For $2\times 2$ positive families with i.i.d.\ dependence, the author discovered in \cite{Alibabaei26-1} a method to compute $\lambda$ explicitly with polynomial complexity; only $O\big(\big(\log(1/\eps)\big)^3\big)$ arithmetic operations are needed. Then, it was generalized to uniformly hyperbolic families driven by a Markov shift in \cite{Alibabaei26-2}. However, this approach is also essentially limited to the one-dimensional case ($2\times2$ matrices) as it uses rigid functions that ignore the dynamics. The curse of dimensionality makes a na\"ive approach using such functions practically infeasible even for $d=2$.

The purpose of this paper is to develop a method to compute $\lambda$ with polynomial complexity, for positive matrices of arbitrary dimension. To the author's knowledge, this is the first method to compute $\lambda$ for general families of $d\geq2$ in work polynomial in the number of requested digits. Our framework also explains Pollicott's Fredholm determinant approximations; see Remark~\ref{rmk: Fredholm determinant constellations}.

\subsection{Stationary measures}

The notion of \emph{stationary measures} encodes long-term information about the dynamics. Denote by $\norm{x}_1$ the sum of absolute values of entries of $x\in\mathbb R^{d+1}$, and let
\[
\Delta = \left\{x\in\mathbb R^{d+1} \setcond x_j\geq0,\ \norm{x}_1=1\right\}
\]
be a simplex, which we identify with the projective nonnegative cone. Denote its interior by $\Delta^\circ$. A nonzero, nonnegative matrix $A$ acts projectively as
\[
f_A:\Delta\setminus\ker_\Delta A\longrightarrow\Delta,
\qquad
f_A(x)=\frac{A x}{\norm{A x}_1}.
\]
Here $\ker_\Delta A \coloneqq \left\{ x\in\Delta\setcond Ax = 0\right\}$. If $A$ does not contain a zero row, $f_A$ maps $\Delta^\circ$ into itself.

In what follows, we assume that $A_0, \ldots, A_{m-1}$ are (entrywise) nonnegative and no $A_i$ contains a zero row. We will write $f_i\coloneqq f_{A_i}$ throughout this paper. We define the transfer operator $\Lcal$ for a continuous function $g:\Delta\to\mathbb R$ (or $g:\Delta^\circ\to\mathbb R$) by
\begin{equation}\label{eq:transfer operator}
\Lcal g(x)=\sum_i\pi_i \, g\circ f_i(x).
\end{equation}
For a finite signed measure $\nu$ on $\Delta$, we define $\Lcal^*\nu \coloneqq \sum_{i=0}^{m-1} \pi_i \nu\circ f_i^{-1}$.

\begin{definition}
A probability measure $\nu$ on $\Delta$ is said to be \emph{stationary} when $\Lcal^* \nu = \nu$.
\end{definition}

When all $A_0,\ldots,A_{m-1}$ are invertible, Furstenberg--Kifer \cite{Furstenberg--Kifer} prove
\begin{equation} \label{eq: variational principle for stationary measures}
\lambda
=
\sup_{\nu:\,\text{stationary}}\;
\nu(\Phi),
\qquad
\Phi(x) = \sum_{i = 0}^{m-1} \pi_i \log\norm{A_i x}_1.
\end{equation}
Here and henceforth, $\mu(f) \coloneqq \int_\Delta f\,d\mu$ for a signed measure $\mu$ and an integrable function $f$. Also, we set $\log0=-\infty$. Under positivity, the stationary measure $\nu$ is unique.

\subsection{Main results \RN{1}}

Our central insight is that, instead of approximating the stationary measure, we should approximate the concept of \emph{stationarity}. With this perspective, it is natural to consider signed measures, not just probability measures.
\begin{definition}
A signed measure $\eta$ on $\Delta^\circ$ is called a \emph{constellation} if it is finitely supported and $\eta(\Delta^\circ)  = 1$. In particular, for a natural number $n$ and $b\in\Delta^\circ$, an \emph{$n$-constellation with anchor $b$} is a signed measure $\eta$ on $\Delta^\circ$ such that
\[
\eta \;=\; \delta_b \;+\, \sum_{j=1}^n\alpha_j(\delta_{u_j}-\delta_{v_j}).
\qquad
\big( \alpha_j \in \mathbb R, \;\; u_j,v_j\in\Delta^\circ \big)
\]
Here, $\delta_x$ is the Dirac measure at $x\in\Delta$.
\end{definition}

Our method tries to simulate stationarity with sequentially evolving constellations. By introducing an auxiliary probability experiment involving joint exponential clocks, we arrive at the following kernel: for a \emph{reference} $a\in\Delta^\circ$,
\begin{equation}\label{eq:kernel}
K_a(x,y)
\;\coloneqq\;
\prod_{j=1}^{d+1} \frac{(x_j+a_j)(y_j+a_j)}{2a_j(x_j+y_j)},
\qquad x,y\in\Delta^\circ.
\end{equation}
Lemma \ref{lem: inner product of ell is the kernel} proves its positive-definiteness. A posteriori, we recognize $K_a$ as the restriction of the reproducing kernel of the Hardy space on the unit complex polydisc through the coordinatewise Cayley transform; see \S\S \ref{subsec: Connection to the Hardy kernel}. For a finitely supported signed measure $\sigma=\sum_jb_j\delta_{x_j}$ on $\Delta^\circ$, define
\begin{equation}\label{eq:kernel-norm}
 \norm{\sigma}_{K_a}^2 \;\coloneqq\; \sum_{j,k}b_jb_kK_a(x_j,x_k).
\end{equation}

\begin{theorem}[Clock inequality] \label{thm: Clock inequality}
Let $A_0,\ldots,A_{m-1}$ be nonnegative matrices, and suppose that no $A_i$ contains a zero row. Then for any constellation $\eta$ and $a\in\Delta^\circ$,
\begin{equation*}
\left| \lambda - \eta(\Phi) \right|^2
\;\leq\;
\frac{\pi^2}{2}\norm{\eta-\Lcal^*\eta}^2_{K_a}.
\end{equation*}
\end{theorem}

This shows that a constellation can accurately predict the infinite-time object $\lambda$ by being nearly invariant under a single step of the dynamics. This near-stationarity is quantified by the $K_a$-norm of the \emph{stationary defect} $\eta-\Lcal^*\eta$, which is explicitly calculable.

\smallskip

The stationary defect, measured with $K_a$, guides the discovery of informative points to add to the current constellation. For positive matrices, this yields a polynomial-complexity algorithm for computing the Lyapunov exponent. The construction of the constellations is simple and deterministic, and is given at the beginning of Section~\ref{section: construction of constellations}.

\begin{theorem} \label{thm: polynomial computation}
Fix any $a,b\in\Delta^\circ$ and a finite family of positive matrices. Then one can construct a sequence $\eta_1,\eta_2,\ldots$ where each $\eta_n$ is an $n$-constellation with anchor $b$ satisfying
\[
\norm{\eta_n-\Lcal^*\eta_n}_{K_a} \leq Ce^{-cn^{1/d}}
\]
for some $C,c>0$ independent of $n$. Furthermore, $\lambda$ can be approximated to error $\eps>0$ using
\begin{equation}\label{eq:universal-work}
O\left(\big(\log(1/\eps)\big)^{3d}\right)
\end{equation}
arithmetic operations.
\end{theorem}
In this paper, arithmetic operations are defined as field operations and evaluation of logarithms.

\begin{remark} \label{rmk: Fredholm determinant constellations}
For positive invertible matrices, Pollicott \cite{Pollicott10} truncates the Fredholm determinant of the transfer operator at order $N$: this yields a linear functional $\mathcal I_N$ on analytic functions that lets us approximate $\lambda$ by $\widehat\lambda_N=\mathcal I_N(\Phi)$. He proves (Theorem~\ref{thm: Pollicott superexponential decay}) that there are $C, \gamma>0$ with
\begin{equation*}
\abs{\lambda-\widehat\lambda_N}\leq C\exp\left(-\gamma \,N^{1+1/d}\right).
\end{equation*}
In Appendix~\ref{appendix: Fredholm determinants and constellations}, we will see that computing $\widehat\lambda_N$ amounts to determining the support and masses of a constellation $\eta_N^{\mathrm F}$ satisfying $\mathcal I_N=\eta_N^{\mathrm F}$. Moreover, for each $a\in\Delta^\circ$, there are $\widetilde C,\widetilde\gamma>0$ such that
\begin{equation*}
\norm{\eta_N^{\mathrm F}-\Lcal^*\eta_N^{\mathrm F}}_{K_a}\leq \widetilde C\exp\left(-\widetilde\gamma \,N^{1+1/d}\right).
\end{equation*}
The clock inequality therefore recovers Pollicott's rate of convergence and provides an explicitly calculable error bound for each of these approximations.

The Fredholm determinant method constructs $\eta_N^{\mathrm F}$ on all periodic orbits with period up to $N$. Our method instead reduces the stationary defect directly by selecting informative orbits. This lets us achieve substantially more precision with the same number of points for the examples we have tested. See Appendix~\ref{appendix: Comparison to Fredholm determinant approach} for a numerical comparison.
\end{remark}

\subsection{Main results \RN{2}}

The natural question is whether the constructed constellations are close to optimal. We will see that the answer is affirmative in the asymptotic sense, and the approximation rate is controlled by an intrinsic value we call \emph{constellation dimension}, rather than the ambient dimension $d$.

Let us define the minimal number of points needed to achieve an accurate approximation of stationarity. For the remainder of this subsection, we assume that all $A_i$ are positive. Take any $a\in\Delta^\circ$, and let
\begin{equation} \label{eq: minimal model size}
r_{\min}^{\,a}(\eps)
\;\coloneqq\;
\min\left\{\#\supp\eta \setcond
\eta\text{ is a constellation, and }
\norm{\eta-\Lcal^*\eta}_{K_a}\leq \eps
\right\}.
\end{equation}
We call this the \emph{minimal model size}. Here, $\#$ denotes the number of points of a finite set.

Asymptotically, this $r_{\min}^{\,a}$ admits a spectral interpretation. We define $k_x:\Delta^\circ\to\mathbb R$ by $k_x(y)~\coloneqq~K_a(x,y)$, and their inner product by
\[
\ip{k_x}{k_y} \coloneqq K_a(x,y).
\]
Let $\big(\Hscr_a, \ip{\cdot}{\cdot}_{\Hscr_a}\big)$ be the resulting reproducing-kernel Hilbert space generated by $\{k_x:\Delta^\circ\to\mathbb R \,:\, x\in\Delta^\circ\}$. We denote by $L^2(\nu)$ the set of real-valued functions that are square integrable against the stationary measure $\nu$, identified up to $\nu$-a.e.\ equality, and define the bounded linear operator
\[
T_a:\Hscr_a \longrightarrow L^2(\nu), \qquad T_a(u) = u - \Lcal(u).
\]
We denote by $T_a^*$ its adjoint operator, and define the stationary-defect covariance operator:
\[
\Cscr_a: \Hscr_a\longrightarrow\Hscr_a, \qquad \Cscr_a = T_a^*T_a^{}.
\]
Then
\[
\ip{u}{\Cscr_a u}_{\Hscr_a}
=\int_\Delta|u-\Lcal u|^2\,d\nu.
\qquad (u\in\Hscr_a)
\]

%The operator $\Cscr_a$ records which observations have appreciable stationary defect under the stationary law. It is positive and trace class, as proved in Section~\ref{sec:source}.

%Letting $s_x(y) \coloneqq k_x(y) - \sum_i\pi_i\,k_{f_ix}(y)$, for any continuous function $u:\Delta\to\mathbb R$, we have
%\[ \ip{u}{s_x}_{\mathcal H} = u(x) - \mathscr Lu(x). \]
% Therefore, $s_x$ is a functional that extracts the stationary defect of a function.

Lemma~\ref{lem: the covariance operator is positive and trace class} proves that $\Cscr_a$ is positive and trace class, so we apply the Spectral Theorem to $\Cscr_a$ and write the eigenvalues $\sigma_j = \sigma_j(\Cscr_a)$ as $\sigma_1\geq\sigma_2\geq\cdots\geq0$ in non-increasing order and with multiplicities, where we append zeros if there are only finitely many. For $t\geq0$, let
\[
\Ncal_a(t) \coloneqq 1+\#\left\{j\in \mathbb N \setcond \sigma_j\geq e^{-t}\right\}.
\]

\begin{definition}
The \emph{lower} and \emph{upper constellation dimensions} are, respectively,
\begin{equation}\label{eq: constellation dimensions}
\ld \;\coloneqq\; \liminf_{t\to\infty}\frac{\log\Ncal_a(t)}{\log t},
\qquad
\ud \;\coloneqq\; \limsup_{t\to\infty}\frac{\log\Ncal_a(t)}{\log t}.
\end{equation}
When they agree, their common value $\lud$ is called the \emph{constellation dimension}.
\end{definition}

\begin{proposition} \label{prop: constellation dimension universal bound}
For a finite family of positive matrices, the values of $\ld$ and $\ud$ are independent of the choice of $a$, and
\[
0\leq \ld \leq \ud \leq \dim_{\mathbb C} \overline{\Kcal}^{\mathrm{Zar}} \leq d.
\]
\end{proposition}
Here, $\Kcal$ is the unique nonempty compact set $\Kcal\subset\Delta^\circ$, called the \emph{attractor}, satisfying
\[
\Kcal = \bigcup_{i=0}^{m-1}f_i(\Kcal).
\]
Existence and uniqueness of such $\Kcal$ follow since the projective maps of positive matrices are contractive in the Hilbert metric (Lemma~\ref{lemma: Birkhoff-Hopf constants}). Also, $\overline{\Kcal}^{\mathrm{Zar}}$ is the Zariski closure of $\Kcal$ (Definition~\ref{def: Zariski closure}). For a Zariski-closed $V\subseteq k^n$ over a field $k$, we write $\dim_k V$ for the Krull dimension of the ring $k[V]$ of polynomial functions on $V$.

Constellation dimensions govern the rate of optimal approximation by constellations.
\begin{proposition} \label{prop: optimal cardinality is the constellation dimension}
For every positive matrix family and $a\in\Delta^\circ$,
\begin{equation*}
\liminf_{\eps\to0} \frac{\log r_{\min}^{\,a}(\eps)}{\log\log(1/\eps)} \;=\; \ld,
\qquad
\limsup_{\eps\to0} \frac{\log r_{\min}^{\,a}(\eps)}{\log\log(1/\eps)} \;=\; \ud.
\end{equation*}
\end{proposition}
The lower and upper constellation dimensions could be fractional and may not coincide; see Example \ref{ex: Liouville reset}. However, an invertible subfamily imposes a lower bound through the algebraic dimension of its attractor.

\begin{theorem} \label{thm: invertible subfamily lower bounds constellation dimension}
Let $A_0,\ldots,A_{m-1}$ be a family of positive matrices. If $\Kcal_0$ is the attractor of a nonempty subfamily consisting of invertible matrices, then
\[
\ld \geq \dim_\mathbb C\overline{\Kcal_0}^{\mathrm{Zar}}.
\]
\end{theorem}

Invertibility of the whole family therefore forces the following rigidity by Proposition~\ref{prop: constellation dimension universal bound}.
\begin{corollary} \label{cor: rigidity of constellation dimension}
For a finite family of positive invertible matrices, $\lud$ is a nonnegative integer satisfying
\[
\lud = \dim_{\mathbb C} \overline{\Kcal}^{\mathrm{Zar}}.
\]
\end{corollary}

When the anchor $b$ is in the attractor, the constellations we obtain in Theorem~\ref{thm: polynomial computation} are asymptotically optimal, in the sense that their asymptotic growth rate is the constellation dimension.

\begin{theorem} \label{thm: intrinsic complexity}
Let $A_0,\ldots,A_{m-1}$ be positive matrices. Fix $a\in\Delta^\circ$ and $b\in\Kcal$. For each $\eps>0$, let $\eta(\eps)$ be the first constellation in Theorem~\ref{thm: polynomial computation} with $\norm{\eta(\eps)-\Lcal^*\eta(\eps)}_{K_a}\leq \eps$. Then,
\begin{equation}\label{eq: optimality of the construction}
\liminf_{\eps\to0} \frac{\log \big( \# \supp \eta(\eps) \big)}{\log\log(1/\eps)} \;=\; \ld,
\qquad
\limsup_{\eps\to0} \frac{\log \big( \# \supp \eta(\eps) \big)}{\log\log(1/\eps)} \;=\; \ud.
\end{equation}

Assume in addition that there are two invertible matrices with distinct Perron-Frobenius vectors. Let $W(\eps)$ be the number of arithmetic operations required to achieve error at most $\eps$ using the implementation specified in \S\S\ref{subsection: proof of universal poly complexity}. Then,
\begin{equation}\label{eq: work count expressed through constellation dimension}
\liminf_{\eps\to0} \frac{\log W(\eps)}{\log\log(1/\eps)}=3\ld,
\qquad
\limsup_{\eps\to0} \frac{\log W(\eps)}{\log\log(1/\eps)}=3\ud.
\end{equation}
\end{theorem}

\subsection{Roadmap}

Section~\ref{section: examples} presents the examples. In Section~\ref{section: exponential clock and stationarity}, we represent the logarithmic observables using an auxiliary probability experiment based on joint exponential clocks. This representation yields the kernel $K_a$ and is used to prove the clock inequality (Theorem~\ref{thm: Clock inequality}).

\smallskip

Using the reproducing kernel Hilbert space obtained from $K_a$, Section~\ref{section: construction of constellations} constructs the constellations explicitly by devising a ``score'' search that navigates the selection of the next pair of points. In Section~\ref{sec: polynomial complexity}, polynomial approximation of the relevant space and a volume argument prove the convergence and the polynomial complexity in Theorem~\ref{thm: polynomial computation}. The volume argument follows the idea of DeVore--Petrova--Wojtaszczyk \cite{DeVore--Petrova--Wojtaszczyk}; see \S\S\ref{subsec: polynomial complexity overview} for the comparison.

\smallskip

Section~\ref{sec: constellation dimension} justifies the definitions of constellation dimensions, including their independence of the reference point $a$. Proposition~\ref{prop: constellation dimension universal bound} is proven here.

\smallskip

Section~\ref{sec: constellation dimension and minimality} then relates constellation dimensions to the minimal model size. The eigenspaces of $\Cscr_a$ provide a good approximation space for the constellations we constructed. The upper bound then follows by proving a H\"older-type interpolation inequality that controls supremum norms using stationary defects. We prove the lower bound by constructing an observation $u$ that vanishes on the support of a constellation $\eta$ and is visible to the stationary measure $\nu$, but at the same time has a controlled solution $h$ to the Poisson equation $u - \nu(u) = (I-\Lcal)h$. This proves Proposition~\ref{prop: optimal cardinality is the constellation dimension} and equation~\eqref{eq: optimality of the construction}.

\smallskip

Section~\ref{section: invertibility and algebraic dimension} proves Theorem~\ref{thm: invertible subfamily lower bounds constellation dimension}. Invertibility lets us pull back a polynomial on a small copy of the attractor. By considering an analogue of Lagrange interpolation polynomials, we can turn this into a bump-like function that prohibits cancellations, which yields the spectral lower bound. The section also completes the proof of Theorem~\ref{thm: intrinsic complexity}. 

In Appendix \ref{appendix: Fredholm determinants and constellations}, we will see how the Fredholm determinant method can be understood in our framework, and Appendix \ref{appendix: Comparison to Fredholm determinant approach} contains a numerical comparison to it. Appendices \ref{appendix: Reduction of the random Fibonacci sequence}, \ref{appendix: Constellation dimension for Sierpi\'nski gasket family}, and \ref{appendix: Liouville reset} supply the discussion on the examples.

\section{Examples} \label{section: examples}

\begin{example} \label{ex: the random Fibonacci sequence}
Viswanath defined the random Fibonacci sequence $\{f_n\}_{n=1}^\infty$ by $f_1=f_2=1$ and
\[
f_{n+1} = f_n \pm f_{n-1},
\]
where the sign is chosen independently to be plus or minus with probability $1/2$ each. In his celebrated paper \cite{Viswanath}, he calculated its growth rate to be, almost surely,
\[
\sqrt[n]{\abs{f_n}} \xrightarrow{n\to\infty} 1.13198824\cdots.
\]
This number $\rho$ is called \emph{Viswanath's constant}, and the previously most accurate calculation is $1.1319882487943$ by Bai \cite{Bai}.

Our method obtains the following certified number, displayed here to $128$ decimal places:
\begin{align*}
\rho \;=\;
&1.1319882487943061492200406895676898694177702561274706026177581431 \\
&\phantom{0.} 7138032486366457166421444015056805832041038452410722048305488722\cdots.
\end{align*}
Only a $587$-constellation\footnote{See Remark~\ref{rmk: visualizer and reproducible materials} for the GitHub link to the constellation.} with $623$ distinct support points was enough to certify the number above. The computation took about $5$ minutes on the author's laptop.

We need a nontrivial reduction to apply our method: first, we have
\[
\begin{pmatrix}
f_n & f_{n+1}
\end{pmatrix}
=
\begin{pmatrix}
f_{n-1} & f_n
\end{pmatrix}
B_*
\quad \text{for} \quad
B_+ 
=
\begin{pmatrix}
0 & 1 \\ 1 & 1
\end{pmatrix},
\quad
B_-
=
\begin{pmatrix}
0 & -1 \\
1 & 1
\end{pmatrix},
\]
with $*$ chosen to be $+$ or $-$ with probability $1/2$. Thus the Lyapunov exponent $\lambda_{\mathrm{orig}}$ of the symmetric random products of $\{B_+,B_-\}$ satisfies $\log \rho = \lambda_{\mathrm{orig}}$. 

Now let $\lambda_{\mathrm{red}}$ denote the Lyapunov exponent of the i.i.d.\ products of the following family $\{A_i\}_{i=0}^3$ with probabilities $\{\pi_i\}_{i=0}^3$:
\[
\begin{array}{c@{\;}c@{\qquad}c@{\;}c}
A_0=\begin{pmatrix}2&1\\1&1\end{pmatrix},
&\displaystyle \;\;\pi_0=\frac{7-3\sqrt5}{2},
&A_1=\begin{pmatrix}1&2\\0&1\end{pmatrix},
&\displaystyle \;\;\pi_1=\sqrt5-2,\\[14pt]
A_2=\begin{pmatrix}1&0\\2&1\end{pmatrix},
&\displaystyle \pi_2=\sqrt5-2,
&A_3=\begin{pmatrix}1&1\\1&2\end{pmatrix},
&\displaystyle \;\;\pi_3=\frac{3-\sqrt5}{2}.
\end{array}
\]
Appendix~\ref{appendix: Reduction of the random Fibonacci sequence} proves that
\[
\lambda_{\mathrm{orig}} \;=\; \dfrac{\sqrt5-1}{8} \,\lambda_{\mathrm{red}}.
\]
The value of $\rho$ displayed above was obtained by applying the construction of Section \ref{section: construction of constellations} to this nonnegative family $\{A_i\}_{i=0}^3$ and using Theorem \ref{thm: Clock inequality} to certify the error.

The family $\{A_i\}_{i=0}^3$ is not \emph{positive-conjugate}; no invertible matrix $P$ makes all $P^{-1}A_iP$ simultaneously positive. Nevertheless, we can prove that our construction works for it: $\norm{\eta_n-\Lcal^*\eta_n}_{K_a}\leq Ce^{-c\sqrt{n}}$, and that $W(\eps)=O\big(\big(\log(1/\eps)\big)\,^{\!6}\,\big)$. We omit the proof for conciseness.
\end{example}

\begin{example} \label{ex: Sierpinski gaskets}
Write $\mathbb T = \mathbb R/\mathbb Z$ and let $\Scal\subset\mathbb T^2$ be the Sierpi\'nski gasket. More precisely,
\[
\Scal \coloneqq
\left\{
\left( \sum_{n=1}^\infty \frac{a_n}{2^n}, \sum_{n=1}^\infty \frac{b_n}{2^n} \right) \in \mathbb T^2
\setcond
(a_n,b_n)\in \{(0,0),(1,0),(0,1)\} \text{ for every } n\geq1
\right\}.
\]
(Employing the definition using equilateral triangles does not alter the following conclusion.)

Kenyon--Peres \cite[Example~5.10]{Kenyon--Peres: translated Cantor sets} proved that, for Lebesgue a.e.\ $t\in\mathbb T^2$,
\[
\dim_H\big(\Scal\cap(\Scal+t)\big)
=
\frac{\lambda}{\log 2},
\]
where $\dim_H$ denotes the Hausdorff dimension, the intersection is taken in $\mathbb T^2$, and $\lambda$ is the Lyapunov exponent of the symmetric random products of the following $8\times8$ singular matrices:
\[
A_0 =
\begin{pmatrix}
3&0&0&0&0&0&0&0\\
0&1&0&0&1&0&0&0\\
1&0&1&0&0&0&0&0\\
0&1&1&0&0&0&0&0\\
1&0&0&0&2&0&0&0\\
0&1&1&0&1&0&0&0\\
1&1&1&0&0&0&0&0\\
0&1&1&0&1&0&0&0
\end{pmatrix},
\qquad
A_1 =
\begin{pmatrix}
1&0&1&0&0&0&0&0\\
0&1&1&0&0&0&0&0\\
0&0&3&0&0&0&0&0\\
0&0&1&1&0&0&0&0\\
1&1&1&0&0&0&0&0\\
0&0&2&0&0&1&0&0\\
0&0&2&0&0&0&1&0\\
0&0&2&0&0&0&0&1
\end{pmatrix},
\]
\[
A_2 =
\begin{pmatrix}
1&1&0&0&0&0&0&0\\
0&1&0&0&0&0&0&0\\
0&1&0&0&0&1&0&0\\
0&1&0&1&0&0&0&0\\
1&1&0&0&0&0&0&0\\
0&1&0&0&0&1&0&0\\
0&2&0&0&0&0&0&1\\
0&2&0&0&0&0&0&1
\end{pmatrix},
\qquad
A_3 =
\begin{pmatrix}
0&1&1&0&0&0&0&0\\
0&0&0&1&0&0&0&0\\
0&0&0&1&0&1&0&0\\
0&0&0&3&0&0&0&0\\
0&1&0&0&0&0&1&0\\
0&0&0&2&0&1&0&0\\
0&0&0&2&0&0&0&1\\
0&0&0&2&0&0&0&1
\end{pmatrix}.
\]

Kenyon--Peres \cite{Kenyon--Peres: translated Cantor sets} report a bound of $0.97<\dim_H<1.166$. Using our method, we certify the following first 24 decimal places for Lebesgue a.e.\ $t\in\mathbb T^2$.
\[
\dim_H\big(\Scal\cap(\Scal+t)\big)
=
1.057167333781180286955509\cdots.
\]
A $10850$-constellation\footnote{See Remark~\ref{rmk: visualizer and reproducible materials} for the GitHub link to the constellation.} with $12676$ distinct support points
%1.0571673337811802869555091 for an $11497$-constellation with $13382$ distinct support points
was used here, which took roughly a day for the author's laptop to compute. For this family, it can be proved that $\lud = 4$, $\norm{\eta_n-\Lcal^*\eta_n}_{K_a}\leq Ce^{-cn^{1/\lud}}$, and that $W(\eps) = O\big(\big(\log(1/\eps)\big)\,^{\!3\lud}\big)$. (The matrices are not positive-conjugate, but constellation dimensions can be defined with the same expressions.) See Appendix \ref{appendix: Constellation dimension for Sierpi\'nski gasket family} for an explanation of why $\lud = 4$.
\end{example}

% Accuracy of 10^{-24} corresponds to measuring one light-year—about 9.5 trillion kilometres—with an uncertainty of only 10 nanometres, thousands of times smaller than the width of a human hair.

\begin{remark}
The Fredholm determinant method by Pollicott \cite{Pollicott10} cannot be applied directly to the nonnegative matrix family (or any of its conjugations) in Example~\ref{ex: Sierpinski gaskets}. Indeed, we have $\det(xI-A_2)=x^3(x-1)^5$, so the leading eigenvalue of $A_2$ is not simple. This makes the trace formula of equation~\eqref{eq: Fredholm fixed-point trace formula} fail for this family, since
\begin{equation} \label{eq: Fredholm determinant formula for matrix}
\det\left(I-D\overline A(x)\right)
= \prod_{j=2}^{d+1} \left(1-\frac{\lambda_j(A)}{\lambda_1(A)}\right)
\end{equation}
is $0$ for $A = A_2$. Here $\overline A$ is the projective action of a matrix $A$, $D\overline A$ is its Jacobian, $x$ is a Perron--Frobenius direction, and the eigenvalues $\lambda_i(A)$ are counted with multiplicity.
\end{remark}

\begin{example}[Liouville reset] \label{ex: Liouville reset}
The conclusion of Corollary \ref{cor: rigidity of constellation dimension} can fail even for a positive $2\times2$ family with only one singular matrix. Let $q_1=2$, $q_{k+1}=2^{2^{q_k}}$, and $\alpha=1+\sum_{k=1}^\infty\frac1{q_k}$. Then $\alpha\approx 1.5625$ is a Liouville number. Consider the symmetric random product of
\[
 A_0 = \begin{pmatrix}9&3\\3&1\end{pmatrix},
 \qquad
 A_1 = \frac12\begin{pmatrix}3&1\\1&3\end{pmatrix},
 \qquad
 A_2 = \begin{pmatrix}
 1+2^{-\alpha}&1-2^{-\alpha}\\
 1-2^{-\alpha}&1+2^{-\alpha}
 \end{pmatrix}.
\]
Then $A_0$ has rank one, and $A_1,A_2$ are invertible. Appendix~\ref{appendix: Liouville reset} proves that, for this family,
\[
\ld=\frac12, \qquad \ud=\frac23, \qquad \dim_{\mathbb C} \overline{\Kcal}^{\mathrm{Zar}} = 1.
\]
\end{example}

\section{Exponential clock and stationarity} \label{section: exponential clock and stationarity}

The purpose of this section is to prove Theorem \ref{thm: Clock inequality}.

\subsection{Exponential clock}

The observables are functions of the form $\log(p^\top x)$, where $p\in\Delta^\circ$ represents the growth introduced by matrix products, and $x\in\Delta^\circ$ represents a point in the support of a constellation we want to investigate.

Suppose that $\eta$ is a constellation, and let $\sigma = \eta - \Lcal^* \eta$. Write $\sigma = \sum_k b_k\delta_{x^{(k)}}$, where $\delta_x$ denotes the Dirac measure at $x$. We will eventually integrate with respect to $\sigma$, which has mass $0$. Therefore, we can fix any $a\in\Delta^\circ$ and consider the centered version of the observable:
\[
\log(p^\top x) - \log(p^\top a) \;=\; \log \left( \frac{p^\top x}{p^\top a} \right).
\]

The key insight is this: \emph{we can represent $\log\big((p^\top x)/(p^\top a)\big)$ using an auxiliary probability experiment in which $x$ determines the law and $p$ determines an event.} The original stochastic process generates the random matrix products. The auxiliary experiment will let us control the observable for each fixed $p$ arising from those products.

First notice the following elementary integral identity:
\begin{equation} \label{eq: log is integral of exp}
\log \left( \frac{p^\top x}{p^\top a} \right) = \Int_0^\infty{\frac{e^{-s p^\top a}-e^{-s p^\top x}}{s}\,ds}.
\end{equation}
For now we investigate the exponentials inside the integral. Since we have $e^{-sp^\top x} = \prod_j e^{-sp_jx_j}$, it is natural to consider the following joint exponential-clock probability distribution:
\[
d\Escr_x(t) \;\coloneqq\; \prod_{j=1}^{d+1}x_je^{-x_jt_j}\,dt_j.
\]
Let $T_j$ be its waiting time for each $j = 1, \ldots, d+1$. Namely,
\[
\mathbb P_{\Escr_x}(T_j > t) = e^{-tx_j}. \quad (t\geq0)
\]
By independence,
\[
\mathbb P(T_j > sp_j \text{ for every } j) = \prod_j e^{-sp_jx_j} = e^{-sp^\top x}.
\]
Therefore, equation \eqref{eq: log is integral of exp} can be interpreted in terms of either of the distributions $\Escr_a$ and $\Escr_x$. However, neither proves to be useful.

An interesting identity is known for $\Escr_a$:
\begin{equation} \label{eq: def of normalized square-root likelihood}
\ell_{x,a}
\coloneqq \frac{\sqrt{\frac{d\Escr_x}{d\Escr_a}}}{\mathbb E_{\Escr_a} \left[ \sqrt{\frac{d\Escr_x}{d\Escr_a}} \right]}
= \frac{d\Escr_{(x+a)/2}}{d\Escr_a},
\end{equation}
which is readily verified with elementary calculus. (See \cite[Section~2.5, equation~(2.57)]{Amari} for a general formula.) Therefore, although the likelihood ratio $d\Escr_x/d\Escr_a$ may not be square-integrable, the midpoint version $\ell_{x,a}$ always is. Define the $p$-weighted, minimal waiting time by
\[
\tau_p \;\coloneqq \min_{1\leq j\leq d+1,\ p_j>0}\;\frac{T_j}{p_j}.
\]
We then have, for any $s\geq0$,
\[
\mathbb P_{\Escr_{(x+a)/2}}( \tau_p > 2s ) = \prod_j e^{-2sp_j \frac{x_j+a_j}{2}} = e^{-sp^\top(x+a)},
\]
which implies
\[
e^{-sp^\top x} \;=\; e^{sp^\top a} \; \mathbb E_{\Escr_{(x+a)/2}}\left[ \one_{\tau_p > 2s} \right].
\]
Here, $\one_X$ is the indicator function of a set $X$. Substituting back into equation \eqref{eq: log is integral of exp},
\[
\log \left( \frac{p^\top x}{p^\top a} \right)
=
\Int_0^\infty{ \mathbb E_{\Escr_{(x+a)/2}}\left[ \frac{e^{-sp^\top a}-e^{sp^\top a} \one_{\tau_p > 2s}}{s} \right] ds}
\]
We can interchange the integrals because the $\Escr_{(x+a)/2}$-expectation of the absolute value of the numerator is $O(s)$ near $0$, and exponentially decaying as $s\to\infty$. Letting
\[
G_{p,a} = \Int_0^\infty {\frac{e^{-sp^\top a}-e^{sp^\top a} \one_{\tau_p > 2s}}{s} ds},
\]
we have
\begin{equation} \label{eq: separation of p and x}
\log \left( \frac{p^\top x}{p^\top a} \right)
=
\mathbb E_{\Escr_{(x+a)/2}}\left[ G_{p,a} \right]
=
\mathbb E_{\Escr_a}\left[ G_{p,a} \ell_{x,a} \right].
\end{equation}
This middle expression cleanly separates $p$ and $x$, as promised. Moreover, all the effect of long-term matrix evolution, expressed through $p$, is compressed in $G_{p,a}$.

\begin{lemma} \label{lem: norm of G}
Denote by $\norm{\cdot}_{L^2(\Escr_a)}$ the $L^2$-norm with respect to the distribution $\Escr_a$. We have
\[
\norm{G_{p,a}}_{L^2(\Escr_a)}^2 = \frac{\pi^2}{2}.
\]
\end{lemma}

\begin{proof}
First restrict the double integrals to the interval $[1/R,\,R]$ with $R>0$. Fubini applies to its squared norm, and the expectation computed below is nonnegative and has finite integral. This justifies exchanging the order of integrations:
\begin{align*}
\norm{G_{p,a}}_{L^2(\Escr_a)}^2
&=
\Int_0^\infty{ \Int_0^\infty{ \mathbb E_{\Escr_a}\left[
\frac{1}{st} \big(e^{-s p^\top a}-e^{s p^\top a}\mathbf1_{\{\tau_p>2s\}}\big)
\big(e^{-t p^\top a}-e^{t p^\top a}\mathbf1_{\{\tau_p>2t\}}\big)
\right] ds}dt}\\
&=
\Int_0^\infty{ \Int_0^\infty{
\frac{1}{st} \left(
e^{-(s+t)p^\top a}
-e^{-(s+t)p^\top a}
-e^{-(s+t)p^\top a}
+e^{(s+t-2\max(s,t))p^\top a}
\right) ds}dt}\\
&=
\Int_0^\infty{ \Int_0^\infty{
\frac{1}{st} \left( e^{-|s-t|p^\top a}-e^{-(s+t)p^\top a} \right) ds}dt}.
\end{align*}
By symmetry, the integral is twice the integral on the triangle $0<t<s$. Letting $t=su$ and using Tonelli's Theorem yields
\begin{align*}
\norm{G_{p,a}}_{L^2(\Escr_a)}^2
&\;=\;
2 \Int_0^\infty{ \Int_0^s{ \frac{e^{-(s-t)p^\top a} - e^{-(s+t)p^\top a}}{st}dt}ds} \\
&\;=\;
2\Int_0^1{ \frac1u \Int_0^\infty{ \frac{e^{-s(1-u)p^\top a} - e^{-s(1+u)p^\top a}}{s}ds}du} \\
&\;=\;
2\Int_0^1{ \frac1u \log\frac{1+u}{1-u} du}
\;=\;
4 \sum_{n=0}^\infty \frac{1}{(2n+1)^2} \;=\; \frac{\pi^2}{2}. \qedhere
\end{align*}
\end{proof}

\smallskip

Thus, we have uniform control over the effect of matrix products. By integrating equation \eqref{eq: separation of p and x} with $\sigma=\sum_k b_k\delta_{x^{(k)}}$ over $\Delta$,
\[
\sigma\left( x\mapsto \log\left(\frac{p^\top x}{p^\top a}\right) \right)
=
\sum_k b_k \log \left( \frac{p^\top x^{(k)}}{p^\top a} \right)
%= \sum_k b_k \mathbb E_{\Escr_a}\left[ \ell_{x^{(k)},a} G_{p,a} \right]
=
\mathbb E_{\Escr_a}\left[ G_{p,a} \cdot \left( \sum_k b_k\ell_{x^{(k)},a} \right) \right].
\]
By the Cauchy-Schwarz inequality and Lemma \ref{lem: norm of G},
\begin{equation} \label{eq: integral wrt sigma is bounded by signed sum of ell}
\abs{ \sigma\left( x\mapsto \log\left(\frac{p^\top x}{p^\top a}\right) \right) }^2
\leq
\frac{\pi^2}{2} \norm{ \sum_k b_k\ell_{x^{(k)},a} }_{\Escr_a}^2.
\end{equation}
To evaluate the remaining norm, we need to determine $\mathbb E_{\Escr_a}\left[ \ell_{x^{(k)},a} \ell_{x^{(k')},a}\right]$.

%Now, we have $\norm{ \sum_j b_j\ell_{x,a} }_{\Escr_a}^2 = \sum_{j,k} b_j b_k \mathbb E_{\Escr_a}\left[ \ell_{x_j,a} \ell_{x_k,a}\right]$. Thus 

\begin{lemma} \label{lem: inner product of ell is the kernel}
For every $a,x,y\in\Delta^\circ$,
\[
\mathbb E_{\Escr_a}\left[ \ell_{x,a} \ell_{y,a}\right]
= \prod_{j=1}^{d+1} \frac{(x_j+a_j)(y_j+a_j)}{2a_j(x_j+y_j)}
\eqqcolon K_a(x,y).
\]
In particular, $K_a$ is a positive-definite kernel. That is, $\sum_{i,j} c_ic_j K_a(\zeta_i,\zeta_j) \geq 0$ for every $c_1,\ldots,c_L\in\mathbb R$ and $\zeta_1,\ldots,\zeta_L\in\Delta^\circ$.
\end{lemma}

\begin{proof}
By equation \eqref{eq: def of normalized square-root likelihood},
\begin{equation} \label{eq: explicit formula for ell x a}
\ell_{x,a}
=
\prod_{j=1}^{d+1} \frac{x_j+a_j}{2a_j} e^{-t_j(x_j-a_j)/2}
\end{equation}
Since the clocks are independent,
\begin{align*}
\mathbb E_{\Escr_a}\left[ \ell_{x,a} \ell_{y,a}\right]
&\;=\;
\prod_{j=1}^{d+1} \int_0^\infty \frac{x_j+a_j}{2a_j} \frac{y_j+a_j}{2a_j} e^{-t(x_j+y_j-2a_j)/2}a_j e^{-a_j t}\,dt \\
&\;=\; \prod_{j=1}^{d+1} \frac{(x_j+a_j)(y_j+a_j)}{4a_j} \int_0^\infty e^{-t(x_j+y_j)/2}dt
\;=\; \prod_{j=1}^{d+1} \frac{(x_j+a_j)(y_j+a_j)}{2a_j(x_j+y_j)}. \qedhere
\end{align*}
\end{proof}

Recall the definition of $\norm{\sigma}_{K_a}^2$ in equation \eqref{eq:kernel-norm}.

\begin{lemma} \label{lem: integral by sigma is bounded by norm} 
Let $p\in\Delta$, $a\in\Delta^\circ$. For any signed measure $\sigma$ that is finitely supported on $\Delta^\circ$ and $\sigma(\Delta)=0$,
\begin{equation*}
\abs{ \sigma\left( x\mapsto \log\left(\frac{p^\top x}{p^\top a}\right) \right) }^2
\leq
\frac{\pi^2}{2} \norm{\sigma}_{K_a}^2.
\end{equation*}
\end{lemma}

\begin{proof}
Since $\norm{ \sum_k b_k\ell_{x^{(k)},a} }_{\Escr_a}^2 = \sum_{k, k'} b_k b_{k'} \mathbb E_{\Escr_a}\left[ \ell_{x^{(k)},a} \ell_{x^{(k')},a}\right]$, the conclusion follows from equation \eqref{eq: integral wrt sigma is bounded by signed sum of ell} and Lemma \ref{lem: inner product of ell is the kernel}.
\end{proof}

\subsection{Proof of Theorem \ref{thm: Clock inequality}}

\begin{proof}
Take any constellation $\eta$ and let $\sigma = \eta - \Lcal^* \eta$. Pick any $x\in \Delta^\circ$. Recall that $\Phi(x) = \sum_{i = 0}^{m-1} \pi_i \log\norm{A_i x}_1$, and write, for a natural number $n$,
\[
S_n\Phi (x) \coloneqq \sum_{k=0}^{n-1} \Lcal^k \Phi(x),
\]
and $S_0\Phi \coloneqq 0$. Notice that
\[
\Lcal \Phi(x)
\;=\; \sum_{i,j} \pi_i\pi_j \log\norm{ A_i \frac{A_jx}{\norm{A_jx}_1} }_1
\;=\; \sum_{i,j} \pi_i\pi_j \log\norm{A_i A_j x}_1 - \sum_j \pi_j \log\norm{A_jx}_1,
\]
Therefore, by telescoping,
\[
S_n\Phi(x) = \sum_{w\,:\,|w|=n} \pi_w \log\norm{A_wx}_1,
\]
where for a word $w = (i_1,\ldots,i_n)$, we define $|w|\coloneqq n$, $\pi_w \coloneqq \pi_{i_1}\cdots\pi_{i_n}$, $A_w \coloneqq A_{i_1}\cdots A_{i_n}$, and $f_w \coloneqq f_{i_1}\circ\cdots\circ f_{i_n}$. Let $\widetilde{p_w} \coloneqq A_w^\top \one$ denote the row-sum vector of $A_w^\top$, where $\one = (1, \ldots, 1)^\top$. Let $p_w \coloneqq \widetilde{p_w}/\norm{\widetilde{p_w}}_1$.

Since $\sigma(\Delta)=0$, subtracting the constant $\log(\widetilde{p_w}^\top a)$ does not change the value of integrals. Also $\widetilde{p_w}^\top x /\widetilde{p_w}^\top a = p_w^\top x / p_w^\top a$. Then Lemma \ref{lem: integral by sigma is bounded by norm} gives, for every natural number $n$,
\begin{equation} \label{eq: bound on sigma of time average}
\abs{ \sigma(S_n\Phi) }
%= \abs{ \sum_{w\,:\,|w|=n} \pi_w \sigma\big(x\mapsto \log(p_w^\top x) \big)} \\
=
\abs{ \sum_{w\,:\,|w|=n} \pi_w \sigma\left( x\mapsto \log\left(\frac{p_w^\top x}{p_w^\top a}\right) \right)}
\leq
\frac{\pi}{\sqrt2} \norm{\sigma}_{K_a} \sum_{w\,:\,|w|=n} \pi_w
=
\frac{\pi}{\sqrt2} \norm{\sigma}_{K_a}.
\end{equation}
Since we have $S_n\Phi - \Lcal S_n\Phi = \Phi - \Lcal^n\Phi$, we have $\sigma(S_n\Phi) = \eta(\Phi) - \eta(\Lcal^n\Phi)$. For any $N\in\mathbb N$,
\[
\frac1N \sum_{n=0}^{N-1} \sigma(S_n\Phi)
\;=\;
\eta(\Phi) - \frac1N \eta(S_N\Phi).
\]
By equation \eqref{eq: bound on sigma of time average},
\begin{equation} \label{eq: bound on time average minus raw}
\abs{\frac1N \eta(S_N\Phi) - \eta(\Phi)} \leq \frac{\pi}{\sqrt2} \norm{\sigma}_{K_a}
\end{equation}

It remains to show $\frac1N S_N\Phi(x) \to \lambda$ for all $x\in\Delta^\circ$. Since each entry of a fixed $x\in\Delta^\circ$ is positive,
\[
\norm{B}_x \;\coloneqq\; \sum_{i,j} \,|B_{i,j}|\, x_j
\]
defines a norm on $(d+1)\times(d+1)$ matrices $B=(B_{i,j})_{i,j}$. Since $\frac1N \log \norm{A_w}_x$ with $|w|=N$ is uniformly bounded for $w$ and $N$ due to the no-zero-row assumption, a.s.-convergence in equation \eqref{eq: definition of lambda} implies $L^1$-convergence. Therefore,
\[
\frac1N S_N\Phi(x)
=
\mathbb E\left[ \frac1N \log \norm{A_w x}_1 \right]
=
\mathbb E\left[ \frac1N \log \norm{A_w}_x \right]
\xrightarrow{N\to\infty}
\lambda
\]
by the norm-independence in equation \eqref{eq: definition of lambda}. Letting $N\to\infty$ in equation \eqref{eq: bound on time average minus raw} proves the claim.
\end{proof}

\section{Construction of constellations} \label{section: construction of constellations}

\subsection{Construction}

Below we explain the construction of constellations. Intuitions are provided in the next subsection. Let $A_0,\ldots,A_{m-1}$ be positive matrices, and fix $a,b\in\Delta^\circ$. Then $\ker_\Delta A_i = \varnothing$ for all $i$.

\begin{enumerate}[label=\textup{\arabic*.},ref=\arabic*,leftmargin=*]
\item \textbf{Preparation.}
For $x\in\Delta^\circ$, we define the functions $k_x,s_x: \Delta^\circ\to\mathbb R$ by
\[
k_x \coloneqq K_a(x,\cdot),\qquad s_x \coloneqq k_x-\sum_i\pi_i k_{f_ix}.
\]
A pair $(u,v)\in (\Delta^\circ)^2$ is said to be \emph{diagonal} if $u=v$. We define $\rho,\theta:(\Delta^\circ)^2\to\mathbb R$ by
\[
\rho(u,v) \coloneqq \frac{\max_j(u_j/v_j)}{\min_j(u_j/v_j)}-1,
\!\qquad
\theta(u,v) \coloneqq \frac{\sum_i\pi_i\rho(f_iu,f_iv)}{\rho(u,v)}
\!\quad(u\ne v),
\!\qquad
\theta(u,v) \coloneqq 0
\!\quad(u = v).
\]

Let $X_0 \coloneqq b$ and $\eta_0 \coloneqq \delta_b$. In what follows we recursively define a tuple $X_n = (b,u_1,v_1,\ldots,u_n,v_n)\in(\Delta^\circ)^{2n+1}$ and an $n$-constellation $\eta_n$ with anchor $b$:
\begin{equation}\label{eq:anchored-constellation}
\eta_n \;=\; \delta_b + \sum_{j=1}^n\alpha_{n,j}(\delta_{u_j}-\delta_{v_j}).
\qquad
(\alpha_{n,j} \in \mathbb R)
\end{equation}

\item \textbf{Stationary equations.}
Given $X_n$ with $n\geq1$, we define the coefficients $\alpha_{n,j}$ by imposing stationarity on $\eta_n$ at $s_{u_j}-s_{v_j}$. That is,
\begin{equation}\label{eq: stationary equations}
\eta_n\big(s_{u_j}-s_{v_j}\big) = \eta_n\Big(\Lcal\big(s_{u_j}-s_{v_j}\big)\Big)
\qquad(1\le j\le n).
\end{equation}
Under our construction, these finite linear equations for $(\alpha_{n,1},\ldots,\alpha_{n,n})\in \mathbb R^n$ have a unique solution; see Lemma \ref{lem: solvability}.

\item \textbf{Enlarging the constellation.}
Given $X_n$ and $\eta_n$ with $n\geq0$, set $r_n:\Delta^\circ\to\mathbb R$ and $E_n\in\mathbb R$ as
\[
r_n \,\coloneqq\, s_b + \sum_{j=1}^n\alpha_{n,j}(s_{u_j}-s_{v_j}),
\qquad
E_n \,\coloneqq\, \norm{\eta_n - \Lcal^*\eta_n}_{K_a},
\]
with the convention $\sum_{j=1}^0 = 0$. If $E_n = 0$, $\eta_n$ satisfies $\lambda = \eta_n(\Phi)$, so we set $\eta_m = \eta_n$ for all $m>n$. Suppose $E_n > 0$. Let us define the ``score'' function
\[
q_n:(\Delta^\circ)^2 \longrightarrow \mathbb R, \qquad q_n(u,v) = r_n(u) - r_n(v).
\]
Since $E_n^2 = \sum_i\pi_iq_n(b,f_ib)>0$, we can pick a non-diagonal pair $(u,v)=(b,f_ib)$ maximizing $q_n/\rho$ among $i\in\{0,\ldots,m-1\}$. We accept the pair $(u,v)$ if
\begin{equation}\label{eq:acceptance}
\left|q_n(u,v)-\sum_i\pi_iq_n(f_iu,f_iv)\right| \geq \frac{1-\theta(u,v)}2\,q_n(u,v).
\end{equation}
Otherwise, not all the children $(f_iu,f_iv)$ are diagonal. Thus we can replace the current pair by the non-diagonal child pair $(f_iu,f_iv)$ maximizing $q_n/\rho$, and repeat. This score test succeeds after at most $O\big(\log(2+1/E_n)\big)$ updates of the initial pair; see Proposition~\ref{prop: score search terminates}. (In the selections above, choose the smallest index $i$ in a tie.)

For the accepted pair $(u_{n+1},v_{n+1})$, we set $X_{n+1} = \left(X_n, u_{n+1},v_{n+1}\right)$, and return to item $2$.
\end{enumerate}

\begin{remark}
Our construction was initially inspired by the following greedy basis selection in Banach spaces studied by DeVore--Petrova--Wojtaszczyk \cite{DeVore--Petrova--Wojtaszczyk}. It aims to approximate a compact subset $\mathcal F$ of a Banach space by finite-dimensional linear spaces. Fix $\gamma\in(0,1]$. After choosing $f_1,\ldots,f_n$, they assume that one can choose $f_{n+1}$ satisfying
\[
\dist(f_{n+1},V_n)\,\geq\,\gamma\,\max_{f\in\mathcal F}\dist(f,V_n)
\quad \text{for} \quad V_n \coloneqq \Span\{f_1,\ldots,f_n\}.
\]
They prove that $V_n$ approximates $\mathcal F$ at a rate controlled by the Kolmogorov widths of $\mathcal F$.
\end{remark}

\subsection{Intuition behind the construction} \label{subsection: intuition behind the construction}

To develop an intuition, let us first revisit the Hilbert space $\Hscr_a$. It has the following \emph{reproducing property}:
\begin{equation} \label{eq: reproducing property}
\text{\emph{For any $u\in\Hscr_a$ and $x\in\Delta^\circ$, one has $\ip{u}{k_x}_{\Hscr_a} = u(x)$.}}
\end{equation}
%\begin{proof} The conclusion is immediate if $u$ is a linear combination of $\{k_x \,|\, x\in\Delta^\circ\}$. Otherwise, let $u_n\to u$ in $\norm{\cdot}_{\Hscr_a}$, where each $u_n$ is a linear combination of $\{k_x \,|\, x\in\Delta^\circ\}$. By Cauchy-Schwarz,
%\[ \abs{u_n(x)-u_m(x)} = \abs{\ip{u_n-u_m}{k_x}} \leq \norm{u_n-u_m}_{\Hscr_a}\norm{k_x}_{\Hscr_a} \leq \sqrt{K_a(x,x)}\norm{u_n-u_m}_{\Hscr_a}. \] Thus $u(x) = \lim_nu_n(x) = \lim_n\ip{u_n}{k_x} = \ip{u}{k_x}$. \end{proof}

Therefore, $k_x$ is the Riesz representation of $\delta_x$ in $\Hscr_a$, and $s_x$ is that of the stationary defect $\delta_x - \Lcal^* \delta_x$.

Now imagine increasing $\alpha_{n,j}$ by $t$ in equation \eqref{eq:anchored-constellation}. The change in the constellation is $t\left(\delta_{u_j}-\delta_{v_j}\right)$, so the change in the stationary defect is represented by $t\left(s_{u_j}-s_{v_j}\right)$. Thus the residual $r_n$ represents the entire stationary defect of $\eta_n$. By definition,
\[
\norm{r_n}_{\Hscr_a} = \norm{\eta_n-\Lcal^*\eta_n}_{K_a}.
\]
Let
\[
V_n = \mathrm{Span}\left\{s_{u_j}-s_{v_j} \setcond 1\leq j\leq n\right\}.
\]
All $r_n$ attainable by changing the coefficients $\alpha_{n,j}$ form the affine space $s_b + V_n$. We choose the coefficients $\alpha_{n,j}$ so that $r_n$ is the point of this affine space nearest to zero, thus minimizing $\norm{r_n}_{\Hscr_a} = \norm{\eta_n-\Lcal^*\eta_n}_{K_a}$. Equivalently, $r_n\perp V_n$, or $\ip{r_n}{\,s_{u_j}-s_{v_j}}_{\Hscr_a} = 0$ for all $1\leq j\leq n$. This is the stationary equation imposed in equation 
\eqref{eq: stationary equations}.

\medskip

But how do we determine which points to add? Since constellations have to have mass one, the available corrections are of the form
\begin{equation} \label{eq: available change for constellations}
\eta_n \;\leadsto\; \eta_n + t(\delta_u-\delta_v), \qquad t\in\mathbb R, \;\;u,v\in\Delta^\circ.
\end{equation}
Since $(f_i)_*(\delta_u-\delta_v) = \delta_{f_i(u)}-\delta_{f_i(v)}$, it is natural to consider the space $(\Delta^\circ)^2$ of pairs $(u,v)$, equipped with the following random maps:
\[
(f_i\otimes f_i)(u,v) \coloneqq (f_iu, f_iv)
\text{\:\: is applied with probability $\pi_i$}.
\]

Now, let $\Ecal(\eta) \coloneqq \norm{\eta-\Lcal^*\eta}_{K_a}^2$ for a constellation $\eta$. We want to minimize $\Ecal$. It responds to the change in equation \eqref{eq: available change for constellations} as
\[
\Ecal\big(\eta_n+t(\delta_u-\delta_v)\big)
=
\Ecal(\eta_n) + 2t\ip{r_n}{\,s_u-s_v} + t^2\norm{s_u-s_v}_{\Hscr_a}^2.
\]
Thus the first-order ``usefulness'' of a pair $(u,v)$ can be measured with $\ip{r_n}{\,s_u-s_v}\eqqcolon D_n(u,v)$.
\begin{lemma} \label{lem: first-order derivative is stationary defect of score}
\[
D_n(u,v) \;=\; q_n(u,v) - \sum_i\pi_iq_n(f_iu,f_iv).
\]
\end{lemma}
\begin{proof}
Since $s_x = k_x - \sum_i\pi_i k_{f_ix}$, by the reproducing property of equation \eqref{eq: reproducing property},
\[
D_n(u,v)
=r_n(u) - \sum_i\pi_ir_n(f_i(u)) - \left( r_n(v) - \sum_i\pi_ir_n(f_i(v)) \right)
=q_n(u,v)-\sum_i\pi_iq_n(f_iu,f_iv). \qedhere
\]
\end{proof}

However, one-step usefulness itself neither predicts nor controls the usefulness of its future trajectories. Therefore, we consider the expectation of accumulated future variation: let $(U_k,V_k)$ be the $k$-th pair along the common random trajectory originating from $(u,v)$, and define
\[
\Ucal_n(u,v)
\;\coloneqq\; \mathbb E \left[ \sum_{k=0}^\infty D_n(U_k,V_k) \right]
\;=\; \sum_{k=0}^\infty \sum_{\abs{w} = k} \pi_w D_n\big(f_w u, f_w v\big).
\]

In fact, a direct calculation shows that
\begin{equation*} %\label{eq: score is usefulness}
\Ucal_n(u,v) \;=\; q_n(u,v).
\end{equation*}
Thus, $q_n(u,v)$ measures the potential future contributions of trajectories starting at $(u,v)$, and its stationary defect $D_n$ measures the immediate usefulness of incorporating the pair.

In the construction, we divide $q_n$ by the pair's geometric separation $\rho$ to normalize the size. We follow a trajectory by maximizing $q_n/\rho$, and accept a pair if it has enough $D_n$ relative to its potential.

\subsection{Proof of solvability}

The proof uses the contractions introduced by the positivity of the matrices. We use the Hilbert metric
\begin{equation*}
d_H(y,z)
\;\coloneqq\;
\log\!\left(\frac{\max_{1\leq p\leq d+1} y_p/z_p}{\min_{1\leq q\leq d+1} y_q/z_q}\right)
\qquad (y,z\in \mathbb R_{>0}^{d+1}).
\end{equation*}
Since $d_H(ty,z) = d_H(y,tz) = d_H(y,z)$ for all $t>0$, it also defines a metric on the projective positive cone. 

For a positive matrix $A=(a_{pq})_{1\leq p,q\leq d+1}$, let
\[
\Gamma(A) \coloneqq \max_{1\leq p,q,r,s\leq d+1} \frac{a_{pr} a_{qs}}{a_{ps} a_{qr}},
\qquad
\kappa(A) \coloneqq \frac{\sqrt{\Gamma(A)}-1}{\sqrt{\Gamma(A)}+1}.
\]

\begin{lemma} \label{lemma: Birkhoff-Hopf constants}
For every positive matrix $A$, one has $0\leq \kappa(A) <1$, and for every $x,y\in \mathbb R_{>0}^{d+1}$,
\[
d_H\big(Ax,Ay\big) \leq \kappa(A) d_H(x,y).
\qquad d_H(A^\top x, A^\top y) \leq \kappa(A) d_H(x,y).
\]
Moreover, $\kappa(A)$ is the optimal Lipschitz constant for $A, A^\top: \mathbb R_{>0}^{d+1}\to \mathbb R_{>0}^{d+1}$ with respect to $d_H$.
\end{lemma}

\begin{proof}
This is an application of the Birkhoff-Hopf theorem. Since $A=(a_{pq})_{1\leq p,q\leq d+1}$ is positive, we have $\Gamma(A)\geq 1$, hence $\kappa(A)\in [0,1)$. By the definition of $d_H$, we can scale $x,y$ and assume $x,y\in \Delta^\circ$. Set $u=A x$, $v=A y$. Then
\[
u_p v_q
=
\sum_{r,s=1}^{d+1} x_r y_s\, a_{pr} a_{qs}.
\qquad (1\leq p,q\leq d+1)
\]
By the definition of $\Gamma(A)$, we have $a_{pr} a_{qs} \leq \Gamma(A)\, a_{ps} a_{qr}$ for every $1\leq p,q,r,s\leq d+1$. Hence
\[
u_p v_q
\leq
\Gamma(A) \sum_{r,s=1}^{d+1} x_r y_s\, a_{ps} a_{qr}
=
\Gamma(A)\, \left(  \sum_{r=1}^{d+1} a_{qr}x_r \right) \left( \sum_{s=1}^{d+1} a_{ps}y_s \right)
=
\Gamma(A)\, u_q v_p.
\]
Therefore $\dfrac{u_p/v_p}{u_q/v_q} \leq \Gamma(A)$ for $1\leq p,q\leq d+1$, and so $ d_H(f_A(x),f_A(y)) \leq \log \Gamma(A)$. Thus, $\diam_{d_H}\big( f_A(\Delta^\circ) \big) \leq \log \Gamma(A)$. To prove the reverse inequality, choose indices $p,q,r,s$ attaining the maximum in the definition of $\Gamma(A)$. Write $e_i\in \Delta$ for the standard basis vectors. Then
\[
d_H\big(f_A(e_r),f_A(e_s)\big) = \log \left( \frac{\max_{\alpha} a_{\alpha r}/a_{\alpha s}} {\min_{\beta} a_{\beta r}/a_{\beta s}} \right) \geq \log \left( \frac{a_{pr}/a_{ps}}{a_{qr}/a_{qs}} \right) = \log \Gamma(A).
\]
This proves $\diam_{d_H}\big( f_A(\Delta^\circ) \big) = \log \Gamma(A)$. Then
\[
\tanh\left( \frac{1}{4}\diam_{d_H}\big( f_A(\Delta^\circ) \big) \right)
=
\tanh\left( \frac{1}{4}\log \Gamma(A) \right)
=
\frac{\sqrt{\Gamma(A)}-1}{\sqrt{\Gamma(A)}+1}
=
\kappa(A).
\]
The assertion follows by applying the Birkhoff-Hopf theorem \cite[Theorem 3.2]{Bushell} to $A$. The same computation applies to $A^\top$, since the cross-ratio maximum defining $\Gamma(A)$ is unchanged by transposition.
\end{proof}

Choose $0<\kappa<1$ such that $\kappa(A_i) = \kappa(A_i^\top) \leq\kappa$ for every $i$.

\begin{lemma} \label{lem: basic inequalities for rho and theta}
For every $x,y\in\Delta^\circ$,
\[
\norm{x-y}_1\leq\rho(x,y),
\qquad
\rho(f_ix,f_iy)\leq\kappa\rho(x,y),
\qquad
0\leq\theta(x,y)\leq\kappa<1.
\]
\end{lemma}

\begin{proof}
Since $\sum_jx_j = \sum_jy_j = 1$, we have $\min_j (x_j/y_j) \leq 1 \leq \max_j (x_j/y_j)$. Then,
\[
0 \;\leq\; \frac1{\min_j (x_j/y_j)} -1 \;\leq\; \frac{\max_j (x_j/y_j)}{\min_j (x_j/y_j)} -1 \;=\; \rho(x,y).
\]
By symmetry, we obtain $\abs{x_k/y_k -1 } \leq \rho(x,y)$ for every $k$. Multiplying by $y_k$ and summing over $k$ proves $\norm{x-y}_1\leq\rho(x,y)$. Next, since $\rho(x,y)=e^{d_H(x,y)}-1$, the second inequality follows from Lemma \ref{lemma: Birkhoff-Hopf constants} and $e^{\kappa t}-1\leq\kappa(e^t-1)$ for $t\geq0$. As a consequence we obtain the inequality for $\theta$.
\end{proof}

Now we can prove the solvability of the stationary equations.

\begin{lemma} \label{lem: solvability}
Equation \eqref{eq: stationary equations} has a solution $(\alpha_{n,1},\ldots,\alpha_{n,n})\in \mathbb R^n$ and it is unique.
\end{lemma}

\begin{proof}
Equation \eqref{eq: stationary equations} is equivalent to $\ip{r_n}{\,s_{u_j}-s_{v_j}}_{\Hscr_a} = 0$ for all $1\leq j\leq n$. Since $V_n = \mathrm{Span}\left\{s_{u_j}-s_{v_j} \setcond 1\leq j\leq n\right\}$ is a finite-dimensional subspace of the Hilbert space $\Hscr_a$, there is an orthogonal projection $v_b$ of $s_b$ onto it. The vector $s_b - v_b$ is orthogonal to $V_n$, and this gives the coefficients $\alpha_{n,j}$.

Next we prove uniqueness. At the enlargement step of $X_n$ to $X_{n+1}$, let us see that every intermediate pair $(u,v)$ satisfies $q_n(u,v)>0$ by induction. The initial pair $(u,v)$ satisfies $q_n(u,v)>0$ due to $E_n^2 = \sum_i\pi_iq_n(b,f_ib)>0$. If the score test failed for some $(u,v)$, then
\begin{equation} \label{eq: failed test}
\sum_i\pi_iq_n(f_iu,f_iv) \geq \frac{1+\theta(u,v)}2\,q_n(u,v) > 0.
\end{equation}
Therefore, the chosen child pair also satisfies $q_n>0$. Thus an accepted pair $(u,v)$ satisfies, by Lemma~\ref{lem: first-order derivative is stationary defect of score},
\[
\abs{\ip{r_n}{\,s_u-s_v}}
= \abs{D_n(u,v)}
\geq \frac{1-\theta(u,v)}2\,q_n(u,v) > 0
\]

This implies that each enlargement adds a pair $(u,v)$ whose corresponding direction $s_u-s_v$ lies outside of the current span $V_n$; thus the vectors $\left\{s_{u_j}-s_{v_j} \setcond 1\leq j\leq n\right\}$ are linearly independent, making the coefficients of the orthogonal projection unique.
\end{proof}

\subsection{Proof of termination}

Our next objective is to prove that the score search in Step $3$ of the construction terminates.

Write $A_i = \left( a^{(i)}_1, \ldots, a^{(i)}_{d+1}\right)$, where each $a^{(i)}_j$ is its $j$-th column vector. Let $\gamma^{(i)}_j \coloneqq a^{(i)}_j\big/\|a^{(i)}_j\|_1$, and let $S\subset\Delta^\circ$ be the convex hull of $\{a,b\}\cup\big\{ \gamma^{(i)}_j\,|\, 0\leq i\leq m-1, \,1\leq j\leq d+1\big\}$. Then $S$ is compact and $f_i(\Delta)\subset S$ for every $i$.

To give an overview, suppose that from an intermediate, rejected pair $(u,v)$, we chose a child $(f_iu,f_iv)$. The important fact is that positivity of matrices makes $\rho$ decay with each forward map, and the failure of the score test in equation~\eqref{eq:acceptance} forces $(f_iu,f_iv)$ to retain a sufficiently large fraction of $q_n(u,v)$: Lemma~\ref{lem: normalized future usefulness is increasing} proves that their ratio is strictly increasing.
\begin{equation} \label{eq: normalized future usefulness is increasing}
\frac{q_n(f_iu,f_iv)}{\rho(f_iu,f_iv)}
\geq
\frac{1+\kappa}{2\kappa} \cdot \frac{q_n(u,v)}{\rho(u,v)},
\end{equation}
where $(1+\kappa)/(2\kappa) > 1$. At the same time, this ratio is bounded from above: Proposition~\ref{prop: score search terminates} proves that there is some $C>0$ such that for any non-diagonal $(u,v)\in S^2$,
\begin{equation} \label{eq: normalized future usefulness is bounded}
\frac{\abs{q_n(u,v)}}{\rho(u,v)} \leq C E_n
\end{equation}
Combined, these imply that the search cannot continue indefinitely.

\begin{lemma} \label{lem: normalized future usefulness is increasing}
Equation~\eqref{eq: normalized future usefulness is increasing} holds true for every intermediate pair $(u,v)$ and its selected child $(f_iu,f_iv)$ during Step $3$ in the construction.
\end{lemma}

\begin{proof}
If $(u,v)$ is rejected and $\theta(u,v)=0$, then all children are diagonal, making equation~\eqref{eq:acceptance} true, a contradiction. Thus $\theta(u,v)>0$. Now, for every non-diagonal $(f_ju,f_jv)$,
\[
q_n(f_ju,f_jv)
\leq
\frac{q_n(f_iu,f_iv)}{\rho(f_iu,f_iv)}
\,\rho(f_ju,f_jv).
\]
Multiplying by $\pi_j$ and summing over $j$,
\[
\sum_j\pi_jq_n(f_ju,f_jv)
=
\sum_{j:\ \rho(f_ju, f_jv)>0}\pi_jq_n(f_ju,f_jv)
\leq
\frac{q_n(f_iu,f_iv)}{\rho(f_iu,f_iv)}
\,\theta(u,v)\rho(u,v).
\]
Then by equation~\eqref{eq: failed test} and Lemma~\ref{lem: basic inequalities for rho and theta},
\[
\frac{q_n(f_iu,f_iv)}{\rho(f_iu,f_iv)}
>\frac{1+\theta(u,v)}{2\theta(u,v)}\frac{q_n(u,v)}{\rho(u,v)}
\geq\frac{1+\kappa}{2\kappa}\frac{q_n(u,v)}{\rho(u,v)}. \qedhere
\]
\end{proof}

\bigskip

Our next objective is to prove boundedness of $q_n/\rho$. Let $L_{\mathrm{clock}}^2(\Escr_a)$ be the closed subspace of $L^2(\Escr_a)$ generated by the functions $\ell_{x,a}$ with $x\in \Delta^\circ$. Lemma \ref{lem: inner product of ell is the kernel} shows that $\ip{\ell_{x,a}}{\ell_{y,a}}_{\Escr_a} \coloneqq \mathbb E_{\Escr_a}\left[\ell_{x,a}\ell_{y,a}\right] = \ip{k_x}{k_y}_{\Hscr_a}$. Thus $k_x\mapsto \ell_{x,a}$ extends to a linear isometry from $\Hscr_a$ onto $L_{\mathrm{clock}}^2(\Escr_a)$. Denote the isometry by
\begin{equation} \label{eq: clock isometry}
J: \Hscr_a \xrightarrow{\cong} L_{\mathrm{clock}}^2(\Escr_a).
\end{equation}

Also, since $S$ is compact, it has a relatively open neighborhood $U\subset \Delta^\circ$ such that there is $\delta>0$ for which any $x=(x_1,\ldots,x_{d+1})\in U$ satisfies $x_j\geq \delta > 0$ for all $j$.

\begin{lemma} \label{lem: bounded derivatives of features}
The map $\Delta^\circ\ni x\mapsto k_x\in\Hscr_a$ is Fr\'echet differentiable any finite number of times on a neighborhood of $S$, and its derivatives of each fixed order are uniformly bounded on $S$.
\end{lemma}

\begin{proof}
Let $V = \left\{ h\in\mathbb R^{d+1} \setcond \sum_j h_j = 0\right\}$ be the tangent space of $\Delta^\circ$. Every finite-order directional derivative of $\ell_{x,a}$ is $\ell_{x,a}$ times a polynomial of $t$ with bounded degree, and its coefficients are uniformly bounded as a function of $x$ on $U$. For example, for any direction $h\in V$ and $t\in[0,\infty)^{d+1}$,
\[
D\ell_{x,a}[h](t) = \ell_{x,a}(t)\sum_j h_j \left( \frac{1}{x_j+a_j} - \frac{t_j}{2} \right).
\]
Thus, by equation~\eqref{eq: explicit formula for ell x a} and the choice of $U$, there are $C_2,C_3>0$ with
\[
\abs{ D\ell_{x,a}[h](t)}^2d\Escr_a(t)
\leq C_2 \norm{h}_1^2 \left(1+\sum_j t_j\right)^2 \ell_{x,a}(t)^2d\Escr_a(t)
\leq C_3 \norm{h}_1^2 \left(1+\sum_j t_j\right)^2 e^{-\delta\sum_jt_j}dt.
\]
The right-hand side is independent of $x$ and integrable. Integration gives, for some $C_4>0$,
\[
\norm{D\ell_{x,a}[h]}_{L^2(\Escr_a)}^2
\leq C_4 \norm{h}_1^2.
\]
By dominated convergence, $D\ell_{x,a}[h]$ is continuous in $\norm{\cdot}_{L^2(\Escr_a)}$ as a function of $x$ for any $h\in V$. Since $V$ is finite-dimensional, this implies continuity of $D\ell_{x,a}: V\to L^2(\Escr_a)$ in $x$ in operator norm. Then the scalar fundamental theorem of calculus, the Cauchy-Schwarz inequality, and Tonelli's theorem imply
\begin{align*}
\frac{ \norm{\ell_{x+h,a} - \ell_{x,a} - D\ell_{x,a}[h]}_{L^2(\Escr_a)}^2 }{ \norm{h}_1^2 }
&\leq \frac{1}{\norm{h}_1^2}
\int_0^1 \norm{\left( D\ell_{x+sh,a} - D\ell_{x,a} \right)[h]}_{L^2(\Escr_a)}^2 ds \\
&\leq \sup_{0\leq s\leq 1} \norm{D\ell_{x+sh,a} - D\ell_{x,a}}_{\mathrm{op}}^2
\;\xrightarrow{\;h\to0\;}\; 0.
\end{align*}
This proves that $\ell_{x,a}$ is differentiable. A similar argument proves that it is differentiable any finite number of times. The compactness of $S$ and $L_{\mathrm{clock}}^2(\Escr_a) \cong \Hscr_a$ finish the proof.
\end{proof}

\begin{lemma} \label{lem: derivative bound for kx and sx}
There is a constant $C>0$ such that for every $x,y\in S$,
\[
\norm{k_x-k_y}_{\Hscr_a}\leq C\rho(x,y),
\qquad
\norm{s_x-s_y}_{\Hscr_a}\leq C(1+\kappa)\rho(x,y).
\]
\end{lemma}

\begin{proof}
Lemma \ref{lem: bounded derivatives of features}, combined with $\norm{x-y}_1\leq\rho(x,y)$, proves $\norm{k_x-k_y}_{\Hscr_a}\leq C\rho(x,y)$. The triangle inequality and $\rho(f_ix,f_iy)\leq\kappa\rho(x,y)$ prove the bound for $\norm{s_x-s_y}_{\Hscr_a}$.
\end{proof}

\begin{proposition} \label{prop: score search terminates}
If $E_n>0$, the score search in Step $3$ of the construction terminates after at most $O\big(\log(2+1/E_n)\big)$ updates of the initial pair.
\end{proposition}

\begin{proof}
Lemma~\ref{lem: derivative bound for kx and sx} and the reproducing property give, for every non-diagonal $(u,v)\in S^2$,
\begin{equation} \label{eq: the normalized score is bounded by the error}
\frac{|q_n(u,v)|}{\rho(u,v)}
=\frac{|\ip{r_n}{k_u-k_v}_{\Hscr_a}|}{\rho(u,v)}
\leq C\norm{r_n}_{\Hscr_a}
= CE_n.
\end{equation}
Also, by the orthogonality of $r_n$ to $V_n$ and the reproducing property,
\begin{equation*}
E_n^2
=\norm{r_n}_{\Hscr_a}^2
=\ip{r_n}{s_b}_{\Hscr_a}
=\sum_i\pi_iq_n(b,f_ib).
\end{equation*}
Suppose $E_n>0$. Then not all the pairs are diagonal, so $\sum_i\pi_i\rho(b,f_ib) > 0$. Let $p_i = \dfrac{\pi_i \rho(b,f_ib)}{\sum_j \pi_j \rho(b,f_jb)}$. Then $p_0, \ldots,p_{m-1}$ are nonnegative and $\sum_ip_i = 1$. Thus, if $(b,f_kb)$ is the initial pair maximizing $q_n/\rho$,
\begin{equation} \label{eq: initial pair is at least squared error}
\frac{E_n^2}{\sum_i\pi_i\rho(b,f_ib)}
=
\frac{\sum_i\pi_iq_n(b,f_ib)}{\sum_i\pi_i\rho(b,f_ib)}
=
\sum_{i\,:\,\rho(b,f_ib)>0} p_i \frac{q_n(b,f_ib)}{\rho(b,f_ib)}
\leq
\frac{q_n(b,f_kb)}{\rho(b,f_kb)}.
\end{equation}
Combining equations \eqref{eq: the normalized score is bounded by the error} and \eqref{eq: initial pair is at least squared error} with Lemma~\ref{lem: normalized future usefulness is increasing} proves the statement.
\end{proof}

\section{Polynomial complexity} \label{sec: polynomial complexity}

\subsection{Overview} \label{subsec: polynomial complexity overview}

In this section we will prove that the error $E_n$ of the $n$-constellation $\eta_n$ is exponentially small in $n^{1/d}$. After that we will count the number of arithmetic operations needed, completing the proof of Theorem~\ref{thm: polynomial computation}. Let us briefly explain the proof strategy.

For each accepted pair $(u_n,v_n)$, we define $\phi_n\in\Hscr_a$ as
\[
\phi_n \coloneqq \frac{s_{u_n}-s_{v_n}}{\rho(u_n,v_n)}.
\]
Then $V_n=\mathrm{Span}\{\phi_1,\ldots,\phi_n\}$. By Gram-Schmidt orthogonalization,
\begin{equation}  \label{eq: Gram-Schmidt orthogonalization}
\det\left( \big(\ip{\phi_i}{\phi_j}_{\Hscr_a}\big)_{1\leq i,j\leq n} \right)
\;=\;
\prod_{j=1}^n\dist(\phi_j,V_{j-1})^2
\end{equation}
Lemma \ref{lem: error is bounded by the increment of new pair} proves that the determinant is bounded from below by a power of the error $E_n$. This is because a large remaining error forces the newly accepted direction $\phi_n$ to contribute substantial new ``height'', bounding the squared volume (=Gram determinant) from below.

On the other hand, Proposition~\ref{prop: kolmogorov width of the features is stretched exponentially small} and Lemma~\ref{lem: normalized differences admit nice approximation} prove that the Kolmogorov width of $\left\{ (s_x-s_y)/\rho(x,y) \setcond x,y\in S, \,x\ne y \right\}$ is exponentially small in $n^{1/d}$. Vectors near a small-dimensional space cannot keep increasing the volume (Lemma~\ref{lem: volume is bounded by the approximation space}).

These observations yield the exponential decay of $E_n$ in $n^{1/d}$, whose proof will be given in \S\S\ref{subsection: proof of universal poly complexity}. The final sandwich inequality is essentially the elegant volume argument underlying the convergence estimates for the greedy algorithm in \cite[Theorem~3.2 and Corollary~3.3]{DeVore--Petrova--Wojtaszczyk}. However, neither direction of the bounds follows from their argument. We note in particular that their greedy algorithm assumes that one can choose the next point with a substantial new height; we must independently establish this for our score search.

\subsection{The volume bounds the error}

\begin{lemma} \label{lem: error is bounded by the increment of new pair}
There is $c>0$ such that for every integer $n\geq0$, if $E_n>0$,
\begin{equation*}
\dist(\phi_{n+1},V_n)\geq cE_n
\end{equation*}
\end{lemma}

\begin{proof}
By the orthogonality of $r_n$ to $V_n$ and the reproducing property,
\begin{equation*}
E_n^2
=\norm{r_n}_{\Hscr_a}^2
=\ip{r_n}{s_b}_{\Hscr_a}
=\sum_i\pi_iq_n(b,f_ib).
\end{equation*}
Since $E_n>0$, by equation \eqref{eq: initial pair is at least squared error} and Lemma~\ref{lem: normalized future usefulness is increasing},
\[
\frac{E_n^2}{\sum_i\pi_i\rho(b,f_ib)}
\leq
\frac{q_n(u_{n+1},v_{n+1})}{\rho(u_{n+1},v_{n+1})}
\]
Now $(u_{n+1},v_{n+1})$ satisfies equation \eqref{eq:acceptance}. Take any $v\in V_n$. By $r_n\perp V_n$ and Lemma~\ref{lem: basic inequalities for rho and theta},
\[
\abs{\ip{r_n}{\phi_{n+1}-v}_{\Hscr_a}}
=\frac{\abs{D_n(u_{n+1},v_{n+1})}}{\rho(u_{n+1},v_{n+1})}
\geq \frac{1-\kappa}{2}\frac{q_n(u_{n+1},v_{n+1})}{\rho(u_{n+1},v_{n+1})}
\geq \frac{1-\kappa}{2}\cdot\frac{E_n^2}{\sum_i\pi_i\rho(b,f_ib)}.
\]
On the other hand, the Cauchy--Schwarz inequality implies $\abs{\ip{r_n}{\phi_{n+1}-v}_{\Hscr_a}}\leq E_n\norm{\phi_{n+1}-v}_{\Hscr_a}$, which proves the assertion.
\end{proof}

\medskip

\subsection{Exponentially small Kolmogorov width imposes small volume}

For a natural number $n$ and a bounded subset $U\subset \Hscr_a$, the \emph{Kolmogorov $n$-width} of $U$ is
\[
d_n^{\mathrm{Kol}}(U;\Hscr_a) \;\coloneqq\; \inf_{\substack{V\subset\Hscr_a \\ \dim V\leq n}} \,\sup_{u\in U} \,\,\dist_{\Hscr_a}(u,V).
\]

Let $\Mscr \coloneqq \left\{ k_x \setcond x\in S \right\} \subset \Hscr_a$. The following statement proves, in particular, that $d_n^{\mathrm{Kol}}\big(\Mscr;\Hscr_a\big) \leq Ce^{-cn^{1/d}}$ for some $c,C>0$.

\begin{proposition} \label{prop: kolmogorov width of the features is stretched exponentially small}
For every natural number $N$, there is a subspace in $\Hscr_a$ with only $O(N^d)$ dimensions such that it approximates every $k_x\,$ for $x\in S$ exponentially accurately.

More precisely, there is a polynomial map $P_N: S\to \Hscr_a$ of degree at most $N$ such that
\begin{equation} \label{eq:polynomial-clock-approximation}
\sup_{x\in S}\norm{k_x-c_a(x)P_N(x)}_{\Hscr_a}\leq Ce^{-cN}.
\end{equation}
Here, $c_a:S\to\mathbb R$ is defined by $c_a(x) \coloneqq \displaystyle \prod_{j=1}^{d+1}\dfrac{x_j+a_j}{2a_j}$, and $C,c>0$ are independent of $N$. The space $\Span\left\{ c_a(x)P_N(x) \setcond x\in S\right\}\subset\Hscr_a$ has dimension at most $\binom{N+d}{d}$.
\end{proposition}

\begin{proof}
Define $B_a, \vartheta_x\in L^2(\Escr_a)$ by
\[
\vartheta_x(t)=\frac12\sum_{j=1}^{d+1}(1-x_j)t_j,
\qquad
B_a(t)=e^{-\vartheta_a(t)}.
\]
Then $\ell_{x,a}(t)=c_a(x)B_a(t)e^{\vartheta_x(t)}$. Recall that $J: \Hscr_a \xrightarrow{\cong} L_{\mathrm{clock}}^2(\Escr_a)$ is the isometry, and write $\Pi: L^2(\Escr_a) \to L_{\mathrm{clock}}^2(\Escr_a)$ to denote the orthogonal projection. For each $x\in S$, we define
\[
Q_N(x): [0,\infty)^{d+1} \longrightarrow \mathbb R,
\qquad
Q_N(x)(t) = B_a(t) \sum_{k=0}^N\frac{\vartheta_x(t)^k}{k!},
\]
Then $Q_N(x)\in L^2(\Escr_a)$. Indeed, $Q_N$ is a polynomial in $x$ of degree at most $N$, and each of its coefficients (as a polynomial in $x$) is $B_a(t)$ times a polynomial in $t$. Since
\[
B_a(t)^2d\Escr_a(t) = \left(\prod_{j=1}^{d+1} a_j\right) e^{-\sum_jt_j}dt,
\]
we see that $Q_N(x)\in L^2(\Escr_a)$. Finally, for each $x\in S$, we define $P_N(x)\in\Hscr_a$ by
\[
P_N(x) = J^{-1} \big( \Pi Q_N(x) \big).
\]
Since $\Pi\ell_{x,a} = \ell_{x,a}$ and $\Pi$ does not increase norms,
\begin{equation} \label{eq: approximation error to kx}
\norm{k_x-c_a(x)P_N(x)}_{\Hscr_a}
= \norm{\ell_{x,a}-c_a(x)\Pi Q_N(x)}_{L^2(\Escr_a)}
\leq c_a(x) \norm{\left(e^{\vartheta_x} - \sum_{k=0}^N \frac{\vartheta_x^k}{k!} \right)\cdot B_a }_{L^2(\Escr_a)}.
\end{equation}
By the definition of $\delta$, we have $x_j\geq\delta$ for every $x\in S$ and $j$. Then, Taylor's theorem implies
\begin{equation} \label{eq: Taylor tail}
0\leq e^{\vartheta_x}-\sum_{k=0}^N\frac{\vartheta_x^k}{k!}
\leq \frac{e^{(1+\delta)\vartheta_x}}{(1+\delta)^{N+1}}.
\end{equation}
Direct integration gives
\[
\norm{\left(\frac{e^{(1+\delta)\vartheta_x}}{(1+\delta)^{N+1}} \right)\cdot B_a }_{L^2(\Escr_a)}^2
= \frac1{(1+\delta)^{2(N+1)}}\prod_{j=1}^{d+1} \frac{a_j}{1-(1+\delta)(1-x_j)}
\leq \frac{\prod_j (a_j/\delta^2)}{(1+\delta)^{2(N+1)}} < \infty.
\]
Equations \eqref{eq: approximation error to kx} and \eqref{eq: Taylor tail} then imply that $\norm{k_x-c_a(x)P_N(x)}_{\Hscr_a}$ is exponentially small in $N$, uniformly for all $x\in S$.
\end{proof}

The approximations of $k_x$ can be upgraded to those of normalized differences $\dfrac{k_x-k_y}{\rho(x,y)}$.

\begin{lemma} \label{lem: normalized differences admit nice approximation}
Let $K\subset S$. Suppose a $q$-dimensional subspace $\Uscr$ of $\Hscr_a$ approximates every $k_x$, $x\in K$, within $0<\gamma\leq1/2$. Then a subspace of dimension at most $(d+1)q$ approximates every
\[
\frac{k_x-k_y}{\rho(x,y)},\qquad x,y\in K,\quad x\ne y,
\]
within $C\sqrt\gamma$, for some fixed $C>0$ independent of $x$, $y$, $K$, $q$, and $\gamma$. If $K$ is forward invariant, the same space approximates every $(s_x-s_y)/\rho(x,y)$ within $C(1+\kappa)\sqrt\gamma$. Here, forward invariant means $f_i(K)\subset K$ for all $i$.
\end{lemma}

\begin{proof}

By the isometry $J: \Hscr_a \xrightarrow{\cong} L_{\mathrm{clock}}^2(\Escr_a)$ in equation \eqref{eq: clock isometry}, we work in $L^2(\Escr_a)$. Let $u_x\in J(\Uscr)$ approximate $\ell_{x,a}$ within $\gamma$ for each $x\in K$. Then, for all $x,y\in K$ with $x\ne y$,
\[
\norm{ \frac{\ell_{x,a}-\ell_{y,a}}{\rho(x,y)} -\frac{u_x-u_y}{\rho(x,y)} }_{L^2(\Escr_a)}
\leq \frac{2\gamma}{\rho(x,y)}.
\]
Thus, for ``far'' pairs satisfying $\rho(x,y)\geq\sqrt{\gamma}$, this proves the assertion.

Assume that $\rho(x,y)<\sqrt{\gamma}$. Since $S$ is convex, Taylor's theorem and Lemma \ref{lem: bounded derivatives of features} give some $C_1>0$ such that for any $x,y\in K$,
\begin{equation} \label{eq: derivatives approximate the normalized differences}
\norm{ \frac{\ell_{x,a}-\ell_{y,a}}{\rho(x,y)} - \frac1{\rho(x,y)}\sum_{j=1}^d(x_j-y_j)\frac{\partial}{\partial x_j}\ell_{x,a} }_{L^2(\Escr_a)}
\leq C_1\frac{\norm{x-y}_1^2}{\rho(x,y)}
\leq C_1\rho(x,y)
\leq C_1\sqrt{\gamma}.
\end{equation}
Here, we used $\norm{x-y}_1\leq\rho(x,y)$ from Lemma~\ref{lem: basic inequalities for rho and theta}. Therefore we need to approximate the derivatives $\frac{\partial}{\partial x_j}\ell_{x,a}$. Direct differentiation gives
\[
\frac{\partial}{\partial x_j}\ell_{x,a}
=
\left( \frac1{x_j+a_j} -\frac1{x_{d+1}+a_{d+1}} \right)\ell_{x,a}
-\frac12(t_j-t_{d+1})\ell_{x,a}.
\]
The terms inside the first parentheses are uniformly bounded since $x\in S$. The term $t_j-t_{d+1}$ is unbounded, so we truncate it: let $B=1/\sqrt{\gamma}$ and
\[
M_j(t)=\max\{-B,\min\{t_j-t_{d+1},B\}\}.
\]
Then $M_j$ cuts $t_j-t_{d+1}$ so that it stays within $[-B,B]$. Let us see that $M_ju_x$ approximates $(t_j-t_{d+1})\ell_{x,a}$ well enough. We have
\[
\norm{(t_j-t_{d+1})\ell_{x,a}-M_ju_x}_{L^2(\Escr_a)}
\leq
\norm{(t_j-t_{d+1}-M_j)\ell_{x,a}}_{L^2(\Escr_a)} + \norm{M_j(\ell_{x,a}-u_x)}_{L^2(\Escr_a)}.
\]
Since $\abs{t_j-t_{d+1}-M_j}\leq (t_j-t_{d+1})^2/B$, the first term is bounded by $C_2/B = C_2\sqrt{\gamma}$ with some $C_2>0$. The second term is bounded by $B\gamma = \sqrt{\gamma}$ by definitions. Thus,
\begin{equation} \label{eq: derivatives are approximated by multiplying truncation}
\norm{
\frac{\partial}{\partial x_j}\ell_{x,a}
-
\left\{ \left( \frac1{x_j+a_j} -\frac1{x_{d+1}+a_{d+1}} \right)u_x - \frac12M_ju_x \right\}
}_{L^2(\Escr_a)}
\leq
C_3\sqrt{\gamma}
\end{equation}
for some $C_3>0$. By combining equations \eqref{eq: derivatives approximate the normalized differences} and \eqref{eq: derivatives are approximated by multiplying truncation}, we see that the space
\[
W \coloneqq J(\Uscr) + \sum_{j=1}^d M_j J(\Uscr) \subset L^2(\Escr_a)
\]
approximates $(\ell_{x,a}-\ell_{y,a})/\rho(x,y)$ within $C\sqrt{\gamma}$ for some $C>0$. Each space $M_jJ(\Uscr)$ has dimension at most $q$, so the dimension of $W$ is at most $q+dq$. Finally, since orthogonal projection does not increase the norm, and since $(\ell_{x,a}-\ell_{y,a})/\rho(x,y)$ is already in $L^2_{\mathrm{clock}}(\Escr_a)$, the space $\Pi W$ preserves the approximation, and thus $J^{-1}(\Pi W)$ is the approximation space we wanted to obtain.

Next suppose $K$ is forward invariant. Take $w$ and $w_i\in J^{-1}(\Pi W)$ so that they approximate $(k_x-k_y)/\rho(x,y)$ and $(k_{f_ix}-k_{f_iy})/\rho(f_ix,f_iy)$ within $C\sqrt{\gamma}$, respectively; when $f_ix = f_iy$, we set $w_i=0$. Then, using Lemma~\ref{lem: basic inequalities for rho and theta},
\[
\norm{
\frac{s_x-s_y}{\rho(x,y)}
-
\left( w - \sum_i \pi_i \frac{\rho(f_ix,f_iy)}{\rho(x,y)} w_i \right)
}_{\Hscr_a}
\leq
C\sqrt{\gamma} \left( 1 + \sum_i \pi_i \frac{\rho(f_ix,f_iy)}{\rho(x,y)} \right)
\leq
C(1+\kappa)\sqrt\gamma. \qedhere
\]

\end{proof}

The following lemma says that $2n$ vectors close to an $n$-dimensional subspace must form a thin $2n$-dimensional parallelepiped. It is a special case of \cite[Lemma~2.1]{DeVore--Petrova--Wojtaszczyk}; for completeness, we include the proof.

\begin{lemma} \label{lem: volume is bounded by the approximation space}
Let $(\Hscr, \ip{\cdot}{\cdot}_{\Hscr})$ be a real Hilbert space and $\{v_1,\ldots,v_{2n}\}\subset\Hscr$, for some integer $n\geq1$. Suppose that a subspace $W\subset\Hscr$ satisfies
\begin{equation} \label{eq: condition on finite dimensional approximation to Gram}
\dim W = n,
\qquad
\dist(v_j, W)\leq\eta \quad (1\leq j\leq 2n)
\end{equation}
for some $\eta>0$. Then,
\[
\det\left( \big(\ip{v_i}{v_j}_{\Hscr}\big)_{1\leq i,j\leq 2n} \right)
\leq
\left(2\eta\max_{1\leq j\leq 2n}\norm{v_j}_{\Hscr}\right)^{2n}.
\]
\end{lemma}

\begin{proof}
Let $M = \max_{1\leq j\leq 2n}\norm{v_j}_{\Hscr}$ and $U = \Span\{v_1,\ldots,v_{2n}\}$. We can assume that the Gram determinant is nonzero; then $\dim U = 2n$ and $M>0$. Denote by $\widetilde{W}$ the projection of $W$ onto $U$, and enlarge $\widetilde W$ if necessary until $\dim\widetilde W = n$. Then equation \eqref{eq: condition on finite dimensional approximation to Gram} still holds for $\widetilde W$. Let $P:\Hscr\to \widetilde W$ be the orthogonal projection, and $I:U\to U$ the identity map. Define
\[
T:U\longrightarrow U, \qquad T = \frac{\eta}M P + (I-P).
\]
Then for every $1\leq j\leq 2n$,
\[
\norm{Tv_j}_{\Hscr}^2 = \left(\frac{\eta}M\right)^2 \norm{Pv_j}_{\Hscr}^2 + \norm{(I-P)v_j}_{\Hscr}^2
\leq 2\eta^2
\]
Also, since $T$ scales precisely $n$ coordinate directions by $\eta/M$, we have $\det T = (\eta/M)^n$. Since the Gram matrix is positive semidefinite, by Hadamard's inequality,
\[
\left(\frac\eta{M}\right)^{2n}
\det\left( \big(\ip{v_i}{v_j}_{\Hscr}\big)_{1\leq i,j\leq 2n} \right)
=
\det\left( \big(\ip{Tv_i}{Tv_j}_{\Hscr}\big)_{1\leq i,j\leq 2n} \right)
\leq
\prod_{j=1}^{2n} \norm{Tv_j}_{\Hscr}^2
\leq (2\eta^2)^{2n}. \qedhere
\]

\end{proof}

\subsection{From approximation space to the error bound}

The preceding lemmas give the following consequence, which we will reuse in Section~\ref{sec: constellation dimension and minimality}.

\begin{proposition} \label{prop: approximation controls the defect}
Let $K\subset S$ be forward invariant and contain the anchor $b$, and let $q\geq1$ be an integer. If a subspace of $\Hscr_a$ of dimension at most $q$ approximates every $k_x$, $x\in K$, within $0<\gamma\leq1/2$, then
\begin{equation}\label{eq:width-to-error}
E_{2(d+1)q}\leq C\gamma^{1/4},
\end{equation}
where $C>0$ is independent of $K,q,\gamma$.
\end{proposition}

\begin{proof}
Let $p=(d+1)q$. If $E_{2p}=0$, the assertion follows, so assume otherwise. Lemma~\ref{lem: normalized differences admit nice approximation} gives a subspace $W$ of dimension at most $p$ approximating all $\phi_1,\ldots,\phi_{2p}$ within $C\sqrt\gamma$. Since the error $E_n = \norm{r_n}_{\Hscr_a}$ is the distance from the origin to $s_b + V_n$, and $V_n$ is increasing, $E_n$ is non-increasing. Then Lemma~\ref{lem: error is bounded by the increment of new pair} gives $\dist(\phi_j,V_{j-1})\geq cE_{2p}$ for $j\leq2p$. By equation \eqref{eq: Gram-Schmidt orthogonalization},
\[
(cE_{2p})^{4p}
\leq
\det\big(\ip{\phi_i}{\phi_j}_{\Hscr_a}\big)_{i,j=1}^{2p}.
\]
For the other direction, by applying Lemma~\ref{lem: volume is bounded by the approximation space} to the vectors $\phi_1, \ldots, \phi_{2p}$ and $W$ enlarged to dimension $p$ if necessary, we obtain the following, proving the assertion:
\[
\det\big(\ip{\phi_i}{\phi_j}_{\Hscr_a}\big)_{i,j=1}^{2p}
\leq\left(C_0\sqrt\gamma\right)^{2p}. \qquad (C_0>0)\qedhere
\]
\end{proof}

\subsection{Proof of Theorem~\ref{thm: polynomial computation}} \label{subsection: proof of universal poly complexity}

\begin{proof}

For each $N\geq1$, applying Proposition~\ref{prop: kolmogorov width of the features is stretched exponentially small} yields a subspace in $\Hscr_a$ of dimension at most $q=\binom{N+d}{d}$ approximating every $k_x\,$ for $x\in S$ within $\gamma=C_0e^{-c_0N}$ for some fixed $C_0,c_0>0$. For $n\geq1$, we choose $N$ as
\[
N=\left\lfloor\left(\frac{n}{2(d+1)}\right)^{1/d}\right\rfloor-d.
\]
Sufficiently large $n$ makes $N\geq1$ and $C_0e^{-c_0N}\leq 1/2$. The choice of $N$ ensures that $2(d+1)q\leq n$. Proposition~\ref{prop: approximation controls the defect} then implies, by combining with the monotonicity of $E_n$,
\begin{equation*}
E_n \leq E_{2(d+1)q}\leq C_1e^{-\frac{c_0}{4}N}.
\end{equation*}
After adjusting the constant this implies exponential smallness of $E_n$ in $n^{1/d}$, for all $n\geq 1$.

Next we count the arithmetic operations needed to achieve $E_n\leq \eps$. While we search for the next pair $(u_{n+1},v_{n+1})$, Proposition~\ref{prop: score search terminates} bounds the search depth by $O(\log(1/\eps))$ since $E_n>\eps$. Since the score has $n+1$ terms, the search for $(u_{n+1},v_{n+1})$ costs $O((n+1)\log(1/\eps))$. Since $E_n\leq \eps$ for $n$ at most $O\left(\big(\log(1/\eps)\big)^d\right)$, the total cost of all searches is $O\left(\big(\log(1/\eps)\big)^{2d+1}\right)$.

At each $n$, the stationary equations~\eqref{eq: stationary equations} have the form $G_n\alpha = c$ for some $\alpha,c\in\mathbb R^n$ and $n\times n$ real matrix $G_n$. We perform the factorization $G_n = L_n D_n L_n^\top$, where $L_n$ is lower triangular and $D_n$ is diagonal. Since $G_{n+1}$ is obtained by appending a row and a column to $G_n$, updating this factorization and solving only costs $O((n+1)^2)$. (See \cite[Algorithm 4.1.1]{Golub--VanLoan} for example.) Since $E_n\leq \eps$ for $n$ at most $O\left(\big(\log(1/\eps)\big)^d\right)$, the total cost of solving is $O\left(\big(\log(1/\eps)\big)^{3d}\right)$.
\end{proof}

\section{Constellation dimension} \label{sec: constellation dimension}

In this section we establish the well-definedness of $\Cscr_a$ and prove Proposition~\ref{prop: constellation dimension universal bound}. We are given positive matrices $A_0,\ldots,A_{m-1}$. Let $\nu$ be the unique stationary measure. Then $\Kcal = \supp\nu$.

\subsection{Constellation dimension is independent of the reference point}

Let us first justify the definitions. Fix $a\in\Delta^\circ$. Recall that $L^2(\nu)$ is the space of $\nu$-square integrable real functions, and $T_a:\Hscr_a \to L^2(\nu)$ is defined by $T_a(u) = u - \Lcal(u)$. By the reproducing property,
\[
\norm{T_a u}_{L^2(\nu)}^2
=\int_\Kcal\abs{\ip{u}{s_x}_{\Hscr_a}}^2\,d\nu(x)
\leq
\norm{u}_{\Hscr_a}^2
\int_\Kcal\norm{s_x}_{\Hscr_a}^2\,d\nu(x).
\]
Since $x\mapsto s_x$ is a continuous function on the compact set $\Kcal$, the last expression is bounded, and we see that $T_a$ is a bounded operator. Therefore we can define the bounded operator $\Cscr_a:\Hscr_a\to \Hscr_a$ by $\Cscr_a = T_a^* T_a^{}$. It is positive by definition; here, positive means $\ip{u}{\Cscr_au}_{\Hscr_a}\geq0$ for every $u\in\Hscr_a$.

Notice that
\begin{equation}\label{eq:covariance-integral}
\Cscr_a u=\int_\Kcal\ip{u}{s_x}_{\Hscr_a}s_x\,d\nu(x).
\end{equation}
This identity can be confirmed by taking an inner product with an arbitrary element of $\Hscr_a$. Thus $\Cscr_a$ is a covariance operator, since $\int s_xd\nu(x) = 0$.

\begin{lemma} \label{lem: the covariance operator is positive and trace class}
The operator $\Cscr_a$ is positive and trace class for any $a\in\Delta^\circ$.
\end{lemma}
\begin{proof}
We have already proved positivity. Since $L^2(\Escr_a)$ is a separable space, by the isometry in equation \eqref{eq: clock isometry}, $\Hscr_a$ is also separable. Then it has a countable orthonormal basis $\{e_j\}_{j=1}^\infty$. By Tonelli's theorem and Parseval's identity,
\[
\mathrm{tr} \,\Cscr_a
=\sum_{j=1}^\infty \ip{e_j}{\Cscr_a e_j}_{\Hscr_a}
=\int_\Kcal\sum_j|\langle e_j,s_x\rangle_{\mathscr H_a}|^2\,d\nu(x)
=\int_\Kcal \norm{s_x}_{\Hscr_a}^2\,d\nu(x)
<\infty. \qedhere
\]

\end{proof}

\medskip

Let us see how the choice of the reference point $a$ is immaterial to the definition of constellation dimension. We need several lemmas.

\begin{lemma} \label{lem: signed measure comparison}
For any $a,a'\in\Delta^\circ$, there is $C\geq1$ such that every signed measure $\sigma$ finitely supported on $\Delta^\circ$ satisfies
\begin{equation*}
C^{-1}\norm{\sigma}_{K_{a'}}^2 \leq \norm{\sigma}_{K_a}^2 \leq C\norm{\sigma}_{K_{a'}}^2.
\end{equation*}
\end{lemma}

\begin{proof}
Take any $a,a',x,y\in\Delta^\circ$. Notice that we can rewrite $K_a$ as
\[
K_a(x,y)= \frac{1}{2^{d+1}}\prod_{k=1}^{d+1}\left(
a_k\int_0^\infty e^{-x_kt_k}e^{-y_kt_k}\,dt_k
\;+\;\frac1{a_k}\int_0^\infty (x_ke^{-x_kt_k})(y_ke^{-y_kt_k})\,dt_k
\;+\;1
\right).
\]
Expanding the product on the right-hand side, the terms including $x$ or $y$ have the same expression in $x$ and $y$. That is, they have the form
\[
w(a)\int F(x,t)F(y,t)\,dt,
\]
where $w(a)>0$ is the product of the selected coefficients, $F(x,t)$ is the product of the selected factors $e^{-x_kt_k}$, $1$, or $x_ke^{-x_kt_k}$, and $t$ collects the selected integration variables.

Take any finitely-supported signed measure $\sigma = \sum_i c_i \delta_{x_i}$ with $c_i\in\mathbb R$ and $x_i\in\Delta^\circ$. We have
\[
w(a)\sum_{i,j}c_ic_j\int F(x_i,t)F(x_j,t)\,dt
=w(a)\Int{\left(\sum_i c_iF(x_i,t)\right)^2dt\geq0}.
\]
Therefore, the quadratic form $\norm{\sigma}_{K_a}^2 = \sum_{i,j}c_ic_jK_a(x_i,x_j)$ is a finite sum of nonnegative quantities independent of $a$, with positive coefficients depending on $a$. Comparing these coefficients for $a$ and $a'$ gives a constant $C\geq1$ such that
\[
C^{-1}\sum_{i,j}c_ic_jK_{a'}(x_i,x_j)
\leq\sum_{i,j}c_ic_jK_a(x_i,x_j)
\leq C\sum_{i,j}c_ic_jK_{a'}(x_i,x_j). \qedhere
\]
\end{proof}

Next, we characterize membership in $\Hscr_a$ through signed measures.
\begin{lemma} \label{lem: characterization of membership in the Hilbert space}
A function $u:\Delta^\circ\to\mathbb R$ belongs to $\Hscr_a$ if and only if there is $M\geq0$ such that
\[
\abs{\sigma(u)}\leq M\norm{\sigma}_{K_a}
\]
for every finitely-supported signed measure $\sigma$ on $\Delta^\circ$. In this case, we have $\norm{u}_{\Hscr_a}\leq M$.
\end{lemma}

\begin{proof}
If $u\in\Hscr_a$, the reproducing property gives, for any $\sigma=\sum_i c_i\delta_{x_i}$,
\[
\sigma(u)=\ip{u}{\;\sum_i c_i k_{x_i}}_{\Hscr_a},
\qquad
\norm{\sigma}_{K_a}=\norm{\sum_i c_i k_{x_i}}_{\Hscr_a}.
\]
Then applying Cauchy--Schwarz proves the bound with $M=\norm{u}_{\Hscr_a}$.

Conversely, the bound makes the assignment $\sum_i c_i k_{x_i}\longmapsto\sum_i c_i u(x_i)$ well-defined: if $\sum_ic_ik_{x_i}~=~\sum_j d_jk_{y_j}$ represents the same vector, then
\[
\abs{\sum_i c_i u(x_i)-\sum_j d_j u(y_j)}
\leq
M\norm{\sum_i c_i k_{x_i}-\sum_j d_j k_{y_j}}_{\Hscr_a}
=0.
\]
Furthermore, this assignment is bounded by $M$: we have $\abs{\sum_i c_i u(x_i)} = \abs{\sigma(u)} \leq M \norm{\sum_i c_i k_{x_i}}_{\Hscr_a}$. By the Bounded Linear Transformation Theorem, this assignment extends uniquely to an $M$-bounded linear functional $\mathscr I$ on $\Hscr_a$. Let $v_u\in\Hscr_a$ be its Riesz representation. Then for any $x\in\Delta^\circ$,
\[
v_u(x) = \ip{v_u}{k_x}_{\Hscr_a} = \mathscr I(k_x) = u(x).
\]
Thus $u=v_u\in\Hscr_a$, proving the characterization. The assertion $\norm{u}_{\Hscr_a}\leq M$ follows from the following:
\[
\norm{u}_{\Hscr_a}^2 = \norm{v_u}_{\Hscr_a}^2 = \mathscr I(v_u) \leq M\norm{v_u}_{\Hscr_a} = M\norm{u}_{\Hscr_a}. \qedhere
\]
\end{proof}

\begin{lemma}\label{lem: choice of the reference is immaterial}
For $a,a'\in\Delta^\circ$, we have $\Hscr_a=\Hscr_{a'}$ as sets of functions on $\Delta^\circ$. Also, there is $C\geq1$ such that every $u\in\Hscr_a$ satisfies
\begin{equation*}
C^{-1}\norm{u}_{\Hscr_{a'}}^2 \leq \norm{u}_{\Hscr_a}^2 \leq C\norm{u}_{\Hscr_{a'}}^2.
\end{equation*}
\end{lemma}

\begin{proof}
Lemma~\ref{lem: signed measure comparison} shows that $\sigma\mapsto\sigma(u)$ is bounded for $\norm{\cdot}_{K_a}$ if and only if it is bounded for $\norm{\cdot}_{K_{a'}}$. Lemma~\ref{lem: characterization of membership in the Hilbert space} therefore gives $\Hscr_a=\Hscr_{a'}$ as sets of functions. Also, by the Cauchy-Schwarz inequality and Lemma~\ref{lem: signed measure comparison}, there is $C\geq1$ with
\[
\abs{\sigma(u)}
\leq \norm{u}_{\Hscr_a}\norm{\sigma}_{K_a}
\leq \sqrt{C}\norm{u}_{\Hscr_a}\norm{\sigma}_{K_{a'}}
\]
Applying Lemma~\ref{lem: characterization of membership in the Hilbert space} to $a'$ with $M=\sqrt{C}\norm{u}_{\Hscr_a}$, we obtain $\norm{u}_{\Hscr_{a'}}\leq \sqrt{C} \norm{u}_{\Hscr_a}$. The reverse direction follows as well.
\end{proof}

\begin{corollary}\label{cor: reference independence of constellation dimension}
For $a,a'\in\Delta^\circ$, there is $C\geq1$ such that
\begin{equation}\label{eq: comparable spectrum for different references}
C^{-1}\sigma_j(\Cscr_{a'})\leq\sigma_j(\Cscr_a)
\leq C\sigma_j(\Cscr_{a'})\qquad(j\geq1).
\end{equation}
Consequently, the lower and upper constellation dimensions are independent of the choice of $a$.
\end{corollary}

To prove this, we apply the following Min--Max principle for self-adjoint operators.
\begin{theorem}[Min--Max principle]{\cite[Problem~37(4)]{Brezis}} \label{thm: min--max principle}
Let $H$ be a Hilbert space, and let $A:H\to H$ be a positive compact self-adjoint operator. Let $\lambda_1(A)\geq\lambda_2(A)\geq\cdots\geq0$ be its positive eigenvalues in nonincreasing order, repeated according to multiplicity, and append zeros if there are only finitely many. Then, for every integer $j\geq1$ with $j\leq\dim H$,
\[
\lambda_j(A)
=\max_{\substack{V\subset H\\ \dim V=j}}
\min_{u\in V\setminus\{0\}}
\frac{\langle u,Au\rangle_H}{\norm{u}_H^2}.
\]
%where the maximum is taken over all $j$-dimensional linear subspaces of $H$.
\end{theorem}

\begin{proof}[Proof of Corollary~\ref{cor: reference independence of constellation dimension}]
For every nonzero $u\in\Hscr_a=\Hscr_{a'}$, the \emph{Rayleigh quotient} is
\[
\frac{\ip{u}{\Cscr_a u}_{\Hscr_a}}{\norm u_{\Hscr_a}^2}
=\frac{\norm{u-\Lcal u}_{L^2(\nu)}^2}{\norm u_{\Hscr_a}^2}.
\]
The numerator is independent of $a$, and Lemma~\ref{lem: choice of the reference is immaterial} compares the denominators. Thus Theorem~\ref{thm: min--max principle} gives equation~\eqref{eq: comparable spectrum for different references}. For $t\geq\log C$, it follows that
\[
\Ncal_{a'}(t-\log C) \;\leq\; \Ncal_a(t) \;\leq\; \Ncal_{a'}(t+\log C).
\]
We see that both limits in equation~\eqref{eq: constellation dimensions} are unchanged by changing $a$.
\end{proof}

\subsection{Algebraic dimension bounds the constellation dimension}

Corollary~\ref{cor: reference independence of constellation dimension} established that the choice of $a$ is immaterial. We will fix $a$ from now on. Let us recall the definition of the Zariski closure.

\begin{definition} \label{def: Zariski closure}
Let $H_{\mathbb C} \coloneqq \left\{ z\in\mathbb C^{d+1} \,:\, \sum_j z_j = 1\right\} \supset \Delta$ be the complex affine hyperplane, and $\mathbb C\left[H_{\mathbb C}\right]$ be its polynomial ring. For $K\subset\Delta$, define the ideal
\[
I(K) \coloneqq \left\{ P\in \mathbb C\left[H_{\mathbb C}\right] \;\setcond\; P(x) = 0 \text{ for every } x\in K \right\}.
\]
Then the \emph{Zariski closure} of $K$ is defined as
\[
\overline{K}^{\mathrm{Zar}} \coloneqq \left\{ z\in H_{\mathbb C} \setcond P(z) = 0 \text{ for every } P\in I(K)\right\}.
\]
\end{definition}

The polynomials vanishing on $\Kcal$ can reduce the dimension of the approximation space that we obtained in Section~\ref{sec: polynomial complexity}. For each integer $N\geq0$, let $\mathcal P_N$ be the set of real polynomials in $d$ variables of total degree at most $N$. That is,
\[
\mathcal P_N
\;\coloneqq\;
\Span_{\mathbb R}
\left\{
X_1^{\alpha_1}\cdots X_d^{\alpha_d} \setcond \alpha_j\in\mathbb Z_{\geq0}, \;\alpha_1+\cdots+\alpha_d\leq N
\right\}
\subset\mathbb R[X_1,\ldots,X_d].
\]
We use the affine coordinates $x_1,\ldots,x_d$ on $\Delta$. That is, $P\in \mathcal P_N$ defines
\[
P|_\Delta: \Delta\longrightarrow\mathbb R,
\qquad (x_1,\ldots,x_{d+1})\longmapsto P(x_1,\ldots,x_d).
\]
We write $\mathcal P_N(K)\coloneqq\left\{ P|_K \,:\, P\in\mathcal P_N \right\}$ to denote the restrictions to $K\subset\Delta$.

\begin{lemma}\label{lem: polynomial dimension growth}
Let $K\subset\Delta$ be nonempty and $r=\dim_\mathbb C\overline K^{\mathrm{Zar}}$. There are constants $C,c>0$ such that, for every integer $N\geq0$,
\begin{equation} \label{eq: polynomial dimension growth}
c(1+N)^r \;\leq\; \dim\mathcal P_N(K) \;\leq\; C(1+N)^r.
\end{equation}
\end{lemma}

This is immediate from the following affine Hilbert polynomial theorem.

\begin{theorem}[affine Hilbert polynomial theorem]{\cite[Corollary~10.80]{Knapp}} \label{thm: affine Hilbert polynomial}
Let $k$ be an algebraically closed field and let $I\subsetneq k[x_1,\ldots,x_d]$ be an ideal. Then
\[
\dim_k V(I)=\deg_s H_{\mathrm a}(s,I).
\]
\end{theorem}
Here $\dim_k V(I)$ is the Krull dimension of $V(I)$, and $\deg_s$ is the degree of a polynomial in $s$. The polynomial $H_{\mathrm a}(s,I)\in\mathbb Q[s]$ is the \emph{affine Hilbert polynomial} of $I$, characterized by
\[
H_{\mathrm a}(N,I)=\dim^{(k)}\left(\frac{k[x_1,\ldots,x_d]_{\leq N}}{I\cap k[x_1,\ldots,x_d]_{\leq N}}\right)
\]
for all sufficiently large integers $N$ (\cite[Corollary~10.72]{Knapp}). The function in $N$ on the right-hand side is called the \emph{affine Hilbert function} of $I$, where $k[x_1,\ldots,x_d]_{\leq N}$ denotes the vector space of polynomials of total degree at most $N$, and $\dim^{(k)}$ denotes the dimension as a $k$-vector space.

\begin{proof}[Proof of Lemma~\ref{lem: polynomial dimension growth}]
On $H_{\mathbb C}$, the last coordinate satisfies $x_{d+1}=1-\sum_{j=1}^d x_j$. We therefore use $x_1,\ldots,x_d$ as coordinates and identify the ring of polynomial functions on $H_{\mathbb C}$ with $R~\coloneqq~\mathbb C[x_1,\ldots,x_d]$.

We apply Theorem~\ref{thm: affine Hilbert polynomial} with $k=\mathbb C$ and $I = I(K)$. Then $V(I)=\overline K^{\mathrm{Zar}}$, whose complex dimension is $r$.

We can write any $F\in R_{\leq N}$ as $F=P+iQ$ with real polynomials $P,Q$. Since $K$ consists of real points, $F$ vanishes on $K$ if and only if both $P$ and $Q$ vanish there. Moreover, two polynomials in $R_{\leq N}$ have the same restriction to $K$ precisely when their difference belongs to $I(K)$. Thus
\[
\dim_{\mathbb R}\mathcal P_N(K)
=\dim_{\mathbb C}\left( \left\{ F|_K \,:\, F\in R_{\leq N} \right\} \right)
=\dim_{\mathbb C}\left(\frac{R_{\leq N}}{I(K)\cap R_{\leq N}}\right).
\]
The right-most expression is the affine Hilbert function of $I$. Thus, by Theorem~\ref{thm: affine Hilbert polynomial}, $\dim_{\mathbb R}\mathcal P_N(K)$ agrees with a polynomial of degree $r$ for all sufficiently large $N$. We can take $C,c>0$ so that equation~\eqref{eq: polynomial dimension growth} holds for all $N\geq0$.
\end{proof}

The final ingredient for the proof of Proposition~\ref{prop: constellation dimension universal bound} is the following lemma. If $U\subset\Hscr_a$ approximates all $k_x$ for $x\in\Kcal$, then every unit vector $u\in U^\perp$ is uniformly small on $\Kcal$ by the reproducing property. Since $\Kcal$ is forward invariant, $\Lcal u(x)$ averages values of $u$ at points of $\Kcal$, so $\Lcal u$ is uniformly small there as well. Thus $u-\Lcal u$ is small on $\Kcal$.

\begin{lemma}\label{lem: approximation bounds the spectrum}
Let $q\geq0$ be an integer and $U\subset\Hscr_a$ a subspace of dimension at most $q$. Then
\begin{equation*}
\sqrt{\sigma_{q+1}} \;\leq\; 2\sup_{x\in\Kcal}\dist(k_x,U).
\end{equation*}
\end{lemma}

\begin{proof}
Let $\gamma = \sup_{x\in\Kcal}\dist(k_x,U)$, and take any $u\perp U$ with $\norm{u}_{\Hscr_a}=1$. The reproducing property and the Cauchy--Schwarz inequality give, for every $v\in U$,
\[
\abs{u(x)}
=\abs{\ip{u}{k_x}_{\Hscr_a}}
=\abs{\ip{u}{k_x-v}_{\Hscr_a}}
\leq \norm{u}_{\Hscr_a}\norm{k_x-v}_{\Hscr_a}
=\norm{k_x-v}_{\Hscr_a}.
\]
Taking the infimum over $v\in U$, we obtain $\abs{u(x)} \leq \dist(k_x,U) \leq \gamma$. Since $f_i(\Kcal)\subset\Kcal$,
\begin{equation} \label{eq: covariance bound by sup}
\sqrt{\ip{u}{\Cscr_a u}_{\Hscr_a}}
=\norm{u-\Lcal u}_{L^2(\nu)} \leq2 \gamma.
\end{equation}
If $\sigma_{q+1}=0$, the bound is immediate. If $\sigma_{q+1}>0$, the first $q+1$ eigenvectors of $\Cscr_a$ span a space containing a unit vector $u$ perpendicular to $U$, whose Rayleigh quotient is at least $\sigma_{q+1}$, which proves the claim by equation~\eqref{eq: covariance bound by sup}.
\end{proof}

\begin{proof}[Proof of Proposition~\ref{prop: constellation dimension universal bound}]

Corollary~\ref{cor: reference independence of constellation dimension} proves that constellation dimensions are independent of the choice of $a$. Take $P_N: S\to\Hscr_a$ from Proposition~\ref{prop: kolmogorov width of the features is stretched exponentially small}, and recall that $P_N$ is a polynomial on $S$ of degree at most $N$. Therefore, we can write it as
\[
P_N(x) = \sum_{\alpha\in\mathbb Z_{\geq0}^{d+1}, \;\abs{\alpha}\leq N}x^{[\alpha]} w_\alpha. \qquad (w_\alpha\in\Hscr_a)
\]
Here, $x^{[\alpha]} \coloneqq x_1^{\alpha_1}\cdots x_{d+1}^{\alpha_{d+1}}$ for $x = (x_1,\ldots,x_{d+1})\in S$ and $\alpha\in\mathbb Z_{\geq0}^{d+1}$.

Write $d_N=\dim\mathcal P_N(\Kcal)$ and choose a basis $p_1,\ldots,p_{d_N}$ of $\mathcal P_N(\Kcal)$. Since each monomial $x^{[\alpha]}$ in the expression above is an element of $\mathcal P_N(\Kcal)$, the restriction of $P_N$ to $\Kcal$ can be written as
\[
P_N|_{\Kcal}(x)=\sum_{j=1}^{d_N}p_j(x)u_j\qquad(u_j\in\Hscr_a).
\]
Consequently, every $c_a(x)P_N(x)$ for $x\in\Kcal$ belongs to $U_N \coloneqq \Span\{u_1,\ldots,u_{d_N}\}\subset\Hscr_a$, where the definition of $c_a$ is as in Proposition~\ref{prop: kolmogorov width of the features is stretched exponentially small}. Since multiplication by $c_a(x)$ only changes the scalar coefficients, by Proposition~\ref{prop: kolmogorov width of the features is stretched exponentially small}, there are $C_0,c_0>0$ with
\[
\sup_{x\in\mathcal K}\dist(k_x,U_N) \leq C_0e^{-c_0N}.
\]
Since $\dim U_N \leq d_N$, Lemma~\ref{lem: approximation bounds the spectrum} gives $\sigma_{d_N+1}\leq 4C_0^2e^{-2c_0N}$. Therefore, we can take $M>0$ large enough so that $N(t)\coloneqq \lceil M(1+t)\rceil$ satisfies $\sigma_{d_{N(t)}+1}< e^{-t}$ for all $t>0$. Then the definition of $\Ncal_a(t)$ and Lemma~\ref{lem: polynomial dimension growth} give
\[
\Ncal_a(t) \;\leq\; d_{N(t)}+1 \;=\; \dim\mathcal P_{N(t)}(\Kcal) + 1\;\leq\; C_1(1+t)^r
\]
for some fixed $C_1>0$, where $r=\dim_{\mathbb C}\overline{\Kcal}^{\mathrm{Zar}}$. This proves $\ud\leq r$.
\end{proof}

\section{Constellation dimension and minimality} \label{sec: constellation dimension and minimality}

\subsection{Overview}

In this section we will prove that the minimal model size defined in equation \eqref{eq: minimal model size} has asymptotic lower/upper exponent equal to the lower/upper constellation dimension. Since the upper bound is proved using our own construction of constellations, we also establish that our construction has the same lower/upper exponent; namely, equation~\eqref{eq: optimality of the construction}.

\textbf{Upper exponent:} This is shown by constructing a better space approximating $k_x$ for all $x\in\Kcal$. Proposition~\ref{prop: approximation controls the defect} already establishes that a small approximation space implies small errors for constellations. Let $U\subset\Hscr_a$ be a candidate approximation space. The reproducing property gives
\begin{equation} \label{eq: characterization of distance through reproducing property}
\dist(k_x,U) = \sup \left\{ \,\abs{u(x)} \;:\; u\perp U, \;\norm{u}_{\Hscr_a}=1 \right\}
\end{equation}
Thus using $U$ to approximate all vectors $k_x$ for $x\in\Kcal$ is equivalent to controlling the supremum norm on $\Kcal$ of every unit function perpendicular to $U$.

Recall that
\[
\ip{u}{\Cscr_a u}_{\Hscr_a}
= \norm{u-\Lcal u}_{L^2(\nu)}^2.
\]
Suppose $U$ contains all eigenfunctions with eigenvalues at least $e^{-t}$. Then by the spectral theorem, every unit $u$ with $u\perp U$ satisfies
\[
\norm{u-\Lcal u}_{L^2(\nu)}^2
\leq e^{-t}.
\]
The crucial inequality is the following: there are $C\geq1$ and $0<\alpha<1$ such that every $u\in\Hscr_a$ with $\nu(u)=0$ satisfies
\begin{equation}\label{eq: Holder-type interpolation inequality}
\norm u_{C(\Kcal)}\leq C\norm{u-\Lcal u}_{L^2(\nu)}^\alpha\norm u_{\Hscr_a}^{1-\alpha}.
\end{equation}
Here and henceforth,
\[
\norm f_{C(K)} \coloneqq \sup_{x\in K}|f(x)|
\]
denotes the supremum norm for a continuous function $f:K\to \mathbb C$ for some $K$. This inequality says that a function of bounded Hilbert norm with small stationary defect is uniformly close to its $\nu$-mean. We therefore add the mean section $\int k_xd\nu(x)$ to $U$, orthogonality to which implies zero $\nu$-mean. Then every unit $u$ with $u\perp U$ has small stationary defect and zero $\nu$-mean, and thus has small $\norm u_{C(\Kcal)}$. This is how the spectrum of $\Cscr_a$ constructs an approximation space and bounds the error $E_n$ of our construction.

\medskip

\textbf{Lower exponent:} This direction requires us to control every conceivable constellation. Fix $t>0$, and let $U_t$ be the span of the eigenfunctions of $\mathscr C_a$ with eigenvalues at least $e^{-t}$.  Consider a constellation $\eta$ with $\#\supp\eta = r$, and let $\eps = \norm{\eta - \Lcal^*\eta}_{K_a}$. We want to obtain a lower bound for $\eps$ in terms of $r$.

Suppose $r<\dim U_t$. Since evaluation at $r$ points imposes at most $r$ linear conditions, there is a nonzero $u\in U_t$ vanishing at every point in $\supp\eta$. In particular, we have $\eta(u^2)=0$. On the other hand, the spectrum ensures that this function does not vanish under the stationary measure. Indeed, stationarity makes $\Lcal$ non-expanding on $L^2(\nu)$, so with the definition of $U_t$,
\[
e^{-t}\norm{u}_{\Hscr_a}^{2}
\leq \norm{u-\mathcal Lu}_{L^2(\nu)}^{2}
\leq 4\norm{u}_{L^2(\nu)}^{2}
=4\nu(u^2).
\]
Therefore, $\eta$ misses the observation $u^2$ while $\nu$ sees its substance. We need to control $\nu(u^2) = (\nu - \eta)u^2$ from above with the stationary defect $\eps$ of $\eta$. If we define $h$ as the solution to the Poisson equation $u^2-\nu(u^2)=(I-\mathcal L)h$, we have
\[
\nu(u^2) = \abs{\nu(u^2)-\eta(u^2)}
\leq \abs{(\eta-\Lcal^*\eta)(h)}
\leq \eps\norm{h}_{\Hscr_a}.
\]
This is the essential strategy, but the problem is that we cannot obtain a bound on $\norm{h}_{\Hscr_a}$ with this design. So we need to replace the eigenfunctions by more suitable functions, and we will see that this is indeed possible.

%We will show that we can assign some $f_u$ to $u\in U_t$. We prove the two estimates
%\[ \nu(f_u^2)\geq c e^{-4t}\nu(u^2), \qquad \abs{\eta(f_u^2)-\nu(f_u^2)} \leq C\eps\,\nu(u^2), \]
%for some $C,c>0$. We now impose the vanishing conditions on this new family. Because $u\mapsto f_u$ is linear, $r<\dim U_t$ still allows a nonzero $u$ satisfying $f_u(x_i)=0$ at every constellation point. For this choice, the same stationary average has two incompatible bounds:
%\[ c e^{-4t}\nu(u^2) \leq \nu(f_u^2) = \abs{\eta(f_u^2)-\nu(f_u^2)} \leq C\eps\,\nu(u^2). \]

\subsection{Stationary defect bounds the supremum}

Let us begin by proving equation~\eqref{eq: Holder-type interpolation inequality}. This is an ``interpolation inequality'', because we always have $\norm{u-\mathcal Lu}_{L^2(\nu)} \leq 2\norm{u}_{C(\mathcal K)} \leq C'\norm{u}_{\mathscr H_a}$ for some $C'>0$.

\begin{lemma} \label{lem: Holder-type interpolation inequality}
There are $C\geq1$ and $0<\alpha<1$ such that every $u\in\Hscr_a$ with $\nu(u)=0$ satisfies 
\begin{equation*}
\norm u_{C(\Kcal)}\;\leq\; C\norm{u-\Lcal u}_{L^2(\nu)}^\alpha\norm u_{\Hscr_a}^{1-\alpha}.
\end{equation*}
\end{lemma}

\begin{proof}
The reproducing property, Cauchy--Schwarz inequality, and Lemma~\ref{lem: derivative bound for kx and sx} give
\begin{equation}\label{eq: Lipschitz estimate for observables}
|u(z)-u(z')|
=\abs{\ip{u}{k_z-k_{z'}}_{\Hscr_a}}
\leq \norm{u}_{\Hscr_a}\norm{k_z-k_{z'}}_{\Hscr_a}
\leq C_1\norm u_{\Hscr_a}\rho(z,z')
\end{equation}
for $z,z'\in S$ and some $C_1>0$. By the stationarity of $\nu$ and $\nu(u)=0$,
\[
\Lcal^Nu(x)
=\Lcal^Nu(x)-\int_{\Kcal}u(y)\,d\nu(y)
=\int_\Kcal\sum_{|w|=N}\pi_w
\big(u(f_w x)-u(f_w y)\big)\,d\nu(y).
\]
Equation~\eqref{eq: Lipschitz estimate for observables}, Lemma~\ref{lem: basic inequalities for rho and theta}, and compactness of $\Kcal$ give
\begin{equation} \label{eq: transfer operator is contractive}
\norm{\Lcal^Nu}_{C(\Kcal)}
\leq C_2\kappa^N\norm u_{\Hscr_a}.
\qquad (C_2>0)
\end{equation}

Fix $x\in\Kcal$. Since $\Kcal=\bigcup_{|w|=N}f_w(\Kcal)$, there is a word $w$ with $|w|=N$ such that
$x=f_w z$ for some $z\in\Kcal$. By Lemma~\ref{lem: basic inequalities for rho and theta} and equation~\eqref{eq: Lipschitz estimate for observables},
\begin{equation} \label{eq: mixing error estimate}
\abs{u(x)-\nu(u\circ f_w)}
\leq \int_\Kcal \abs{u(f_w z)-u(f_w y)}\,d\nu(y)
\leq C_3\kappa^N\norm u_{\Hscr_a}.
\qquad (C_3>0)
\end{equation}

Since $\nu=\sum_{|w_0|=N}\pi_{w_0}(f_{w_0})_*\nu$, the Cauchy--Schwarz inequality gives, for every $g\in L^2(\nu)$,
\begin{equation} \label{eq: average of long composition}
\abs{\nu(g\circ f_w)}
=\left|\int_\Kcal g(f_w y)\,d\nu(y)\right|
\leq \frac{1}{\sqrt{\pi_w}} \left( \pi_w\int_\Kcal|g(f_w y)|^2\,d\nu(y) \right)^{1/2}
\leq \frac{1}{\sqrt{\pi_w}}\norm{g}_{L^2(\nu)}.
\end{equation}

Since $\abs{\sum_i\pi_i g(f_i y)}^2 \leq\sum_i\pi_i\abs{g(f_i y)}^2$ by Jensen's inequality, integrating and using stationarity gives $\norm{\Lcal g}_{L^2(\nu)}^2\leq\norm{g}_{L^2(\nu)}^2$. Finally, notice that
\[
u=\sum_{j=0}^{N-1}\Lcal^j(u-\Lcal u)+\Lcal^Nu.
\]
We compose both sides with $f_w$ and $\nu$-integrate. Equation~\eqref{eq: average of long composition} and $\norm{\Lcal g}_{L^2(\nu)}^2\leq\norm{g}_{L^2(\nu)}^2$ give
\begin{align*}
\abs{\nu(u\circ f_w)}
&\leq \sum_{j=0}^{N-1} \frac1{\sqrt{\pi_w}} \norm{ \Lcal^j(u-\Lcal u)}_{L^2(\nu)} + \norm{\Lcal^Nu}_{C(\Kcal)} \\
&\leq \frac{N}{\sqrt{\pi_w}}\norm{u-\Lcal u}_{L^2(\nu)} +\norm{\Lcal^Nu}_{C(\Kcal)}.
\end{align*}
Since $\pi_w\geq\pi_*^N$ for $\pi_*\coloneqq \min_i\pi_i$, equations \eqref{eq: transfer operator is contractive} and \eqref{eq: mixing error estimate} imply
\begin{equation}\label{eq:word-telescoping}
\norm u_{C(\Kcal)}
\leq N\pi_*^{-\frac{N}{2}}\norm{u-\Lcal u}_{L^2(\nu)} + (C_2+C_3)\kappa^N\norm u_{\Hscr_a}.
\end{equation}

Set
\[
\alpha \coloneqq \frac{\log(1/\kappa)}{\log(2/(\kappa\sqrt{\pi_*}))}\in(0,1).
\]
If $\norm{u-\Lcal u}_{L^2(\nu)}\geq\norm u_{\Hscr_a}$, the assertion follows since $\norm u_{C(\Kcal)}$ is bounded by a constant times $\norm u_{\Hscr_a}$ by the reproducing property and the Cauchy-Schwarz inequality. If $\norm{u-\Lcal u}_{L^2(\nu)}=0$, letting $N\to\infty$ in \eqref{eq:word-telescoping} gives $\norm u_{C(\Kcal)}=0$. If $0<\norm{u-\Lcal u}_{L^2(\nu)}<\norm u_{\Hscr_a}$, choose
\[
N \coloneqq \left\lceil
\frac{\log\big(\norm u_{\Hscr_a}/\norm{u-\Lcal u}_{L^2(\nu)}\big)}{\log(2/(\kappa\sqrt{\pi_*}))}
\right\rceil.
\]
Using $N\leq2^N$ in \eqref{eq:word-telescoping}, we can verify that the assertion holds with constant $\frac{2}{\sqrt{\pi_*}}+C_2+C_3$.
\end{proof}

\begin{lemma} \label{lem: spectral upper bound for our construction}
Suppose $b\in\Kcal$. For $\eps>0$, let $n(\eps)$ be the least nonnegative integer satisfying $E_{n(\eps)}\leq\eps$. Then for all sufficiently small $\eps$,
\begin{equation*}
n(\eps)+1\leq C\Ncal_a\big(C\log(1/\eps)\big).
\end{equation*}
\end{lemma}

\begin{proof}
For $t>0$, define $U_t$ as the linear span of the union of $m_\nu$ and all the eigenfunctions for eigenvalues larger than $e^{-t}$, where
\[
m_\nu \coloneqq \int_\Kcal k_x\,d\nu(x).
\]
The integral above exists by compactness of $\Kcal$. Also $m_\nu(a)=1$ since $K_a(x,a)=1$, so
$1\leq\dim U_t\leq\Ncal_a(t)$. Suppose $u\in\Hscr_a$ satisfies $u\perp U_t$ and $\norm{u}_{\Hscr_a}=1$. Then $\nu(u)=\ip{u}{m_\nu}_{\Hscr_a}=0$, and $\norm{u-\Lcal u}_{L^2(\nu)}~\leq~e^{-t/2}$ by the spectral theorem. Lemma~\ref{lem: Holder-type interpolation inequality} and equation~\eqref{eq: characterization of distance through reproducing property} therefore give, for some $C_1>0$,
\begin{equation*}
\sup_{x\in\Kcal}\dist(k_x,U_t)\leq C_1e^{-\alpha t/2}.
\end{equation*}

Since $b\in\Kcal$, the monotonicity of $E_n$ and Proposition~\ref{prop: approximation controls the defect} for $K=\Kcal$ give
\[
E_{2(d+1)\Ncal_a(t)}
\leq E_{2(d+1)\dim U_t}
\leq C_2e^{-\alpha t/8}
\]
with some $C_2>0$ for sufficiently large $t$. Then, taking $t=C\log(1/\eps)$ with a sufficiently large fixed $C$ proves the assertion.
\end{proof}

\subsection{Fractional kernel construction} \label{subsection: fractional kernel construction}

Now we treat the lower exponent. We need a function satisfying two requirements: its square must retain some portion of its stationary average, and that square must have a Poisson solution of bounded Hilbert norm. A small power of the kernel achieves both.

Let $\beta=1/(16(d+1))$ and define $F_g: \Delta\to\mathbb R$ for $g\in L^2(\nu)$ by
\begin{equation} \label{eq: definition of fractional kernel assignment}
F_g(x)=\int_\Kcal K_a(x,y)^\beta g(y)\,d\nu(y).
\end{equation}
Since $\Kcal$ is compact, $K_a$ extends to $\Delta\times\Kcal$ and is bounded there; thus $F_g$ is defined on $\Delta$ and depends linearly on $g$.

\begin{lemma}\label{lem: fractional kernel keeps stationary mean}
There is $C>0$ such that every $g\in L^2(\nu)$ with $\nu(g^2)=1$ satisfies
\begin{equation}\label{eq:small-power-retains-size}
\nu(F_g^2)\geq C\norm{\int_\Kcal k_y g(y)\,d\nu(y)}_{\Hscr_a}^{8}.
\end{equation}
\end{lemma}

\begin{proof}
The identity
\[
L(x,y)\coloneqq\log K_a(x,y)=\sum_{j=1}^{d+1}\int_0^\infty \frac1t\big(e^{-tx_j}-e^{-ta_j}\big)\big(e^{-ty_j}-e^{-ta_j}\big)\,dt
\]
shows that $L$ is a positive-definite kernel. Hence $((x,s),(y,t))\mapsto stL(x,y)$ is positive-definite on $\Kcal\times\mathbb R$. Since exponentiation preserves positive-definiteness (\cite[\S3: Corollary~1.14]{Berg--Christensen--Ressel}),
\[
\widetilde K\big((x,s),(y,t)\big)\coloneqq \exp\big(stL(x,y)\big)=K_a(x,y)^{st}
\]
is a positive-definite kernel on $\Kcal\times\mathbb R$.

For $t>0$, define $Q_t\in\mathbb R$ by
\[
Q_t\coloneqq\iint_{\Kcal\times\Kcal}K_a(x,y)^t g(x)g(y)\,d\nu(x)\,d\nu(y).
\]
Restricting $\widetilde K$ to $\Kcal\times\{\beta^{1/2},\beta^{-1/2}\}$ and using positive-definiteness of $\widetilde K$ shows that the matrix
\[
\begin{pmatrix}
Q_\beta & Q_1\\
Q_1 & Q_{1/\beta}
\end{pmatrix}
\]
is positive semidefinite, and hence $Q_1^2\leq Q_\beta Q_{1/\beta}$. Since $\Kcal$ is compact and $\nu(|g|)\leq\sqrt{\nu(g^2)}=1$,
\[
Q_{1/\beta}\leq\max_{x,y\in\Kcal}K_a(x,y)^{1/\beta}\,\nu(|g|)^2\leq C_0
\]
for some $C_0>0$ independent of $g$. Finally, the reproducing property and the definition of $F_g$ give
\[
Q_1=\norm{\int_\Kcal k_y g(y)\,d\nu(y)}_{\Hscr_a}^{2},
\qquad Q_\beta=\nu(gF_g).
\]
The following calculation using the Cauchy--Schwarz inequality and $\nu(g^2)=1$ proves the assertion.
\[
\nu(F_g^2) = \nu(F_g^2)\nu(g^2) \geq \nu(gF_g)^2 = Q_\beta^2 \geq C_0^{-2}Q_1^4 = C_0^{-2}\norm{\int_\Kcal k_y g(y)\,d\nu(y)}_{\Hscr_a}^{8}. \qedhere
\]
\end{proof}

\subsection{Connection to the Hardy kernel} \label{subsec: Connection to the Hardy kernel}

Before continuing with the proof, let us see the connection of $\Hscr_a$ to the Hardy space.
\begin{definition}
The \emph{Hardy space} $H^2(\mathbb D^{d+1})$ consists of holomorphic functions $u:\mathbb D^{d+1}\to\mathbb C$ for which
\[
\norm{u}_{H^2(\mathbb D^{d+1})}^2\coloneqq\sup_{0<r<1}\frac1{(2\pi)^{d+1}}\int_{[0,2\pi)^{d+1}}\abs{u(re^{i\theta_1},\ldots,re^{i\theta_{d+1}})}^2\,d\theta_1\cdots d\theta_{d+1}<\infty.
\]
\end{definition}

This is a reproducing kernel Hilbert space with the \emph{Szeg\H{o} kernel} \cite[\S1.4.3]{Paulsen--Raghupathi}
\[
K_{\mathrm{Hardy}}:\mathbb D^{d+1}\times\mathbb D^{d+1}\longrightarrow\mathbb C,
\qquad K_{\mathrm{Hardy}}(\zeta,\eta)\coloneqq\prod_{j=1}^{d+1}\frac1{1-\zeta_j\overline{\eta_j}}.
\]
Denote by $\operatorname{Re}z$ the real part of $z\in\mathbb C$, and write $\Omega\coloneqq\left\{z\in\mathbb C^{d+1} \setcond \operatorname{Re}z_j>0\text{ for all }j\right\}$ and $\mathbb D\coloneqq\{\zeta\in\mathbb C:\abs{\zeta}<1\}$. Define
\[
\chi_j:\Omega\longrightarrow\mathbb D,\qquad \chi_j(z)=\frac{z_j-a_j}{z_j+a_j}.
\qquad \big(j\in\{1,\ldots,d+1\}\big)
\]
Then $\chi = (\chi_1,\ldots,\chi_{d+1})$ identifies $\Omega$ with the polydisc $\mathbb D^{d+1}$ and sends $a$ to the origin. For $x,y\in\Delta^\circ$, direct substitution gives
\begin{equation} \label{eq: connection to Hardy kernel}
K_{\mathrm{Hardy}}\big(\chi(x),\chi(y)\big) = \prod_{j=1}^{d+1}\frac1{1-\chi_j(x)\chi_j(y)}=\prod_{j=1}^{d+1}\frac{(x_j+a_j)(y_j+a_j)}{2a_j(x_j+y_j)}=K_a(x,y).
\end{equation}
Therefore, our kernel is in fact a restriction of the Hardy kernel through the coordinatewise Cayley transform $\chi$. We will use the following Pull-back theorem.

\begin{theorem}[Pull-back theorem]{\cite[Proposition~5.6,\ Theorem~5.7]{Paulsen--Raghupathi}} \label{thm: pull-back theorem}
Let $X$ and $S$ be sets, $\varphi:S\to X$ a map, and $K:X\times X\to\mathbb C$ a positive-definite kernel. Define $K_\varphi:S\times S\to\mathbb C$ by
\[
K_\varphi(s,t) \coloneqq K\big(\varphi(s),\varphi(t)\big).
\]
Then $K_\varphi$ is positive-definite. Let $\Hscr(K)$ and $\Hscr(K_\varphi)$ denote the corresponding reproducing kernel Hilbert spaces of functions $X\to\mathbb C$ and $S\to\mathbb C$, respectively. Then
\[
\Hscr(K_\varphi) = \left\{ f\circ\varphi \setcond f\in\Hscr(K) \right\},
\]
and every $u\in\Hscr(K_\varphi)$ satisfies
\[
\norm{u}_{\Hscr(K_\varphi)} = \min\left\{ \norm{f}_{\Hscr(K)} \setcond f\in\Hscr(K),\ u=f\circ\varphi \right\}.
\]
\end{theorem}

\begin{remark}
By equation~\eqref{eq: connection to Hardy kernel} and Theorem~\ref{thm: pull-back theorem}, $\Hscr_a$ is isometric with the real Hilbert subspace
\[
\Hscr
\coloneqq
\overline{\Span_{\mathbb R}
\left\{ K_{\mathrm{Hardy}}\big(\,\cdot\,,\chi(x)\big) \setcond x\in\Delta^\circ \right\}}^{\,H^2(\mathbb D^{d+1})}
\subset H^2(\mathbb D^{d+1}),
\]
via $\Hscr\to\Hscr_a$, $F\mapsto F\circ\chi|_{\Delta^\circ}$. This identification will be implicit in the subsequent discussions.
\end{remark}

\subsection{Poisson equation with bounded norm} \label{subsection: Poisson equation with bounded norm}

The viewpoint in the previous subsection provides a necessary tool to prove the following.
\begin{lemma}\label{lem: Poisson equation with bounded norm}
There is $C>0$ such that the following holds. For each $g\in L^2(\nu)$ with $\nu(g^2)=1$, there is $h_g\in\Hscr_a$ satisfying $\norm{h_g}_{\Hscr_a}\leq C$ and
\begin{equation} \label{eq: Poisson equation}
(I-\Lcal)h_g = F_g^2-\nu(F_g^2)\quad\text{on }\Delta^\circ.
\end{equation}
\end{lemma}

To prove this, we will need two technical lemmas.

\begin{lemma}\label{lem: nearby positive forms}
Fix $t\geq0$. Suppose that $b_\ell,c_\ell\in(0,\infty)^{d+1}$, $\ell = 1,2,\ldots,2(d+1)$,  satisfy
\begin{equation} \label{eq: bc assumption}
b_\ell^\top a=c_\ell^\top a=1,
\qquad d_H(b_\ell,c_\ell)\leq t.
\end{equation}
Define
\[
V:\Delta^\circ\longrightarrow\mathbb R,
\qquad
V(x)\coloneqq\prod_{\ell=1}^{2(d+1)}
\left(\frac{b_\ell^\top x}{c_\ell^\top x}\right)^\beta.
\]
Then $V\in\Hscr_a$ and $\norm{V-1}_{\Hscr_a}\leq2^{1/4}t^{1/8}$.
\end{lemma}

\begin{proof}
%Maybe move the proof to appendices?
If $t=0$, then equation~\eqref{eq: bc assumption} implies $b_\ell=c_\ell$ for every $\ell$, so $V=1$. Suppose $t>0$. The same formula defines the extension of $V$ to $\Omega$, taking principal powers of the individual ratios. Then $V$ is holomorphic on $\Omega$. Indeed, for $z\in\Omega$, the arguments of both $b_\ell^\top z$ and $c_\ell^\top z$ lie in $(-\pi/2,\pi/2)$, since $b_\ell,c_\ell$ are positive vectors and $\operatorname{Re}z_j>0$ for every $j$. Their ratios therefore avoid the negative real axis.

To apply the Pull-back theorem~\ref{thm: pull-back theorem}, we estimate the Hardy norm of $(V-1)\circ\chi^{-1}$. For $0<r<1$, let $\mathbb E_r$ denote angular averaging in the Cayley coordinates:
\[
\big(z_r(\theta)\big)_j \coloneqq a_j\frac{1+re^{i\theta_j}}{1-re^{i\theta_j}} \quad (1\leq j\leq d+1),
\qquad
\mathbb E_r[W]\coloneqq\frac1{(2\pi)^{d+1}}
\int_{[0,2\pi)^{d+1}}W\big(z_r(\theta)\big)\,d\theta,
\]
for integrable $W:\Omega\to \mathbb C$. Since $\chi\big(z_r(\theta)\big)=(re^{i\theta_1},\ldots,re^{i\theta_{d+1}})$, by Theorem~\ref{thm: pull-back theorem},
\[
\sup_{0<r<1}\mathbb E_r\big[|V-1|^2\big]
= \norm{(V-1)\circ\chi^{-1}}_{H^2(\mathbb D^{d+1})}^2
\geq \norm{V-1}_{\Hscr_a}^2,
\]
once the supremum is proven to be finite. We will therefore bound $\mathbb E_r\big[|V-1|^2\big]$.

The mean-value theorem gives $\mathbb E_r[V]=V(a)=1$. Then
\begin{equation}\label{eq: mean square of V minus one}
\mathbb E_r\big[|V-1|^2\big]
=\mathbb E_r\big[|V|^2\big]-2\operatorname{Re}\mathbb E_r[V]+1
=\mathbb E_r\big[|V|^2\big]-1.
\end{equation}
Since $\beta=1/(16(d+1))$, the inequality of arithmetic and geometric means gives
\begin{equation}\label{eq: product bounded by fourth roots}
|V(z)|^2
=\left(
\prod_{\ell=1}^{2(d+1)}\abs{\,b_\ell^\top z / c_\ell^\top z\,}^{1/4}
\right)^{1/(2(d+1))}
\!\!\leq\;
\frac1{2(d+1)}\sum_{\ell=1}^{2(d+1)}
\abs{\,b_\ell^\top z / c_\ell^\top z\,}^{1/4}.
\end{equation}
We therefore estimate the $\mathbb E_r$-average of each summand on the right. Fix $\ell$ and write $b_\ell = (b_{\ell,1},\ldots,b_{\ell,\,d+1})$ and $c_\ell = (c_{\ell,1},\ldots,c_{\ell,\,d+1})$. Since $b_\ell^\top a=c_\ell^\top a=1$, we have $\sum_{j=1}^{d+1}a_jc_{\ell,j}\cdot\frac{b_{\ell,j}}{c_{\ell,j}}=1$. Thus the maximum of the ratios $b_{\ell,j}/c_{\ell,j}$ is at least one. By $d_H(b_\ell,c_\ell)\leq t$,
\[
\min_j\frac{b_{\ell,j}}{c_{\ell,j}}\geq e^{-t}\max_j\frac{b_{\ell,j}}{c_{\ell,j}}\geq e^{-t}.
\]
Consequently, $b_\ell-e^{-t}c_\ell$ has nonnegative entries. Define $v_\ell \coloneqq \displaystyle \frac{b_\ell-e^{-t}c_\ell}{1-e^{-t}}$. Then $v_\ell\geq0$ entrywise, and
\begin{equation}\label{eq: ratio with small remainder}
\frac{b_\ell^\top z}{c_\ell^\top z}=e^{-t}+(1-e^{-t})\frac{v_\ell^\top z}{c_\ell^\top z}.
\end{equation}

We now bound $\mathbb E_r\left[\,\abs{\left(v_\ell^\top z / c_\ell^\top z\right)}^{1/4}\,\right]$. Since $b_\ell^\top a=c_\ell^\top a=1$, we have $v_\ell^\top a~=~1$. In particular, $v_\ell\ne0$. Since $\operatorname{Re}z_j>0$ for every $j$, both $v_\ell^\top z$ and $c_\ell^\top z$ have positive real parts. Their ratio therefore avoids the negative real axis. Then the map $z\mapsto \left(v_\ell^\top z / c_\ell^\top z\right)^{1/4}$ is holomorphic on $\Omega$ and has argument in $(-\pi/4,\pi/4)$. Therefore,
\[
\operatorname{Re}\left(v_\ell^\top z / c_\ell^\top z\right)^{1/4}
\geq \cos\left(\frac{\pi}{4}\right) \abs{\left(v_\ell^\top z / c_\ell^\top z\right)}^{1/4}.
\]
Also, by the mean-value theorem, $\mathbb E_r\left[\operatorname{Re}\left(v_\ell^\top z / c_\ell^\top z\right)^{1/4}\right] = \left(v_\ell^\top a / c_\ell^\top a\right)^{1/4}=1$. Therefore,
\begin{equation} \label{eq: angular average of ratio for the modified vector}
\mathbb E_r\left[\,\abs{\,v_\ell^\top z / c_\ell^\top z\,}^{1/4}\,\right] \leq \sqrt2.
\end{equation}

We return to equation~\eqref{eq: ratio with small remainder}. The triangle inequality and the inequality $(u+w)^{1/4}\leq u^{1/4}+w^{1/4}$ for $u,w\geq0$ give 
\[
\abs{\,b_\ell^\top z / c_\ell^\top z\,}^{1/4}
\leq\left(e^{-t}+(1-e^{-t})\abs{\,v_\ell^\top z / c_\ell^\top z\,}\right)^{1/4}
\leq e^{-t/4}+(1-e^{-t})^{1/4}\abs{\,v_\ell^\top z / c_\ell^\top z\,}^{1/4}.
\]
Taking the mean with $\mathbb E_r$ and using equation~\eqref{eq: angular average of ratio for the modified vector}, we obtain
\[
\mathbb E_r\left[\,\abs{\,b_\ell^\top z / c_\ell^\top z\,}^{1/4}\,\right]
\leq e^{-t/4}+\sqrt2(1-e^{-t})^{1/4}.
\]
This holds for every $\ell$. Averaging equation~\eqref{eq: product bounded by fourth roots} therefore gives
\[
\mathbb E_r\big[|V|^2\big]\leq e^{-t/4}+\sqrt2(1-e^{-t})^{1/4}.
\]
Substituting into equation~\eqref{eq: mean square of V minus one} and using $e^{-t/4}\leq1$ and $1-e^{-t}\leq t$, we conclude that
\[
\mathbb E_r\big[|V-1|^2\big]\leq e^{-t/4}-1+\sqrt2(1-e^{-t})^{1/4}\leq\sqrt2\,t^{1/4}.
\]

This holds for every $0<r<1$. Hence $(V-1)\circ\chi^{-1}\in H^2(\mathbb D^{d+1})$, with
\[
\norm{(V-1)\circ\chi^{-1}}_{H^2(\mathbb D^{d+1})}^2=\sup_{0<r<1}\mathbb E_r\big[|V-1|^2\big]\leq\sqrt2\,t^{1/4}.
\]
Theorem~\ref{thm: pull-back theorem} gives $V-1\in\Hscr_a$ and $\norm{V-1}_{\Hscr_a}\leq2^{1/4}t^{1/8}$. Finally, $K_a(x,a)=1$, so the constant function $1$ belongs to $\Hscr_a$, so $V\in\Hscr_a$.
\end{proof}

The series $\sum_{n=0}^\infty G_n$, with $G_n$ defined below, yields the Poisson solution.

\begin{lemma} \label{lem: decay of centered iterates}
There are $C>0$ and $0<\kappa<1$ such that the following holds. For each $g\in L^2(\nu)$ with $\nu(g^2)=1$ and each $n\geq0$, define
\[
G_n: \Delta^\circ \longrightarrow \mathbb R,
\qquad
G_n(x)\coloneqq\Lcal^n(F_g^2)(x)-\Lcal^n(F_g^2)(a).
\]
Then $G_n\in\Hscr_a$ and $\norm{G_n}_{\Hscr_a}\leq C\kappa^{n/8}$.
\end{lemma}

\begin{proof}
Fix $g$ as in the statement and an integer $n\geq0$. For a word $w$ of length $n$ and $y,y'\in\Kcal$, let
\[
U_{w,y,y'}:\Delta^\circ\longrightarrow(0,\infty),
\qquad
U_{w,y,y'}(x)\coloneqq K_a(f_wx,y)^\beta K_a(f_wx,y')^\beta.
\]
For the empty word $w=\varnothing$ with $n=0$, we set $A_w=I$, $f_wx=x$, and $\pi_w=1$.

Recall the definition of $F_g$ in equation~\eqref{eq: definition of fractional kernel assignment}. Since $K_a$ is bounded on $\Delta\times\Kcal$ and $\nu(\abs{g})\leq\sqrt{\nu(g^2)}=1$, Fubini's theorem implies that we can write $F_g^2$ as a double integral. Applying $\Lcal^n$ to $F_g^2$ and subtracting the value at $a$ gives
\begin{equation}\label{eq: centered iterates as kernel products}
G_n(x) = \sum_{\abs{w}=n}\pi_w\iint_{\Kcal\times\Kcal}\big(U_{w,y,y'}(x) - U_{w,y,y'}(a)\big)g(y)g(y')\,d\nu(y)d\nu(y').
\end{equation}

We first estimate the Hilbert norm of $U_{w,y,y'}-U_{w,y,y'}(a)$. Fix $\abs{w}=n$ and $y,y'\in\Kcal$. Let $e_j\in \mathbb R^{d+1}$ be the standard basis vectors. Since $f_wx=A_wx/(\one^\top A_wx)$, the definition of $K_a$ gives
\begin{equation} \label{eq: ratio expression of the kernel}
K_a(f_wx,y)
=\prod_{j=1}^{d+1}\frac{y_j+a_j}{2a_j}\frac{(f_wx)_j+a_j}{(f_wx)_j+y_j}
=\prod_{j=1}^{d+1}\frac{y_j+a_j}{2a_j}\frac{(e_j+a_j\one)^\top A_wx}{(e_j+y_j\one)^\top A_wx}.
\end{equation}
To apply Lemma~\ref{lem: nearby positive forms}, we need the coefficient vectors to have inner product one with $a$. For $j=1,\ldots,d+1$, define vectors $b_j, c_j(\zeta)\in(0,\infty)^{d+1}$ for $\zeta\in\Kcal$ by
\[
b_j\coloneqq\frac{A_w^\top(e_j+a_j\one)}{(e_j+a_j\one)^\top A_wa},
\qquad
c_j(\zeta)\coloneqq\frac{A_w^\top(e_j+\zeta_j\one)}{(e_j+\zeta_j\one)^\top A_wa}.
\]
These vectors are positive and satisfy $b_j^\top a=c_j(\zeta)^\top a=1$. Dividing $K_a(f_wx,\zeta)$ by $K_a(f_wa,\zeta)$ using equation~\eqref{eq: ratio expression of the kernel} cancels the factors independent of $x$ and gives
\[
\frac{K_a(f_wx,\zeta)}{K_a(f_wa,\zeta)}=\prod_{j=1}^{d+1}\frac{b_j^\top x}{c_j(\zeta)^\top x}.
\]
Then
\begin{equation}\label{eq: normalized pair of kernel products}
\frac{U_{w,y,y'}(x)}{U_{w,y,y'}(a)}
=\prod_{j=1}^{d+1}\left(\frac{b_j^\top x}{c_j(y)^\top x}\right)^\beta\left(\frac{b_j^\top x}{c_j(y')^\top x}\right)^\beta.
\end{equation}

We next estimate $d_H\big(b_j,c_j(\zeta)\big)$. By the choice of $\kappa$ right before Lemma~\ref{lem: basic inequalities for rho and theta},
\[
d_H(A_i^\top u,A_i^\top v)\leq\kappa d_H(u,v)
\]
for every $i$ and all positive vectors $u,v$. Applying this inequality $n$ times gives, for all $\zeta\in\Kcal$,
\[
d_H\big(b_j,c_j(\zeta)\big)
= d_H\big(A_w^\top(e_j+a_j\one),A_w^\top(e_j+\zeta_j\one)\big)
\leq \kappa^n d_H(e_j+a_j\one,e_j+\zeta_j\one)
\leq C_0\kappa^n,
\]
where we used the compactness of $\Kcal$ to obtain $C_0>0$. This also holds for the empty word with $n=0$. Therefore, the $2(d+1)$ pairs of vectors $\big(b_j, c_j(y)\big)$ and $\big(b_j, c_j(y')\big)$ for $j\in\left\{1,\ldots,d+1\right\}$ satisfy equation~\eqref{eq: bc assumption} with $t=C_0\kappa^n$.

Applying Lemma~\ref{lem: nearby positive forms} to the $2(d+1)$ pairs of vectors and using equation~\eqref{eq: normalized pair of kernel products}, we obtain $U_{w,y,y'}/U_{w,y,y'}(a)\in\Hscr_a$. Also, $U_{w,y,y'}(a)$ is positive and it is uniformly bounded for $y,y'$ because $K_a$ is bounded on $\Delta\times\Kcal$. The norm estimate from Lemma~\ref{lem: nearby positive forms} therefore gives
\begin{equation}\label{eq:poisson-word-norm}
\norm{U_{w,y,y'} - U_{w,y,y'}(a)}_{\Hscr_a}
= U_{w,y,y'}(a)\norm{\frac{U_{w,y,y'}}{U_{w,y,y'}(a)}-1}_{\Hscr_a}
\leq 2^{1/4}U_{w,y,y'}(a)C_0^{1/8}\kappa^{n/8}
\leq C_1\kappa^{n/8},
\end{equation}
where $C_1>0$ is independent of $w,y,y'$ and $n$.

We return to equation~\eqref{eq: centered iterates as kernel products}. By Lemma~\ref{lem: characterization of membership in the Hilbert space}, to obtain the bound for $\norm{G_n}_{\Hscr_a}$ it suffices to show that for every finitely-supported signed measure $\sigma$ on $\Delta^\circ$, we have $\abs{\sigma(G_n)}\leq C_1\kappa^{n/8}\norm{\sigma}_{K_a}$. Fix such $\sigma$. The reproducing property, Cauchy--Schwarz inequality, and equation~\eqref{eq:poisson-word-norm} give
\[
\abs{\sigma\big(U_{w,y,y'} - U_{w,y,y'}(a)\big)}\leq C_1\kappa^{n/8}\norm{\sigma}_{K_a}.
\]
Since $\sigma$ has finite support, by applying $\sigma$ to equation~\eqref{eq: centered iterates as kernel products} and using Fubini's theorem,
\begin{align*}
\abs{\sigma(G_n)}&\leq\sum_{\abs{w}=n}\pi_w\iint_{\Kcal\times\Kcal}C_1\kappa^{n/8}\norm{\sigma}_{K_a}\abs{g(y)g(y')}\,d\nu(y)d\nu(y')\\
&=C_1\kappa^{n/8}\norm{\sigma}_{K_a}\left(\sum_{\abs{w}=n}\pi_w\right)\nu(\abs{g})^2
\leq C_1\kappa^{n/8}\norm{\sigma}_{K_a},
\end{align*}
where we used $\sum_{\abs{w}=n}\pi_w=1$ and $\nu(\abs{g})^2\leq\nu(g^2)\nu(1^2)=1$. Lemma~\ref{lem: characterization of membership in the Hilbert space} gives $G_n\in\Hscr_a$ and the asserted norm bound.
\end{proof}

\begin{proof}[Proof of Lemma~\ref{lem: Poisson equation with bounded norm}]
Fix $g\in L^2(\nu)$ with $\nu(g^2)=1$, and let $G_n$ be as in Lemma~\ref{lem: decay of centered iterates}. By that lemma,
\[
\sum_{n=0}^\infty \norm{G_n}_{\Hscr_a}\leq C\sum_{n=0}^\infty\kappa^{n/8} = \frac{C}{1-\kappa^{1/8}}<\infty.
\]
Thus the series $\sum_{n=0}^\infty G_n$ converges in $\Hscr_a$. Convergence in $\Hscr_a$ implies convergence at every point by the reproducing property; then we can define
\[
h_g:\Delta^\circ\to\mathbb R,
\qquad
h_g(x)\coloneqq\sum_{n=0}^\infty G_n(x).
\]
Moreover, $\norm{h_g}_{\Hscr_a}\leq C/(1-\kappa^{1/8})$, independently of $g$.

Let us now see that $h_g$ is the solution to the Poisson equation~\eqref{eq: Poisson equation}. Since $\Lcal1=1$, applying $I-\Lcal$ to $G_n$ removes the constant $\Lcal^n(F_g^2)(a)$. Then, for each integer $N\geq0$,
\begin{equation} \label{eq: finite Poisson sum}
(I-\Lcal)\sum_{n=0}^N G_n
= \sum_{n=0}^N\big(\Lcal^n(F_g^2) - \Lcal^{n+1}(F_g^2)\big)
= F_g^2 - \Lcal^{N+1}(F_g^2).
\end{equation}
We will show that $\Lcal^n(F_g^2)(x)\to\nu(F_g^2)$ as $n\to\infty$ for every $x\in\Delta^\circ$. The reproducing property, Cauchy--Schwarz inequality, and Lemma~\ref{lem: decay of centered iterates} give
\[
\abs{G_n(x)}\leq\sqrt{K_a(x,x)}\norm{G_n}_{\Hscr_a}\leq C\sqrt{K_a(x,x)}\kappa^{n/8}.
\]
Thus $G_n(x)\to0$ as $n\to\infty$ for every $x\in\Delta^\circ$. Since $K_a(x,x)$ is bounded on the compact set $\Kcal$, this convergence is uniform on $\Kcal$, so $\nu(G_n)\to0$. Now, stationarity of $\nu$ and the definition of $G_n$ give
\[
\nu(G_n) = \nu\big(\Lcal^n(F_g^2)\big)-\Lcal^n(F_g^2)(a) = \nu(F_g^2)-\Lcal^n(F_g^2)(a).
\]
Using this and the definition of $G_n$, we obtain
\begin{equation} \label{eq: convergence of fractional assignment to stationary mean}
\Lcal^n(F_g^2)(x) - \nu(F_g^2)
= G_n(x) - \nu(G_n)
\xrightarrow{n\to\infty}0.
\end{equation}

It remains to pass to the limit on the left of equation~\eqref{eq: finite Poisson sum}. For each $x\in\Delta^\circ$, the series $\sum_{n=0}^\infty G_n$ converges at $x$ and at each of the finitely many points $f_i x$. Hence
\begin{equation} \label{eq: transferred sum of Gn converges}
\Lcal\left(\sum_{n=0}^N G_n\right)(x)
= \sum_i\pi_i\sum_{n=0}^N G_n(f_i x)
\xrightarrow{N\to\infty}\sum_i\pi_i h_g(f_i x)
= \Lcal h_g(x).
\end{equation}
Hence, passing to the limit in equation~\eqref{eq: finite Poisson sum} gives equation~\eqref{eq: Poisson equation} by equations \eqref{eq: convergence of fractional assignment to stationary mean} and \eqref{eq: transferred sum of Gn converges}.
\end{proof}

\subsection{Exponents for minimal model size} \label{subsection: exponents for minimal model size}

We now bound the stationary defect of an arbitrary constellation from below by the eigenvalues. If the constellation has at most $r$ support points, we can choose a nonzero linear combination $u$ of $r+1$ eigenfunctions such that $F_u$ vanishes at those points. We will bound $\nu(F_u^2)$ from below using the eigenvalues, and from above using the Poisson solution $h_u$.

\begin{proposition}\label{prop: eigenvalues lower bounds the stationary defect}
There is $C\geq1$ such that every constellation $\eta$ with at most $r$ support points satisfies
\begin{equation*}
\sigma_{r+1}^{4}\leq C\norm{\eta-\Lcal^*\eta}_{K_a}.
\end{equation*}
\end{proposition}

\begin{proof}
If $\sigma_{r+1}=0$, the assertion follows. Suppose $\sigma_{r+1}>0$, and let $U\subset\Hscr_a$ be the span of the first $r+1$ orthonormal eigenfunctions of $\Cscr_a$, counting multiplicities. For every $u\in U$, the spectral theorem gives
\[
\norm{u-\Lcal u}_{L^2(\nu)}^2 = \ip{u}{\Cscr_a u}_{\Hscr_a}\geq\sigma_{r+1}\norm{u}_{\Hscr_a}^2.
\]
Also, as we have seen in the proof of Lemma~\ref{lem: Holder-type interpolation inequality}, $\Lcal$ is non-expansive on $L^2(\nu)$, so $\norm{u-\Lcal u}_{L^2(\nu)}\leq2\norm{u}_{L^2(\nu)}$. Combining these inequalities gives
\begin{equation}\label{eq: leading eigenfunctions control Hilbert norm}
\sqrt{\sigma_{r+1}}\norm{u}_{\Hscr_a}\leq2\sqrt{\nu(u^2)}.
\qquad(u\in U)
\end{equation}
In particular, nonzero $u\in U$ satisfies $\nu(u^2)>0$.

For each $x\in\supp\eta$, the map $U\to\mathbb R$, $u\mapsto F_u(x)$, is linear. Thus the equations $F_u(x)=0$ for $x\in\supp\eta$ impose at most $r$ linear conditions on $U$. Since $\dim U=r+1$, there is a nonzero $u\in U$ for which $F_u(x) = 0$ for all $x\in \supp\eta$. Let us scale it so that $\nu(u^2)=1$.

We want to sandwich $\nu(F_u^2)$. With Lemma~\ref{lem: fractional kernel keeps stationary mean} in mind, we will estimate the norm of the kernel average of $u$. Define
\[
\psi:\Delta^\circ\longrightarrow \mathbb R,
\qquad
\psi \coloneqq \int_\Kcal k_yu(y)\,d\nu(y).
\]
The integral above exists by the compactness of $\Kcal$. The reproducing property and Cauchy--Schwarz inequality give
\[
\nu(u^2)=\ip{u}{\psi}_{\Hscr_a}\leq\norm{u}_{\Hscr_a}\norm{\psi}_{\Hscr_a}.
\]
Equation~\eqref{eq: leading eigenfunctions control Hilbert norm} and $\nu(u^2)=1$ imply $\norm{\psi}_{\Hscr_a}\geq\sqrt{\sigma_{r+1}}/2$. Lemma~\ref{lem: fractional kernel keeps stationary mean} then gives a constant $C_1>0$ independent of $\eta$ and $r$ such that
\begin{equation}\label{eq: replacement has spectral lower bound}
\nu(F_u^2)\geq C_1\norm{\psi}_{\Hscr_a}^{8}\geq\frac{C_1}{2^8}\sigma_{r+1}^4.
\end{equation}

We next bound the same $\nu(F_u^2)$ from above by the stationary defect of $\eta$. Let $h_u:\Delta^\circ\to\mathbb R$ be the Poisson solution supplied by Lemma~\ref{lem: Poisson equation with bounded norm}, which satisfies $\norm{h_u}_{\Hscr_a}\leq C_2$ for some fixed $C_2>0$. By the definition of $u$, we have $\eta(F_u^2)=0$. Since $\eta(1)=1$, applying $\eta$ to the Poisson equation~\eqref{eq: Poisson equation} gives
\[
(\eta-\Lcal^*\eta)(h_u) = \eta\big((I-\Lcal)h_u\big) = \eta(F_u^2)-\nu(F_u^2) = -\nu(F_u^2).
\]
The reproducing property, Cauchy--Schwarz inequality, and the bound on $\norm{h_u}_{\Hscr_a}$ therefore give
\[
\nu(F_u^2) = \abs{(\eta-\Lcal^*\eta)(h_u)} \leq \norm{\eta-\Lcal^*\eta}_{K_a}\norm{h_u}_{\Hscr_a} \leq C_2\norm{\eta-\Lcal^*\eta}_{K_a}.
\]
Combining this with equation~\eqref{eq: replacement has spectral lower bound} proves the assertion.
\end{proof}

\begin{lemma} \label{lem: spectral lower bound for minimal model size}
Fix $0<c<1/4$. Then for all sufficiently small $\eps>0$,
\begin{equation*}
r_{\min}^{\,a}(\eps) \geq \frac12\Ncal_a\big(c\log(1/\eps)\big).
\end{equation*}
\end{lemma}

\begin{proof}
Let $C\geq1$ be the constant in Proposition~\ref{prop: eigenvalues lower bounds the stationary defect}, and put $t = c\log(1/\eps)$. Since $4c<1$, for all sufficiently small $0<\eps<1$ we have $t\geq0$ and
\[
C\eps < \eps^{4c} = e^{-4t}.
\]
Let $\eta$ be any constellation satisfying $\norm{\eta - \Lcal^*\eta}_{K_a} \leq \eps$, and put $r \coloneqq \#\supp\eta\geq1$. Proposition~\ref{prop: eigenvalues lower bounds the stationary defect} gives
\[
\sigma_{r+1}^4 \leq C\eps < e^{-4t}, \qquad \text{so} \qquad \sigma_{r+1} < e^{-t}.
\]
Since the eigenvalues are in non-increasing order, at most $r$ of them are more than or equal to $e^{-t}$. By the definition of $\Ncal_a$, we obtain $r \geq \Ncal_a(t) - 1$. Taking the minimum over all constellations gives
\[
r_{\min}^{\,a}(\eps) \geq \max\{1,\Ncal_a(t) - 1\} \geq \frac12\Ncal_a(t) = \frac12\Ncal_a\big(c\log(1/\eps)\big). \qedhere
\]
\end{proof}

We now give the proof that constellation dimension governs the minimal model size.

\begin{proof}[Proof of Proposition~\ref{prop: optimal cardinality is the constellation dimension} and equation~\eqref{eq: optimality of the construction}]
Fix $b\in\Kcal$. Let $n(\eps)$ be the least nonnegative integer satisfying $E_{n(\eps)}\leq\eps$, as in Lemma~\ref{lem: spectral upper bound for our construction}. Thus, for all sufficiently small $\eps>0$, the constellation $\eta(\eps)$ in Theorem~\ref{thm: intrinsic complexity} is $\eta_{n(\eps)}$.

By the definition of $r_{\min}^{\,a}(\eps)$, Lemma \ref{lem: spectral upper bound for our construction}, and Lemma~\ref{lem: spectral lower bound for minimal model size}, we have some $C>0$ with
\begin{equation} \label{eq:minimum-sandwich}
C^{-1}\Ncal_a\big(C^{-1}\log(1/\eps)\big) \leq r_{\min}^{\,a}(\eps) \leq \#\supp\eta(\eps)
\leq 2n(\eps)+1 \leq C\Ncal_a\big(C\log(1/\eps)\big).
\end{equation}
Taking logarithms and limits proves Proposition~\ref{prop: optimal cardinality is the constellation dimension} and equation~\eqref{eq: optimality of the construction}.
\end{proof}

\section{Invertibility and algebraic dimension} \label{section: invertibility and algebraic dimension}

In this section we will prove Theorem~\ref{thm: invertible subfamily lower bounds constellation dimension}.

\subsection{Overview}

Recall that $\ip{u}{\Cscr_a u}_{\Hscr_a}=\norm{u-\Lcal u}_{L^2(\nu)}^2$. By the Min--Max principle (Theorem~\ref{thm: min--max principle}), it suffices to find a large space $U\subset\Hscr_a$ whose every unit function has a stationary defect $\norm{u-\Lcal u}_{L^2(\nu)}$ bounded from below. Lemma~\ref{lem: Holder-type interpolation inequality} provides a way to control $\norm{u-\Lcal u}_{L^2(\nu)}$: every $u\in\Hscr_a$ with $\nu(u)=0$ satisfies
\[
\norm{u}_{C(\Kcal)} \;\leq\; C\norm{u-\Lcal u}_{L^2(\nu)}^\alpha\norm{u}_{\Hscr_a}^{1-\alpha}.
\]
Therefore, if we could prevent functions from being too small on $\Kcal$ relative to their Hilbert norms, then their stationary defect will not be too small. Then, how can we construct $U$ on which small values of $\norm{\cdot}_{C(\Kcal)}$ force a small $\norm{\cdot}_{\Hscr_a}$?

A useful bridge is the Pull-back theorem~\ref{thm: pull-back theorem}: it allows us to bound $\norm{\cdot}_{\Hscr_a}$ using the supremum norm $\norm{\cdot}_{C(\mathbb D^{d+1})}$ on the complex closed polydisk $\overline{\mathbb D}^{d+1}$ through the Cayley transform $\chi$.

Now suppose that $U$ has a basis $\{P_1,\ldots,P_q\}$. Then we need to control every linear combination $Q = \sum_j c_jP_j$. We need $U$ to satisfy two properties:
\begin{align*}
&\text{[Property X]: \qquad Small $\norm{P_j}_{C(\Kcal)}$ $\implies$ Small $\norm{P_j}_{C(\mathbb D^{d+1})}$.} \\
&\text{[Property Y]: \qquad Small $\norm{Q}_{C(\Kcal)}$ $\implies$ Small $\norm{c_jP_j}_{C(\Kcal)}$ for all $j$.}
\end{align*}
Then we obtain the implication:
\[
\text{Small $\norm{Q}_{C(\Kcal)}$ $\implies$ Small $\norm{Q}_{C(\mathbb D^{d+1})}$,}
\]
and thus follows the desired control over the stationary defect of $Q$. The difficult part is to construct $U$ so that Property Y holds, since arbitrary linear combinations could exhibit cancellations among the generators $P_j$.

This is where invertibility plays a role: it lets us pull back a polynomial on a small copy of $\Kcal_0$ to compress it in that region. By combining it with an analogue of Lagrange interpolation polynomials, we can construct bump-like functions scattered across $\Kcal_0$, so that each function's contribution remains visible in its own region and cannot be cancelled by the others. The resulting $U$ therefore satisfies Property Y.

\subsection{Polynomial observations in the Hilbert space}

We first see that every polynomial on $\Delta$ can be turned into a Hilbert-space function by multiplying by a rational function. Write $\mathcal P_n^{\mathrm{hom}}$ for the real homogeneous polynomials of degree $n$ in $d+1$ variables:
\[
\mathcal P_n^{\mathrm{hom}}
\;\coloneqq\;
\Span_{\mathbb R}
\left\{
X_1^{\alpha_1}\cdots X_{d+1}^{\alpha_{d+1}} \setcond \alpha_j\in\mathbb Z_{\geq0}, \;\alpha_1+\cdots+\alpha_{d+1}=n
\right\}
\subset\mathbb R[X_1,\ldots,X_{d+1}].
\]
Since $\sum_{j=1}^{d+1} X_j = 1$ as a polynomial on $\Delta$, every polynomial $P\in \mathcal P_N(\Delta)$ has a homogeneous representation $\widetilde P\in \mathcal P_N^{\mathrm{hom}}$ such that $P(x) = \widetilde P(x)$ for all $x\in\Delta$; this follows by multiplying each monomial of degree $k\leq N$ of $P$ by $\big(\sum_{j=1}^{d+1} X_j\big)^{N-k}$.

\begin{lemma} \label{lem: weighted polynomials}
For an integer $n\geq0$, define $\omega_n: \Delta \to \mathbb R$ by $\omega_n(x) \coloneqq \prod_{j=1}^{d+1}(1+x_j)^{-n}$. Then for every $P\in\mathcal P_n^{\mathrm{hom}}$,
\[
\omega_nP\in\Hscr_a,\qquad
\norm{\omega_nP}_{\Hscr_a}\leq\norm P_{C(\mathbb D^{d+1})},\qquad
e^{-n} \leq \omega_n(x) \leq1 \quad (x\in\Delta).
\]
\end{lemma}

\begin{proof}
The same expression extends $\omega_n$ to a holomorphic function on $\Omega$. We have $\omega_n(z)P(z) = P\big(\omega_1(z)z_1,\ldots,\omega_1(z)z_{d+1}\big)$ for $z\in\Omega$ by homogeneity of $P$. Since
\[
\abs{\omega_1(z)z_j}
=\left|\frac{z_j}{1+z_j}\right|\prod_{k\ne j}\left|\frac1{1+z_k}\right|<1,
\qquad(z\in\Omega)
\]
we have $\abs{\omega_n(z)P(z)}\leq\norm P_{C(\mathbb D^{d+1})}$ on $\Omega$. Thus the Pull-back theorem~\ref{thm: pull-back theorem} implies $\omega_nP\in\Hscr_a$ and the norm bound. The last bound follows from $1\leq\prod_j(1+x_j)\leq e^{\sum_jx_j}=e$.
\end{proof}

\subsection{Preventing cancellation} \label{subsection: preventing cancellation}

We construct polynomial spaces satisfying Properties X and Y.

\begin{lemma} \label{lem: enlarging polynomial spaces}
Let $A_0,\ldots,A_{m-1}$ be positive matrices. Let $\Kcal_0$ be the attractor of a nonempty invertible subfamily $\mathcal A$, and fix $0<s<\dim_\mathbb C\overline{\Kcal_0}^{\mathrm{Zar}}$. There are a constant $B\geq1$ and integers $q,\Lambda\geq2$ with $q>\Lambda^s$ such that there is $\mathcal E_\ell\subset\mathcal P_{\Lambda^\ell}^{\mathrm{hom}}$ for every integer $\ell\geq0$ satisfying
\begin{equation} \label{eq: growing polynomial spaces}
\norm{Q}_{C(\mathbb D^{d+1})}\leq B^{\Lambda^\ell}\norm{Q}_{C(\Kcal_0)}
\;\;\text{for all}\;\; Q\in \mathcal E_\ell, \;\text{and} \;\;
\dim \mathcal E_\ell=q^\ell.
\end{equation}
\end{lemma}

\begin{proof}
Write $\mathcal A = \{M_j\}_{j\in J}$ and $M_w = M_{w_1}\cdots M_{w_n}$ for a word $w = (w_1,\ldots,w_n)$. Within this proof, we multiply each $M_j$, $j\in J$, by a positive scalar so that its maximal column sum is one; this does not change their projective maps or $\Kcal_0$.

Choose $R\geq1$ upper-bounding the absolute row sums of the inverse of every $M_j$. For a word $w\in J^N$, we have $0<\one^\top M_wy\leq1$ for $y\in\Delta$, and every absolute row sum of $M_w^{-1}$ is at most $R^N$. Then $R^{-N}M_w^{-1}z\in\mathbb D^{d+1}$ whenever $z\in\mathbb D^{d+1}$. Let $n\geq1$ be an integer and fix $P\in\mathcal P_n^{\mathrm{hom}}$. We have
\[
\abs{P(M_w^{-1}z)}=R^{Nn}\abs{P(R^{-N}M_w^{-1}z)}\leq R^{Nn}\norm{P}_{C(\mathbb D^{d+1})}.
\qquad (z\in\mathbb D^{d+1})
\]
Since $M_w^{-1}f_wy=y/(\one^\top M_wy)$, the bound above and the homogeneity give
\begin{equation} \label{eq: polynomial on a small copy}
\norm{P\circ M_w^{-1}}_{C(\mathbb D^{d+1})}\leq R^{Nn}\norm{P}_{C(\mathbb D^{d+1})}
\qquad
(P\circ M_w^{-1})(f_wy)=\frac{P(y)}{(\one^\top M_wy)^n}.
\end{equation}

We next choose bump-like polynomials which are close to one on their own copies of $\Kcal_0$ and close to zero on the other copies. Choose an integer $c\geq1$ such that $R^2\kappa^c<1$. By Lemma~\ref{lem: polynomial dimension growth} and $s<\dim_\mathbb C\overline{\Kcal_0}^{\mathrm{Zar}}$, we can choose an integer $p\geq1$ large enough that
\[
q\coloneqq\dim\mathcal P_p(\Kcal_0)>\Lambda^s,
\qquad \Lambda\coloneqq pc+1.
\]
If $P\in \mathcal P_p(\Kcal_0)$ vanishes on $\Kcal_0$, it is $0$ as a polynomial in $\mathcal P_p(\Kcal_0)$. Hence the evaluation functional $\mathcal P_p(\Kcal_0)\longrightarrow\mathbb R$, $P\longmapsto P(x)$, spans the dual of $\mathcal P_p(\Kcal_0)$ as $x$ ranges over $\Kcal_0$. Then we can choose $x_1,\ldots,x_q\in\Kcal_0$ such that
\[
\mathcal P_p(\Kcal_0)\longrightarrow\mathbb R^q,
\qquad P\longmapsto\big(P(x_1),\ldots,P(x_q)\big)
\] 
is an isomorphism. For each $j\in\{1,\ldots,q\}$, we take the preimage of the standard basis vector $e_j$ under this map and choose a homogeneous representation $L_j\in\mathcal P_p^{\mathrm{hom}}$. These are analogues of the Lagrange interpolation polynomials, as they satisfy $L_j(x_k)=\delta_{jk}$ for all $j,k$.

For every integer $N\geq1$, we have $\Kcal_0=\bigcup_{w\in J^N}f_w(\Kcal_0)$. We can therefore choose words $w_1,\ldots,w_q\in J^N$ such that $x_j\in f_{w_j}(\Kcal_0)$ for every $j$. We will show that $2L_j$ acts like a bump function: it is more than $1$ on the same-address copy $f_{w_j}(\Kcal_0)$, and exponentially small in $N$ on the other copies $f_{w_k}(\Kcal_0)$ for $k\ne j$.

By Lemma~\ref{lem: basic inequalities for rho and theta}, the Euclidean diameter of each $f_{w_j}(\Kcal_0)$ is at most $\kappa^N\diam_\rho(\Kcal_0)$. Since the derivatives of the fixed polynomials $L_j$ are bounded on $\Delta$, there is a constant $A\geq d+1$, independent of $N$, such that
\begin{equation} \label{eq: interpolation on small copies}
\norm{L_j}_{C(\mathbb D^{d+1})}\leq A,
\qquad
\abs{L_j(x)-\delta_{jk}}\leq A\kappa^N
\qquad (x\in f_{w_k}(\Kcal_0),\;1\leq j,k\leq q).
\end{equation}

Since $R^2\kappa^c<1$, we can fix $N$ large enough that $2q(2A)^c R^N \kappa^{cN/2}\leq1$. Put $B\coloneqq2q(2A)^cR^N$. Then $B^2\kappa^{cN}\leq1$, and $(2A\kappa^N)^c\leq B^2\kappa^{cN}\leq1$. Thus $A\kappa^N\leq1/2$ since $c\geq1$. Then equation~\eqref{eq: interpolation on small copies} gives
\begin{equation} \label{eq: 2L is a bump function}
2L_j(x)\geq1 \quad \big(x\in f_{w_j}(\Kcal_0)\big),
\qquad
\abs{2L_j(x)}\leq2A\kappa^N \quad \big(x\in f_{w_k}(\Kcal_0) \;\text{with}\; k\ne j\big).
\end{equation}

We now construct the spaces $\mathcal E_\ell$ by induction. Let $\mathcal E_0\coloneqq\Span_{\mathbb R}\{X_1+\cdots+X_{d+1}\}\subset\mathcal P_1^{\mathrm{hom}}$. Since $\sum_j X_j = 1$ on $\Kcal_0$, equation~\eqref{eq: growing polynomial spaces} holds for $\ell=0$. Suppose that $\mathcal E_\ell$ has been constructed and satisfies equation~\eqref{eq: growing polynomial spaces}. Define
\[
\mathcal E_{\ell+1}\coloneqq\left\{\sum_{j=1}^q(2L_j)^{c\Lambda^\ell}(P_j\circ M_{w_j}^{-1})\setcond P_1,\ldots,P_q\in \mathcal E_\ell\right\}.
\]
The composition by $M_{w_j}^{-1}$ compresses the polynomial $P_j$ onto the small copy $f_{w_j}(\Kcal_0)$, whereas the factor $(2L_j)^{c\Lambda^\ell}$ suppresses it on the other chosen copies $f_{w_k}(\Kcal_0)$ for $k\ne j$. This $\mathcal E_{\ell+1}$ is a linear subspace of $\mathcal P_{\Lambda^{\ell+1}}^{\mathrm{hom}}$, since each summand is homogeneous of degree $pc\Lambda^\ell+\Lambda^\ell=\Lambda^{\ell+1}$.

Pick any $Q = \sum_{j=1}^q(2L_j)^{c\Lambda^\ell}(P_j\circ M_{w_j}^{-1})\in \mathcal E_{\ell+1}$ where $P_1,\ldots,P_q\in \mathcal E_\ell$. Fix $j\in\{1,\ldots,q\}$. By equation~\eqref{eq: polynomial on a small copy} and the induction hypothesis,
\[
\norm{P_j\circ M_{w_j}^{-1}}_{C(\mathbb D^{d+1})}\leq R^{N\Lambda^\ell}\norm{P_j}_{C(\mathbb D^{d+1})}\leq R^{N\Lambda^\ell}B^{\Lambda^\ell}\norm{P_j}_{C(\Kcal_0)}.
\]
Together with $\norm{L_j}_{C(\mathbb D^{d+1})}\leq A$ and the definition of $B$, this gives
\begin{equation} \label{eq: complex norm of each polynomial term}
\norm{(2L_j)^{c\Lambda^\ell}(P_j\circ M_{w_j}^{-1})}_{C(\mathbb D^{d+1})}
\leq \big((2A)^cR^NB\big)^{\Lambda^\ell}\norm{P_j}_{C(\Kcal_0)}
= \left(\frac{B^2}{2q}\right)^{\Lambda^\ell}\norm{P_j}_{C(\Kcal_0)}.
\end{equation}

We next estimate the same polynomial on $f_{w_k}(\Kcal_0)$ for $k\ne j$. On this domain, equation~\eqref{eq: 2L is a bump function} gives $\norm{2L_j}_{C(f_{w_k}(\Kcal_0))}\leq2A\kappa^N$. Therefore, using $B^2\kappa^{cN}\leq1$,
\begin{equation} \label{eq: linear combination evaluated at suppressed regions}
\norm{(2L_j)^{c\Lambda^\ell}(P_j\circ M_{w_j}^{-1})}_{C\big(f_{w_k}(\Kcal_0)\big)}
\leq \left(\frac{B^2\kappa^{cN}}{2q}\right)^{\Lambda^\ell}\norm{P_j}_{C(\Kcal_0)}
\leq \frac{\norm{P_j}_{C(\Kcal_0)}}{(2q)^{\Lambda^\ell}}.
\end{equation}

On its own copy $f_{w_j}(\Kcal_0)$, the factor $(2L_j)^{c\Lambda^\ell}$ is at least one by equation~\eqref{eq: 2L is a bump function}. Since $0<\one^\top M_wy\leq1$ for $y\in\Delta$, equation~\eqref{eq: polynomial on a small copy} gives
\begin{equation} \label{eq: linear combination evaluated at bump}
\norm{(2L_j)^{c\Lambda^\ell}(P_j\circ M_{w_j}^{-1})}_{C\big(f_{w_j}(\Kcal_0)\big)}\geq\norm{P_j}_{C(\Kcal_0)}.
\end{equation}

Choose $j$ such that $\norm{P_j}_{C(\Kcal_0)}=\max_{1\leq k\leq q}\norm{P_k}_{C(\Kcal_0)}$. On $f_{w_j}(\Kcal_0)$, the $j$th summand of the equation defining $Q$ has norm at least $\norm{P_j}_{C(\Kcal_0)}$ by equation~\eqref{eq: linear combination evaluated at bump}, while each other summand has norm at most $\norm{P_j}_{C(\Kcal_0)}/(2q)^{\Lambda^\ell}$ by equation~\eqref{eq: linear combination evaluated at suppressed regions}. By the triangle inequality and $\Lambda^\ell\geq1$,
\begin{equation} \label{eq: no cancellation between polynomial copies}
\norm{Q}_{C(\Kcal_0)}\geq\norm{Q}_{C\big(f_{w_j}(\Kcal_0)\big)}
\geq\left(1-\frac{q-1}{(2q)^{\Lambda^\ell}}\right)\norm{P_j}_{C(\Kcal_0)}
\geq\frac12\max_{1\leq k\leq q}\norm{P_k}_{C(\Kcal_0)}.
\end{equation}
(This is Property Y.) Summing equation~\eqref{eq: complex norm of each polynomial term} and using equation~\eqref{eq: no cancellation between polynomial copies}, we obtain
\[
\norm{Q}_{C(\mathbb D^{d+1})}
\leq\left(\frac{B^2}{2q}\right)^{\Lambda^\ell}\sum_{k=1}^q\norm{P_k}_{C(\Kcal_0)}
\leq\frac{2q}{(2q)^{\Lambda^\ell}}B^{2\Lambda^\ell}\norm{Q}_{C(\Kcal_0)}
\leq B^{\Lambda^{\ell+1}}\norm{Q}_{C(\Kcal_0)},
\]
where the last inequality follows from $\Lambda^\ell\geq1$, $B\geq1$, and $\Lambda\geq2$.

Finally, if $Q=0$, equation~\eqref{eq: no cancellation between polynomial copies} implies that every $P_k$ vanishes on $\Kcal_0$. The induction hypothesis then gives $\norm{P_k}_{C(\mathbb D^{d+1})}=0$, so every $P_k$ is the zero polynomial. This shows that each $Q\in \mathcal E_{\ell+1}$ has a unique representation. Hence $\dim \mathcal E_{\ell+1}=q\dim \mathcal E_\ell=q^{\ell+1}$. This completes the induction.
\end{proof}

\subsection{Completion of the proofs of the main theorems}

\begin{proof}[Proof of Theorem~\ref{thm: invertible subfamily lower bounds constellation dimension}]
We have $\Kcal_0\subset\Kcal$. If $\dim_\mathbb C\overline{\Kcal_0}^{\mathrm{Zar}}=0$, the assertion is immediate. Otherwise, fix $0<s<\dim_\mathbb C\overline{\Kcal_0}^{\mathrm{Zar}}$ and let $\mathcal E_\ell$, $\ell\geq0$, be the spaces supplied by Lemma~\ref{lem: enlarging polynomial spaces}. Let
\[
\widetilde U_\ell \coloneqq \left\{ u: \Delta^\circ \to \mathbb R \setcond u = \omega_{\Lambda^\ell}P,\; P\in \mathcal E_\ell \right\},
\qquad
U_\ell \coloneqq \left\{ u\in\widetilde U_\ell \setcond \nu(u)=0 \right\}.
\]
Then $\widetilde U_\ell\subset\Hscr_a$ by Lemma~\ref{lem: weighted polynomials}. Multiplication by $\omega_{\Lambda^\ell}$ is injective on $\mathcal E_\ell$: if $\omega_{\Lambda^\ell}P=0$, then $P=0$ on $\Kcal_0$, and equation~\eqref{eq: growing polynomial spaces} implies that $P$ is $0$ as a polynomial. Thus $\dim\widetilde U_\ell=q^\ell$. The map $\widetilde U_\ell\to\mathbb R$, $u\mapsto\nu(u)$, is linear, so $\dim U_\ell\geq q^\ell-1$.

For any $u = \omega_{\Lambda^\ell}P\in U_\ell$, Lemmas~\ref{lem: weighted polynomials} and~\ref{lem: enlarging polynomial spaces} give
\begin{equation} \label{eq: bounding the sup norm from below by Hilbert norm}
\norm{u}_{\Hscr_a}
\leq B^{\Lambda^\ell}\norm{P}_{C(\Kcal_0)}
\leq B^{\Lambda^\ell}e^{\Lambda^\ell}\norm{u}_{C(\Kcal_0)}
\leq (eB)^{\Lambda^\ell}\norm{u}_{C(\Kcal)}.
\end{equation}
Since $\nu(u)=0$, we can apply Lemma~\ref{lem: Holder-type interpolation inequality} to $u$. Using equation~\eqref{eq: bounding the sup norm from below by Hilbert norm} and rearranging gives
\[
\ip{u}{\Cscr_a u}_{\Hscr_a}
= \norm{u-\Lcal u}_{L^2(\nu)}^2
\geq e^{-C\Lambda^\ell}\norm{u}_{\Hscr_a}^2
\]
for a constant $C\geq1$ independent of $\ell$. Thus using $U_\ell$ in the Min--Max principle (Theorem~\ref{thm: min--max principle}) yields
\[
\Ncal_a(C\Lambda^\ell)= 1+\#\left\{j\in \mathbb N \setcond \sigma_j\geq e^{-C\Lambda^\ell}\right\}
\geq q^\ell>\Lambda^{s\ell}.
\qquad (\ell\geq1)
\]
We conclude that $\ld\geq s$, and the assertion follows because $0<s<\dim_\mathbb C\overline{\Kcal_0}^{\mathrm{Zar}}$ was arbitrary.
\end{proof}

\begin{proof}[Proof of Theorem~\ref{thm: intrinsic complexity}]
Equation~\eqref{eq: optimality of the construction} was proved in \S\S~\ref{subsection: exponents for minimal model size}. It remains to prove equation~\eqref{eq: work count expressed through constellation dimension} under the assumption that there are two invertible matrices with distinct Perron-Frobenius vectors.

First, the attractor of the two invertible matrices is infinite. Otherwise, their injective projective maps would permute this finite attractor, and contraction would force it to be a singleton, contradicting their distinct Perron--Frobenius vectors. Its Zariski closure therefore has dimension at least one, so Theorem~\ref{thm: invertible subfamily lower bounds constellation dimension} gives $\ld\geq1$.

Put $\eps'=\eps/4$, and let $n(\eps')$ be the stopping index defined in Lemma~\ref{lem: spectral upper bound for our construction}. The clock inequality (Theorem~\ref{thm: Clock inequality}) then gives $\abs{\eta_{n(\eps')}(\Phi)-\lambda}<\eps$. The operation count in \S\S~\ref{subsection: proof of universal poly complexity}, including the $O(n(\eps')+1)$ operations for evaluating $\eta_{n(\eps')}(\Phi)$, gives
\[
c\big(n(\eps')+1\big)^3\leq W(\eps)\leq C\left(\big(n(\eps')+1\big)^3+\big(n(\eps')+1\big)^2\log(1/\eps)\right)
\]
for the dense incremental factorization used there. By equation~\eqref{eq:minimum-sandwich}, the ratio $\log\big(n(\eps')+1\big)/\log\log(1/\eps)$ has lower and upper limits $\ld$ and $\ud$. Since $\ld\geq1$, the additional term $\big(n(\eps')+1\big)^2\log(1/\eps)$ does not change these exponents: the displayed bounds give $\log W(\eps)=3\log\big(n(\eps')+1\big)+o\big(\log\log(1/\eps)\big)$. This proves equation~\eqref{eq: work count expressed through constellation dimension}.
\end{proof}

\appendix

\section{Reduction of the random Fibonacci sequence} \label{appendix: Reduction of the random Fibonacci sequence}

\subsection{Background}

Here we will prove the reduction stated in Example~\ref{ex: the random Fibonacci sequence}. The starting point is the cancellation discovered by Janvresse--Rittaud--de la Rue \cite{Janvresse--Rittaud--deLaRue}: $B_+B_-B_-B_+=-B_-$, $B_+B_-B_-B_-=-B_+$. These then yield
\[
B_-^3=-I,
\qquad
(B_+B_-^2)^2=I.
\]
Thus, ignoring the signs, we have a generator of order three, represented by $B_-$, and a generator of order two, represented by $B_+B_-^2$. The original random choices $B_-$ and $B_+$ are then represented by the steps $b$ and $ab$ in the free product
\[
\mathbb Z/2\mathbb Z*\mathbb Z/3\mathbb Z
\;=\;\langle a,b\mid a^2=b^3=1\rangle.
\]
This led us to consider the work by Bordenave--Dubail \cite{Bordenave--Dubail} on traffic equations for random walks on free products of finite groups. In this case the traffic equation can be solved explicitly, and that is how we initially obtained the i.i.d.\ representation using the four matrices. The direct proof below was obtained after seeing the answer.

\subsection{Proof}

Let $q\coloneqq(\sqrt5-1)/2$, so $q+q^2=1$. Consider independent choices of $B_+$ and $B_+B_-$ with probabilities $q$ and $q^2$, and call this stochastic process $\mathcal Y$. Its two-step matrices are
\[
(B_+B_-)^2=A_0,
\qquad (B_+B_-)B_+=A_1,
\qquad B_+(B_+B_-)=A_2,
\qquad B_+^2=A_3,
\]
with probabilities $q^4,q^3,q^3,q^2$, respectively. These are exactly $\pi_0,\pi_1,\pi_2,\pi_3$, so the Lyapunov exponent for $\mathcal Y$ is $\lambda_{\mathrm{red}}/2$. We will see that any stationary measure for $\mathcal Y$ can be used to construct one for the original family. The construction uses the reflection
\[
D\coloneqq B_+B_-B_-=\begin{pmatrix}1&0\\0&-1\end{pmatrix},
\qquad DB_+=-B_-,
\qquad DB_-=-B_+.
\]

For an invertible matrix $B$, write $f_B$ for its projective map:
\[
f_B: \mathbb{RP}^1\longrightarrow\mathbb{RP}^1,
\qquad
[x] \longmapsto [xB],
\]
where $\mathbb{RP}^1$ is the real projective space and $[x]$ denotes the equivalence class of the nonzero row vector $x$. The variational principle in equation~\eqref{eq: variational principle for stationary measures} applies to both families after replacing $\Delta$ with $\mathbb{RP}^1$. Since neither family preserves a proper subspace of $\mathbb{RP}^1$, any stationary measure for each family attains the supremum \cite[p.27, Corollary]{Furstenberg--Kifer}. Since the nonnegative cone $\mathcal V$ is preserved by both $f_{B_+}$ and $f_{B_+B_-}$, there is a stationary measure $\mu$ for $\mathcal Y$ supported on $\mathcal V$; see Furstenberg--Kifer \cite[Lemma~1.3, and p.16]{Furstenberg--Kifer}. By stationarity and $f_{B_-}\circ f_B=f_{BB_-}$,
\[
\mu=q(f_{B_+})_*\mu+q^2(f_{B_+B_-})_*\mu,
\qquad
(f_{B_-})_*\mu=q(f_{B_+B_-})_*\mu+q^2(f_D)_*\mu.
\]
Combining the two identities, we obtain
\[
\nu\coloneqq\frac{(f_{B_+})_*\mu+(f_{B_-})_*\mu}{2}
=\frac1{2q}\mu+\frac{q^2}{2}(f_D)_*\mu.
\]
The coefficients on the right sum to $1$. Writing $\Lcal$ for the transfer operator of the original family, the above translates to $\Lcal^*\mu = \nu$. Since $DB_+=-B_-$ and $DB_-=-B_+$, we also have $\Lcal^*\big((f_D)_*\mu\big) = \nu$. This shows that $\nu$ is a stationary measure for the original family.

We now compute the logarithmic averages for $B_+$ and $B_-$. Write $\norm{(u,v)}_\infty \coloneqq \max\{\abs{u},\abs{v}\}$. For every nonzero nonnegative row $x=(u,v)$,
\begin{equation} \label{eq: the modified family of RF has same log increment}
\norm{xB_+}_\infty=\norm{xB_+B_-}_\infty=u+v.
\end{equation}
The two choices in $\mathcal Y$ therefore give the same logarithmic growth. Then by stationarity of $\mu$,
\[
\int_{\mathbb{RP}^1}\log\frac{\norm{xB_+}_\infty}{\norm{x}_\infty}\,d\mu([x])
=
q\int_{\mathbb{RP}^1}\log\frac{\norm{xB_+}_\infty}{\norm{x}_\infty}\,d\mu([x])
+q^2\int_{\mathbb{RP}^1}\log\frac{\norm{xB_+B_-}_\infty}{\norm{x}_\infty}\,d\mu([x])
=\frac{\lambda_{\mathrm{red}}}{2}.
\]

On the other hand, stationarity of $\mu$ and $B_+B_-B_- = D$ give
\[
\int_{\mathbb{RP}^1}\log\frac{\norm{xB_-}_\infty}{\norm{x}_\infty}\,d\mu([x])
=
q\int_{\mathbb{RP}^1}\log\frac{\norm{xB_+B_-}_\infty}{\norm{xB_+}_\infty}\,d\mu([x])
+q^2\int_{\mathbb{RP}^1}\log\frac{\norm{xD}_\infty}{\norm{xB_+B_-}_\infty}\,d\mu([x]).
\]
The first integral is $0$ by equation~\eqref{eq: the modified family of RF has same log increment}. Thus, using $\norm{xD}_\infty = \norm{x}_\infty$ and equation~\eqref{eq: the modified family of RF has same log increment},
\[
\int_{\mathbb{RP}^1}\log\frac{\norm{xB_-}_\infty}{\norm{x}_\infty}\,d\mu([x])
=
q^2 \int_{\mathbb{RP}^1}\log\frac{\norm{x}_\infty}{\norm{xB_+}_\infty}\,d\mu([x])
=
-q^2\frac{\lambda_{\mathrm{red}}}{2}.
\]

Finally, by $\norm{xD}_\infty = \norm{x}_\infty$, $DB_+=-B_-$, and $DB_-=-B_+$, replacing $x$ by $xD$ interchanges the two logarithms for $B_+$ and $B_-$. Their sum therefore has the same integral against $\mu$ and $(f_D)_*\mu$, and hence against $\nu$. The stationary-measure formula for $\nu$ now gives
\[
\lambda_{\mathrm{orig}}
=\frac12\left(\frac{\lambda_{\mathrm{red}}}{2}-q^2\frac{\lambda_{\mathrm{red}}}{2}\right)
=\frac q4\lambda_{\mathrm{red}}
=\frac{\sqrt5-1}{8}\lambda_{\mathrm{red}}.
\]

\section{Constellation dimension for the Sierpi\'nski gasket family} \label{appendix: Constellation dimension for Sierpi\'nski gasket family}

Let us explain the proof outline for $\lud = 4$ for the matrix family in Example~\ref{ex: Sierpinski gaskets}. For this family, the constellation dimension can be explained with algebraic relations.

First, the attractor contains countably many copies of a fractal with algebraic dimension $3$. Indeed, the two products $(A_0A_3)^3$, $(A_0A_3)^3A_0$ preserve $A_0(\mathbb R^8)$, which has vector dimension $4$. In a suitable basis, their restrictions are positive and invertible. An elementary matrix calculation shows that they generate a Zariski-dense group in $\mathrm{GL}_4(\mathbb R)$ under that basis. Consequently, their attractor, denoted by $\Kcal_0$, has algebraic dimension $3$. The additional dimension comes from repeatedly applying $A_2$. Direct calculation gives
\[
f_{A_2}^{\,k}(\Kcal_0)
\subset
\left\{x\in\Delta\setcond x_7-x_3=kx_2\right\},
\qquad k=1,2,\ldots.
\]
Since $x_2>0$, these copies lie on distinct hyperplanes. (The quantity $(x_7-x_3)/x_2$ records which copy a point belongs to.) This lets us prove $\ld\geq4$ by mimicking \S\S\ref{subsection: preventing cancellation}.

On the other hand, we can show that the same attractor lies in a finite union of affine spaces of dimension at most $4$. Indeed, the identity
\[
A_1^n=A_1+\frac{3^n-3}{6}(A_1^2-A_1),\qquad n\geq1,
\]
and a rank calculation show that the images $A_1^nA_i(\mathbb R^8)$, for $i\in\{0,2,3\}$ and $n\geq1$, lie in three fixed five-dimensional vector spaces. Together with $A_0(\mathbb R^8),A_2(\mathbb R^8),A_3(\mathbb R^8)$, whose dimensions are at most $5$, these give six spaces whose intersections with $\Delta$ contain the attractor. Together with an estimate near the boundary of $\Delta$, this lets us prove $\ud\leq4$ by adapting the polynomial approximation argument in Proposition~\ref{prop: constellation dimension universal bound}.

\section{Liouville reset example} \label{appendix: Liouville reset}

Let us prove the constellation-dimension assertions for Example~\ref{ex: Liouville reset}. With the parametrization $x(r)=\frac12(1+r,1-r)^\top$, we identify $\Delta\cong[-1,1]$. The projective maps $f_{A_0}$, $f_{A_1}$, $f_{A_2}$ send $r$ to $1/2$, $r/2$, $2^{-\alpha}r$, respectively. The application of the singular $A_0$ serves as a ``reset'': we have
\[
\Kcal=\{a_{j,k}:j,k\geq0\}\cup\{0\},
\qquad a_{j,k}\coloneqq\frac12\,2^{-j-\alpha k},
\]
where $j$ and $k$ are the number of times $A_1$ and $A_2$ are applied after a reset, respectively.

The intuition is the following: at some resolutions, $\alpha$ admits a rational approximation that is extraordinarily accurate. In this case many points of $\Kcal$ are nearly identical, making the approximation of stationarity very cheap. At other scales, the two independent counts $j$ and $k$ remain distinguishable, and more information is needed.

\subsection{Lower bounds}

\begin{claim}
There is $C>1$ such that the following holds. If for some $h\geq1$ and $m\in\mathbb N$, there are orbit points $r_1,\ldots,r_m\in[0,1/4]$ whose indices satisfy $j+k\leq h$ and whose mutual distances are at least $e^{-h}$, then $\sigma_m\geq e^{-2Cmh}$.
\end{claim}

\begin{proof}
By Corollary~\ref{cor: reference independence of constellation dimension}, we may set $a=1/2$ as the reference. Recall $T=I-\Lcal$. If $u\in\Hscr_a$ satisfies $u(a)=0$, the Cauchy--Schwarz inequality and stationarity of $\nu$ imply
\begin{equation} \label{eq: Liouville reset proof LT bounds}
\norm{\Lcal u}_{L^2(\nu)}^2\leq\frac23\norm u_{L^2(\nu)}^2, \qquad \norm{Tu}_{L^2(\nu)}\geq\bigl(1-\sqrt{2/3}\bigr)\norm u_{L^2(\nu)}.
\end{equation}

The Cayley transformation is $\chi_1\big(x(r)\big) =\dfrac{2r-1}{2r+5} \eqqcolon \xi(r)$. Its powers $g_j(r)\coloneqq\xi(r)^j$, $j\geq1$, vanish at $a$ and satisfy $\norm{g_j}_{\Hscr_a}\leq1$ by the Pull-back theorem~\ref{thm: pull-back theorem}.

Write a nonzero $u\in\Span\{g_1,\ldots,g_m\}$ as $u=\xi P(\xi)$, where $P$ is a polynomial of degree less than $m$. Let $\beta_1,\ldots,\beta_s\in\mathbb C$ be the roots of $P$. Since $\xi(r_1),\ldots,\xi(r_m)$ are separated by $ce^{-h}$ for some $c>0$, each $\beta_j$ can be closer than $ce^{-h}/3$ to at most one of $\xi(r_1),\ldots,\xi(r_m)$. Therefore, there is $\xi(r)$ such that it is further than $ce^{-h}/4$ from every $\beta_j$. For $\abs z=1$, since $\abs{z-\xi(r)}\leq2$,
\[
\frac{\abs{P(z)}}{\abs{P(\xi(r))}}
=\prod_{\ell=1}^s\frac{\abs{z-\beta_\ell}}{\abs{\xi(r)-\beta_\ell}}
\leq\prod_{\ell=1}^s\left( \frac{\abs{z-\xi(r)} + \abs{\xi(r)-\beta_\ell}}{\abs{\xi(r)-\beta_\ell}} \right)
\leq\left(\frac{8e^h}{c}+1\right)^s,
\]
Since $\abs \xi \geq 1/11$ on $[0,1/4]$, the Pull-back theorem gives
\begin{equation} \label{eq: proof of Liouville reset; bounding the Hilbert norm of u}
\norm{u}_{\Hscr_a}\leq\sup_{\abs z=1}\abs{P(z)}\leq11\left(1+\frac{8e^h}{c}\right)^s\abs{u(r)}\leq e^{C'mh}\abs{u(r)}.
\qquad (C'>0)
\end{equation}
Each $r_i = a_{j,k}$ satisfies $\nu(\{r_i\})\geq3^{-h-1}$, since a reset followed by $j+k$ maps reaches it. Thus we have $\norm{u}_{L^2(\nu)}^2\geq3^{-h-1}\abs{u(r)}^2$. Combining this with equations~\eqref{eq: Liouville reset proof LT bounds} and \eqref{eq: proof of Liouville reset; bounding the Hilbert norm of u} gives
\[
\norm{u}_{\Hscr_a}
\leq e^{Cmh}\norm{Tu}_{L^2(\nu)},
\]
where $C>0$ is independent of $m,h$. As $u\in\Span\{g_1,\ldots,g_m\}$ was arbitrary, the Min--Max principle (Theorem~\ref{thm: min--max principle}) implies the assertion.
\end{proof}

For any $m\in\mathbb N$ and the orbit points $a_{j,0}$, $1\leq j\leq m$, take $h=2m$ and apply the claim: their indices satisfy $j+k\leq m$, and their mutual distances are at least $2^{-m-1}$. Hence $\sigma_m\geq e^{-4Cm^2}$ for every $m$, proving $\ld\geq1/2$.

\medskip

Let us prove $\ud\geq2/3$. For $n\geq1$, the rational number $p_n/q_n=1+\sum_{k\leq n}q_k^{-1}$ is reduced, and
\[
 0<\delta_n\coloneqq\alpha - \dfrac{p_n}{q_n}\leq2^{-2^{q_n}+1}.
\]
Take all $q_n^2$ orbit points $a_{j,k}$ with $1\leq j,k\leq q_n$. For each $(j,k)$, write $j+kp_n/q_n = b_{j,k}/q_n$. Equality $b_{j,k}  =b_{j',k'}$ implies $q_n\mid(k-k')$, and as $\abs{k-k'}<q_n$, this is possible only when $(j,k)=(j',k')$. Thus the indices $j+kp_n/q_n$ are separated by $1/q_n$. Replacing $p_n/q_n$ by $\alpha$ changes their differences by at most $q_n\delta_n<1/(2q_n)$ for large enough $n$, so the indices $j+k\alpha$ remain separated by $1/(2q_n)$.

The map $t\mapsto\frac12\,2^{-t}$ sends $j+k\alpha$ to the orbit point $a_{j,k}$. Since $j+k\alpha\leq(1+\alpha)q_n$ by $j,k\leq q_n$, the mean-value theorem converts the separation $1/(2q_n)$ between the indices into
\[
\abs{a_{j,k}-a_{j',k'}}\geq\frac{\log2}{2}\,2^{-(1+\alpha)q_n}\cdot\frac1{2q_n}=\frac{\log2}{4q_n}\,2^{-(1+\alpha)q_n}\geq e^{-3q_n}
\]
for distinct pairs $(j,k)$, $(j',k')$ and large enough $n$. Since $j+k\leq2q_n$, we may now apply the claim with $m=q_n^2$ and $h=3q_n$, obtaining $\sigma_{q_n^2}\geq e^{-6Cq_n^3}$. This proves $\ud\geq2/3$.

\subsection{Upper bounds}

\begin{claim}
There is a constant $C>0$ such that, for every $M\in\mathbb N$ and distinct $b_1,\ldots,b_M\in[0,1/2]$, the polynomial $P_M(r)=\prod_{i=1}^M(r-b_i)$ satisfies $\sqrt{\sigma_{M+1}}\leq2C6^M\norm{P_M}_{C(\Kcal)}$.
\end{claim}

\begin{proof}
Every fixed $(M+1)$-dimensional subspace of $\Hscr_a$ contains $u\ne0$ satisfying the $M$ conditions $u(b_i)=0$. Take such $u$ with $\norm{u}_{\Hscr_a}=1$. The Pull-back theorem~\ref{thm: pull-back theorem} supplies a function $F\in H^2(\mathbb D^2)$ such that $u(r)=F\circ\chi\big(x(r)\big)$ and $\norm{F}_{H^2(\mathbb D^2)}=1$. This composition lets us extend $u$ holomorphically to $\abs r<1$. By the reproducing property of the Hardy kernel, we have $\abs{u}\leq C$ on $\abs r\leq2/3$, where $C$ is independent of $u$. Since $u(b_i)=0$, the quotient $u/P_M$ is holomorphic on $\abs r<1$. On $\abs r=2/3$, we have $\abs{r-b_i}\geq1/6$, so $\abs{u/P_M}\leq C6^M$. The maximum principle therefore gives
\[
\abs{u(r)}\leq C6^M\abs{P_M(r)}.
\qquad (0\leq r\leq1/2)
\]
Since $\norm{Tu}_{L^2(\nu)}\leq2\norm{u}_{C(\Kcal)}$, the Min--Max principle (Theorem~\ref{thm: min--max principle}) implies the assertion.
\end{proof}

To prove $\ud\leq2/3$, take any integer $h>0$ and set all $a_{j,k}$ with $j+k\leq h$ as $b_1,\ldots,b_M$. Since $\alpha$ is irrational, these points are distinct. If $r\in\Kcal\setminus\{b_1,\ldots,b_M\}$, we have $r\leq2^{-h-1}$. Both $r$ and $a_{j,k}$ are therefore at most $2^{-j-k-1}$ for every $j+k\leq h$. Hence
\[
\abs{P_M(r)}\leq\prod_{j+k\leq h}2^{-j-k-1}
=2^{-(h+1)(h+2)(2h+3)/6}.
\]
Since $M=(h+1)(h+2)/2$, the claim yields some $c'>0$ independent of $h$ with $\sigma_{M+1}\leq e^{-c'h^3}$. As this holds for every sufficiently large integer $h$, we obtain $\ud\leq2/3$.

\medskip

We are left to prove $\ld\leq1/2$. Set $M=2^{q_n/2}$ and take $b_\ell= \dfrac12\cdot2^{-(\ell-1)/q_n}$, $1\leq\ell\leq M$. Then
\[
\left|P_M\left(\frac12\,2^{-j-(p_n/q_n)k}\right)\right|\leq\prod_{i=1}^M b_i.
\qquad (j,k\geq0)
\]
The mean-value theorem gives $\left|2^{-j-\alpha k}-2^{-j-(p_n/q_n)k}\right|\leq\delta_n$. Since $\abs{P'_M}\leq1$ on $[0,1/2]$, 
\[
\norm{P_M}_{C(\Kcal)}
=\sup_{j,k\geq0}\abs{P_M\left(\frac12\,2^{-j-\alpha k}\right)}
\leq\delta_n+\prod_{i=1}^M b_i.
\]
Now, $\delta_n\leq2^{1-M^2}$ and $\prod_i b_i=2^{-M-M(M-1)/(2q_n)}$. The claim consequently gives
\[
\sqrt{\sigma_{M+1}}\leq2C6^M\left(\delta_n+\prod_i b_i\right)\leq e^{-c_0M^2/q_n}
\]
for some $c_0>0$ and all large enough $n$. Hence $\Ncal_a(c_0M^2/q_n)\leq M+1$. This proves $\ld\leq1/2$.

\section{Fredholm determinants and constellations} \label{appendix: Fredholm determinants and constellations}

Here we explain how the Fredholm determinant approximation of \cite{Pollicott10} is actually a construction of constellations. Assume that $A_0,\ldots,A_{m-1}$ are positive invertible matrices.

\subsection{The Fredholm determinant and its truncation} \label{subsec: The Fredholm determinant and its truncation}

The analytic construction in \cite[Sections 3--4]{Pollicott10} provides a bounded connected complex neighborhood $U\subset H_{\mathbb C}$ of $\Kcal$ on which the maps $\Phi, f_i$ have bounded holomorphic extensions $\Phi:U\to \mathbb C$, $f_i:U\to U$ with $f_i(U)\Subset U$, and $\overline U\subset\Omega$, where $\Omega$ is the domain in \S\S\ref{subsec: Connection to the Hardy kernel}. Let $\mathscr A$ be the Banach space of bounded holomorphic functions $u:U\to\mathbb{C}$, equipped with the supremum norm. Fix $g\in\mathscr A$ with $g(\Kcal)\subset\mathbb R$, and let
\[
\Lcal_{t,g}:\mathscr A\longrightarrow\mathscr A,
\qquad
\Lcal_{t,g}u\coloneqq\Lcal(e^{tg}u)
=\sum_{i=0}^{m-1}\pi_i e^{tg\circ f_i(x)} u\circ f_i(x)
\]
Here $t$ belongs to a complex neighborhood $V$ of zero. Since $\Lcal$ is nuclear of order zero \cite[Lemma~3.1(2)]{Pollicott10} and multiplication by $e^{tg}$ is bounded on $\mathscr A$, the operator $\Lcal_{t,g}$ is nuclear of order zero. Therefore its Fredholm determinant $\mathfrak D_g(z,t)$ is defined:
\begin{equation*} %\label{eq: Fredholm determinant}
\mathfrak D_g:\mathbb{C}\times V\longrightarrow\mathbb{C},
\qquad
\mathfrak D_g(z,t)\coloneqq\det(I-z\Lcal_{t,g})
=\sum_{j=0}^{\infty}c_j(t;g)z^j,
\end{equation*}
where $c_j(\,\cdot\,;g):V\to\mathbb{C}$ are analytic and $c_0(t;g)=1$. Recall the trace identity
\begin{equation} \label{eq: Fredholm determinant trace identity}
\mathfrak D_g(z,t)
=\exp\left(-\sum_{\ell=1}^{\infty}\frac{z^\ell}{\ell}\operatorname{tr}(\Lcal_{t,g}^{\,\ell})\right),
\end{equation}
which holds for $z$ near zero. For $N\geq1$, denote by $\mathfrak D_{N,g}$ the truncation of $\mathfrak D_g(z,t)$ up to degree $N$:
\begin{equation*} %\label{eq: truncated Fredholm determinant}
\mathfrak D_{N,g}:\mathbb{C}\times V\longrightarrow\mathbb{C},
\qquad
\mathfrak D_{N,g}(z,t)\coloneqq\sum_{j=0}^{N}c_j(t;g)z^j.
\end{equation*}
Since $\Lcal_{0,g}=\Lcal$, the number $c_j(0;g)$ is independent of $g$; let us write $c_j\coloneqq c_j(0;g)$.

By \cite[Proposition~6.1]{Pollicott10},
\begin{equation} \label{eq: Fredholm stationary integral}
\nu(g)=\frac{\partial_t\mathfrak D_g(1,0)}{\partial_z\mathfrak D_g(1,0)}.
\end{equation}
Replacing the Fredholm determinants by their truncation gives the functional approximating $\nu(g)$:
\begin{equation} \label{eq: Fredholm Lyapunov approximation}
\mathcal I_N(g)
\coloneqq \frac{\partial_t\mathfrak D_{N,g}(1,0)}{\partial_z\mathfrak D_{N,g}(1,0)}
=\frac{\sum_{j=1}^{N}\partial_t c_j(0;g)}{\sum_{j=1}^{N}j c_j},
\end{equation}
whenever the denominator is nonzero.

Then $\widehat\lambda_N\coloneqq \mathcal I_N(\Phi)$ is the proposed approximation to $\lambda = \nu(\Phi)$. Now how can we compute the right-hand side? For each $\ell\geq1$, let $\sigma:\{0,\ldots,m-1\}^{\ell}\to\{0,\ldots,m-1\}^{\ell}$ be the cyclic shift $\sigma(w_1,\ldots,w_\ell)=(w_2,\ldots,w_\ell,w_1)$. For a word $w$ of length $\ell$, let $x_w\in\Kcal$ be the unique fixed point of $f_w$. By equation~\eqref{eq: Fredholm determinant formula for matrix}, $\det\big(I-Df_w(x_w)\big)$ is real and nonzero because the largest eigenvalue of $A_w$ is simple. Then the fixed-point trace formula obtained in \cite[Lemma~4.1(1)]{Pollicott10} is
\begin{equation} \label{eq: Fredholm fixed-point trace formula}
\operatorname{tr}(\Lcal_{t,g}^{\,\ell})=\sum_{\abs{w}=\ell}\frac{\pi_w\exp\left(t\sum_{r=0}^{\ell-1}g(x_{\sigma^r w})\right)}{\det\big(I-Df_w(x_w)\big)}.
\end{equation}
This gives, by equation~\eqref{eq: Fredholm determinant trace identity},
\begin{equation*}
\left.\partial_t\operatorname{tr}(\Lcal_{t,g}^{\,\ell})\right|_{t=0}\!=\ell\sum_{\abs{w}=\ell}\!\frac{\pi_wg(x_w)}{\det\big(I\!-\!Df_w(x_w)\big)},
\quad
\sum_{j=1}^{N}\partial_t c_j(0;g)=-\sum_{\ell=1}^{N}\sum_{\abs{w}=\ell}\frac{\pi_wg(x_w)}{\det\big(I\!-\!Df_w(x_w)\big)}\sum_{j=0}^{N-\ell}c_j.
\end{equation*}
Using these identities, $\mathcal I_N(g)$ takes the following explicit form.
\begin{equation} \label{eq: Fredholm determinant method functional explicit}
\mathcal I_N(g)
\;=\; \sum_{\ell=1}^{N}\sum_{\abs{w}=\ell}
\underbrace{ \left(
-\frac{1}{\sum_{j=1}^{N}j c_j} \cdot \frac{\pi_w}{\det\big(I-Df_w(x_w)\big)}
\sum_{j=0}^{N-|w|}c_j \right)}_{\eqqcolon \omega_{N,w}}
g(x_w).
\end{equation}
The coefficients $c_j$ can be calculated using the following recursion, obtained from equation~\eqref{eq: Fredholm determinant trace identity}:
\begin{equation} \label{eq: recursion for Fredholm determinant coefficients}
c_0=1,\qquad
c_j=-\frac1j\sum_{\ell=1}^{j}c_{j-\ell}
\sum_{\abs{w}=\ell}\frac{\pi_w}{\det\big(I-Df_w(x_w)\big)},
\qquad 1\leq j\leq N,
\end{equation}

This is how \cite{Pollicott10} obtains an explicit approximation $\widehat\lambda_N$ to $\lambda$. In fact, he proves the following.
%(Pollicott performs the calculation for the individual logarithmic observables. Since the numerator in equation~\eqref{eq: Fredholm Lyapunov approximation} is linear in $g$, the two calculations give the same number.)
\begin{theorem}{\cite[Theorem 0.1]{Pollicott10}} \label{thm: Pollicott superexponential decay}
There are $C,\gamma>0$ and $N_{\mathrm{FD}}\in\mathbb N$ such that, for all $N\geq N_{\mathrm{FD}}$, the denominator in equation~\eqref{eq: Fredholm Lyapunov approximation} is nonzero and
\begin{equation*}
\abs{\lambda-\widehat\lambda_N}\leq C\exp\left(-\gamma N^{1+1/d}\right).
\end{equation*}
\end{theorem}

\subsection{The resulting constellation exhibits decay in stationary defect} \label{subsec: The resulting constellation exhibits decay in stationary defect}

\begin{lemma} \label{lem: Fredholm determinant truncated functional is a constellation}
Let $N\geq N_{\mathrm{FD}}$. For each word $w$ with $1\leq\abs{w}\leq N$, the number $\omega_{N,w}$ is real and $\sum_{\ell=1}^{N}\sum_{\abs{w}=\ell}\omega_{N,w}=1$.
\end{lemma}

\begin{proof}
At $t=0$, equation~\eqref{eq: Fredholm fixed-point trace formula} shows that $\operatorname{tr}(\Lcal^\ell)$ is real. Equation~\eqref{eq: Fredholm determinant trace identity} therefore gives $c_j\in\mathbb R$ for every $j$, and hence $\omega_{N,w}\in\mathbb R$. Finally, taking $g=1$ in equation~\eqref{eq: Fredholm determinant method functional explicit} and using equation~\eqref{eq: recursion for Fredholm determinant coefficients} give $\sum_{\ell=1}^{N}\sum_{\abs{w}=\ell}\omega_{N,w}=\frac{\sum_{j=1}^{N}jc_j}{\sum_{j=1}^{N}jc_j}=1$.
\end{proof}

Equation~\eqref{eq: Fredholm determinant method functional explicit} shows that $\mathcal I_N(g)$ is the integral of $g$ against the signed measure
\begin{equation*} %\label{eq: Fredholm constellation}
\eta_N^{\mathrm F} \;\coloneqq\; \sum_{\ell=1}^{N}\sum_{\abs{w}=\ell}\omega_{N,w}\delta_{x_w}.
\end{equation*}
Since $\eta_N^{\mathrm F}$ has mass one by Lemma~\ref{lem: Fredholm determinant truncated functional is a constellation}, it is a constellation satisfying
\[
\widehat\lambda_N=\eta_N^{\mathrm F}(\Phi).
\]
%The weights can have either sign, since the partial sums of determinant coefficients in equation~\eqref{eq: Fredholm quadrature weights} need not have a common sign. Also, note that $\Hscr_a$ contains Cayley coordinates, which are dense in $C(\Kcal)$. Since $\eta_N^{\mathrm F}$ has uniformly bounded total variation, it converges weakly to $\nu$. On the other hand, our $\eta_n$ may not have bounded total variation.

\begin{proposition} \label{prop: Fredholm constellation has stationary defect decay}
For every $a\in\Delta^\circ$, there are $C,\gamma>0$ such that, for all $N\geq N_{\mathrm{FD}}$,
\begin{equation*}
\norm{\eta_N^{\mathrm F}-\Lcal^*\eta_N^{\mathrm F}}_{K_a}\leq C\exp\left(-\gamma N^{1+1/d}\right).
\end{equation*}
\end{proposition}

We use the coefficient estimate for nuclear operators by Bandtlow--Jenkinson \cite{Bandtlow--Jenkinson}. Let $A^2(U)$ be the Bergman space, the set of holomorphic functions $f: U\to \mathbb C$ satisfying $\int_U\abs{f(z)}^2\,dV(z)<\infty$ for the Lebesgue measure $V$ on $\mathbb C^d$.

\begin{theorem}{\cite[Theorem~5.9 and 6.1.]{Bandtlow--Jenkinson}} \label{thm: Fredholm coefficient estimate}
Let $U\subset\mathbb C^d$ be a bounded domain, and let $f_i:U\to U$, $i=0,\ldots,m-1$, be holomorphic maps with $f_i(U)\Subset U$. For every $B>0$, there are $C,\gamma>0$ such that the following holds. For any bounded holomorphic weights $w_i:U\to\mathbb C$ satisfying $\sum_{i=0}^{m-1}\norm{w_i}_{L^\infty(U)}\leq B$, the operator
\[
\mathcal T:A^2(U)\to A^2(U),\qquad \mathcal Th\coloneqq\sum_{i=0}^{m-1}w_i(h\circ f_i)
\]
has Fredholm determinant $\det(I-z\mathcal T)=\sum_{j=0}^{\infty}\alpha_jz^j$, whose coefficients satisfy
\[
\abs{\alpha_j}\leq C\exp\left(-\gamma j^{1+1/d}\right)\qquad(j\geq1).
\]
\end{theorem}

\begin{proof}[Proof of Proposition~\ref{prop: Fredholm constellation has stationary defect decay}]
By the Pull-back theorem~\ref{thm: pull-back theorem} and the reproducing property of the Hardy kernel, every $u\in\Hscr_a$ with $\norm{u}_{\Hscr_a}\leq1$ has a holomorphic extension $\widetilde u:U\to\mathbb C$ satisfying $\norm{\widetilde u}_{\mathscr A}\leq C_a$, where $C_a>0$ is obtained by the compactness of $\chi(\overline U)$ and thus independent of $u$. Define $g:U\to\mathbb C$ by $g\coloneqq\widetilde u-\Lcal\widetilde u$. Since $\norm{\Lcal}_{\mathscr A\to\mathscr A}\leq1$ and $\nu$ is stationary, we have $\norm{g}_{\mathscr A}\leq2C_a$ and $\nu(g)=0$.

Equation~\eqref{eq: Fredholm stationary integral} gives $\partial_t\mathfrak D_g(1,0)=0$. Thus, by equations~\eqref{eq: Fredholm Lyapunov approximation} and \eqref{eq: Fredholm determinant method functional explicit},
\begin{equation} \label{eq: Fredholm stationary defect tail}
(\eta_N^{\mathrm F}-\Lcal^*\eta_N^{\mathrm F})(u)
=\eta_N^{\mathrm F}(g)=\mathcal I_N(g)
=-\frac{\sum_{j>N}\partial_t c_j(0;g)}{\sum_{j=1}^{N}jc_j}.
\end{equation}

We can extend $\Lcal_{t,g}$ to $\widetilde{\Lcal}_{t,g}: A^2(U)\to A^2(U)$ with the same expression. For $\abs{t}\leq1$,
\[
\sum_{i=0}^{m-1}\norm{\pi_i e^{tg\circ f_i}}_{\mathscr A}\leq e^{2C_a}.
\]
Then applying Theorem~\ref{thm: Fredholm coefficient estimate} with $B=e^{2C_a}$ bounds the coefficients for $\det(I-z\widetilde{\Lcal}_{t,g})$ on $A^2(U)$ uniformly in $u$ and $t$. Since $f_i(U)\Subset U$ by the choice of $U$, the map $\widetilde{\Lcal}_{t,g}$ sends $A^2(U)$ boundedly into $\mathscr A\subset A^2(U)$. Thus $\widetilde{\Lcal}_{t,g}$ and $\Lcal_{t,g}$ have the same nonzero generalized eigenspaces, and hence $\det(I-z\widetilde{\Lcal}_{t,g})=\mathfrak D_g(z,t)$. The boundedness of $g$ and equation~\eqref{eq: Fredholm determinant trace identity} allow $\mathfrak D_g(z,t)$ to extend analytically to $t\in\mathbb C$. Consequently, $\abs{c_j(t;g)}\leq C\exp(-\gamma j^{1+1/d})$ for $\abs{t}\leq1$, with $C,\gamma>0$ independent of $u$. Cauchy's estimate gives
\begin{equation} \label{eq: Fredholm determinant coefficients derivative bounds}
\abs{\partial_t c_j(0;g)}\leq C\exp\left(-\gamma j^{1+1/d}\right).
\end{equation}

The denominators in equation~\eqref{eq: Fredholm stationary defect tail} are nonzero for $N\geq N_{\mathrm{FD}}$ and converge to $\partial_z\mathfrak D_g(1,0)$, which is nonzero as $1$ is an algebraically simple eigenvalue of $\Lcal$. Their absolute values therefore have a positive lower bound. Using the reproducing property and equation~\eqref{eq: Fredholm determinant coefficients derivative bounds}, and taking the supremum over $\norm{u}_{\Hscr_a}\leq1$ yield the following, after adjusting $C$ and $\gamma$.
\[
\norm{\eta_N^{\mathrm F}-\Lcal^*\eta_N^{\mathrm F}}_{K_a}
=\sup_{\norm{u}_{\Hscr_a}\leq1}\abs{(\eta_N^{\mathrm F}-\Lcal^*\eta_N^{\mathrm F})(u)}
\leq C\exp\left(-\gamma N^{1+1/d}\right). \qedhere
\]
\end{proof}

\section{Comparison to the Fredholm determinant approach} \label{appendix: Comparison to Fredholm determinant approach}

Let $S_{\square}\subset\mathbb T^2$ be the Sierpi\'nski carpet: that is,
\[
S_{\square} \coloneqq
\left\{
\left( \sum_{n=1}^\infty \frac{a_n}{3^n}, \sum_{n=1}^\infty \frac{b_n}{3^n} \right) \in \mathbb T^2
\setcond
(a_n,b_n)\in \{0,1,2\}^2\setminus\{(1,1)\} \text{ for every } n\geq1
\right\}.
\]
By Kenyon--Peres \cite[Corollary~5.8]{Kenyon--Peres: translated Cantor sets} and Theorem~\ref{thm: polynomial computation}, for Lebesgue a.e.\ $t\in\mathbb T^2$,
\[
\dim_H\big(S_{\square}\cap(S_{\square}+t)\big)
=
\frac{\lambda}{\log 3}
=
1.7853035593787899023276390789798386893785338\cdots\footnote{See Remark~\ref{rmk: visualizer and reproducible materials} for a visualizer of the constellations for this example.},
%1.78530355937878990232763907897983868937853389011173\cdots
\]
where $\lambda$ is the Lyapunov exponent of the symmetric random products of the following matrices:
\begingroup
\small
\setlength{\arraycolsep}{3pt}
\[
\begin{gathered}
B_0=
\begin{pmatrix}
8&0&0&0\\
4&3&0&0\\
4&0&3&0\\
2&2&2&1
\end{pmatrix},
\;
B_1=
\begin{pmatrix}
4&3&0&0\\
3&4&0&0\\
2&2&2&1\\
2&2&1&2
\end{pmatrix},
\;
B_2=
\begin{pmatrix}
3&4&0&0\\
0&8&0&0\\
2&2&1&2\\
0&4&0&3
\end{pmatrix},
\;
B_3=
\begin{pmatrix}
4&0&3&0\\
2&2&2&1\\
3&0&4&0\\
2&1&2&2
\end{pmatrix},
\;
B_4=
\begin{pmatrix}
2&2&2&1\\
2&2&1&2\\
2&1&2&2\\
1&2&2&2
\end{pmatrix},
\\[1em]
B_5=
\begin{pmatrix}
2&2&1&2\\
0&4&0&3\\
1&2&2&2\\
0&3&0&4
\end{pmatrix},
\qquad
B_6=
\begin{pmatrix}
3&0&4&0\\
2&1&2&2\\
0&0&8&0\\
0&0&4&3
\end{pmatrix},
\qquad
B_7=
\begin{pmatrix}
2&1&2&2\\
1&2&2&2\\
0&0&4&3\\
0&0&3&4
\end{pmatrix},
\qquad
B_8=
\begin{pmatrix}
1&2&2&2\\
0&3&0&4\\
0&0&3&4\\
0&0&0&8
\end{pmatrix}.
\end{gathered}
\]
\endgroup
The matrices $A_i=P B_iP^{-1}$ are positive and invertible for every $i=0,\ldots,8$, where
\begingroup
\small
\setlength{\arraycolsep}{3pt}
\[
P=
\begin{pmatrix}
4&-1&-1&0\\
-1&4&0&-1\\
-1&0&4&-1\\
0&-1&-1&4
\end{pmatrix}.
\]
\endgroup
We can prove that $d_* = d = 3$. Also, we have $\max_i \kappa(A_i) = 6/7$.
%The matrix $P$ is obtained by choosing a symmetric tetrahedron $T$, invariant under every $f_{B_i}$, so that $\max_i\diam_{d_H}\!\left(PB_iP^{-1}\Delta\right)$ is small.

We apply our method and the Fredholm determinant method of \cite{Pollicott10} to $\{A_i\}_{i=0}^8$. Table~\ref{tab: constellations for Sierpinski carpets intersection} records the results for our constellation method using Theorem~\ref{thm: polynomial computation} with anchor $b=(1/4,1/4,1/4,1/4)$. Table~\ref{tab: Fredholm determinant method for Sierpinski carpets intersection} records the results for the Fredholm determinant method. Here $N$ is the truncation order of the Fredholm determinant as in \S\S\ref{subsec: The Fredholm determinant and its truncation}, and $\eta_N^{\mathrm{F}}$ is as in \S\S\ref{subsec: The resulting constellation exhibits decay in stationary defect}. The error bounds in both tables are obtained from the clock inequality (Theorem~\ref{thm: Clock inequality}) using interval arithmetic.

To achieve error $\eps>0$, the required arithmetic work is polynomial in $\log(1/\eps)$, or $O(n^3)$ for Table~\ref{tab: constellations for Sierpinski carpets intersection}, and subexponential in $\log(1/\eps)$, or $O(9^N)$ for Table~\ref{tab: Fredholm determinant method for Sierpinski carpets intersection}.

To carry out the large number of computations for Table~\ref{tab: Fredholm determinant method for Sierpinski carpets intersection}, we used family-specific symmetry to reduce the number of needed matrix products; at $N$, the number of raw matrix products of length $N$ is $9^N$, and if it is reduced to $a$, this is indicated as $9^N(a)$ in the ``Orbit'' column.

\begin{table}[!htbp]
\centering
\fontsize{9}{11}\selectfont
\setlength{\tabcolsep}{4pt}
\begin{tabular*}{\textwidth}{@{\extracolsep{\fill}}ll@{\hspace{14pt}}lll@{}}
\toprule
$n$ & $\#\supp\eta_n$ & $\eta_n(\Phi)$ & Observed error & Error bound \\
\midrule
4 & 5 & $1.9613672816\,1765635263\,0878746061\,0440406904\,1467517379$ & $1.09\times10^{-5}$ & $6.18\times10^{-4}$ \\
55 & 57 & $1.9613564293\,4563410269\,0498866274\,9767236370\,7872805186$ & $9.18\times10^{-12}$ & $1.05\times10^{-8}$ \\
150 & 152 & $1.9613564293\,3645185737\,4648693878\,4743870782\,9026358179$ & $2.99\times10^{-15}$ & $9.63\times10^{-13}$ \\
607 & 625 & $1.9613564293\,3645484337\,0855222448\,4798786711\,9112436002$ & $5.07\times10^{-24}$ & $7.94\times10^{-21}$ \\
982 & 1020 & $1.9613564293\,3645484337\,0860296418\,8970784383\,4021003982$ & $3.19\times10^{-28}$ & $9.90\times10^{-25}$ \\
2184 & 2245 & $1.9613564293\,3645484337\,0860296099\,5906677739\,0733112412$ & $1.08\times10^{-36}$ & $9.99\times10^{-33}$ \\
4109 & 4234 & $1.9613564293\,3645484337\,0860296099\,5906666913\,1283942757$ & $9.62\times10^{-45}$ & $9.99\times10^{-41}$ \\
% Unrounded n=6066 approximation (certificate midpoint): 1.96135642933645484337086029609959066669131284904394080056014968769594428093189803300896728937922415181330736223967447160022303503038082499812946780589439294766975
% Unrounded n=6066 certified error bound: 0.0000000000000000000000000000000000000000000001316124262077209102087197099043260623022963712400900731145925606179912784228185993494477013938005866685987616349624
6066 & 6281 & $1.9613564293\,3645484337\,0860296099\,5906666913\,1284904394$ & $7.53\times10^{-51}$ & $1.32\times10^{-46}$ \\
\midrule
$\infty$ & & $1.9613564293\,3645484337\,0860296099\,5906666913\,1284904394\mathrlap{\cdots}$ & & \\
\bottomrule
\end{tabular*}
\caption{Constellation approximations with certified error bounds rounded upward. The $\infty$ row records the true value (certified with the clock inequality).}
\label{tab: constellations for Sierpinski carpets intersection}
\end{table}

\begin{table}[!htbp]
\centering
\fontsize{9}{11}\selectfont
\setlength{\tabcolsep}{3pt}
\begin{tabular}{@{}l@{\hspace{18pt}}ll@{\hspace{12pt}}l@{\hspace{12pt}}l@{\hspace{11pt}}l@{}}
\toprule
$N$ & Orbit & $\#\supp\eta_N^{\mathrm{F}}$ & $\widehat\lambda_N = \eta_N^{\mathrm{F}}(\Phi)$ & Observed error & Error bound \\
\midrule
% N=1, Orbit=9, R=5
% 1.9608003131057206431912820487895032783853821456432342529296875
% Exact counts: N=1, total=9, symmetry representatives=6, support=5.
1 & $9(6)$ & $5$ & $1.9608003131\,0572064319\,1282048790$ & $5.56\times10^{-4}$ & $2.80\times10^{-2}$ \\
% N=2, Orbit=81, R=57
% 1.960544302162972170387560122417625052548828534781932830810546875
% Exact counts: N=2, total=81, symmetry representatives=27, support=57.
2 & $81(27)$ & $57$ & $1.9605443021\,6297217038\,7560122418$ & $8.12\times10^{-4}$ & $4.07\times10^{-2}$ \\
% N=3, Orbit=729, R=625
% 1.960579475398680516780129667875343102423357777297496795654296875
% Exact counts: N=3, total=729, symmetry representatives=130, support=625.
3 & $729(130)$ & $625$ & $1.9605794753\,9868051678\,0129667875$ & $7.77\times10^{-4}$ & $3.88\times10^{-2}$ \\
% N=4, Orbit=6561, R=6281
% 1.959582541666488885042339551656453977557248435914516448974609375
% Exact counts: N=4, total=6561, symmetry representatives=855, support=6281.
4 & $6561(855)$ & $6281$ & $1.9595825416\,6648888504\,2339551656$ & $1.77\times10^{-3}$ & $8.83\times10^{-2}$ \\
% N=5, Orbit=59049, R=60185
% 1.96098802493341657833016145051630019224830903112888336181640625
% Exact counts: N=5, total=59049, symmetry representatives=5934, support=60185.
5 & $59049(5934)$ & $60185$ & $1.9609880249\,3341657833\,0161450516$ & $3.68\times10^{-4}$ & $1.83\times10^{-2}$ \\
% N=6, Orbit=531441, R=561905
% 1.961452744408717389023276422843622412983677349984645843505859375
% Exact counts: N=6, total=531441, symmetry representatives=44487, support=561905.
6 & $531441(44487)$ & $561905$ & $1.9614527444\,0871738902\,3276422844$ & $9.63\times10^{-5}$ & $4.76\times10^{-3}$ \\
% N=7, Orbit=4782969, R=5174265
% 1.961345847353281576287177545570017400677897967398166656494140625
% Exact counts: N=7, total=4782969, symmetry representatives=341802, support=5174265.
7 & $\mathllap{{\scriptstyle\approx}\,}4.78\times10^{6}(3.42\times10^{5})$ & $5174265$ & $1.9613458473\,5328157628\,7177545570$ & $1.06\times10^{-5}$ & $5.21\times10^{-4}$ \\
% N=8, Orbit=43046721, R=47240649
% 1.961357377745979746756950501396232766637695021927356719970703125
% Exact counts: N=8, total=43046721, symmetry representatives=2691675, support=47240649.
8 & $\mathllap{{\scriptstyle\approx}\,}4.30\times10^{7}(2.69\times10^{6})$ & $47240649$ & $1.9613573777\,4597974675\,6950501396$ & $9.48\times10^{-7}$ & - \\
% N=9, Orbit=387420489, R=429047857
% 1.961356365403767380568671174767558795792865566909313201904296875
% Exact counts: N=9, total=387420489, symmetry representatives=21524542, support=429047857.
9 & $\mathllap{{\scriptstyle\approx}\,}3.87\times10^{8}(2.15\times10^{7})$ & $429047857$ & $1.9613563654\,0376738056\,8671174768$ & $6.39\times10^{-8}$ & - \\
% N=10, Orbit=3486784401, R=3883808321
% 1.96135643263951556735345815862814333740971051156520843505859375
% Exact counts: N=10, total=3486784401, symmetry representatives=174348099, support=3883808321.
10 & $\mathllap{{\scriptstyle\approx}\,}3.49\times10^{9}(1.74\times10^{8})$ & $3883808321$ & $1.9613564326\,3951556735\,3458158628$ & $3.30\times10^{-9}$ & - \\
% N=11, Orbit=31381059609, R=35082812825
% 1.961356429205458177448921641694568052116665057837963104248046875
% Exact counts: N=11, total=31381059609, symmetry representatives=1426419858, support=35082812825.
11 & $\mathllap{{\scriptstyle\approx}\,}3.14\times10^{10}(1.43\times10^{9})$ & $35082812825$ & $1.9613564292\,0545817744\,8921641695$ & $1.31\times10^{-10}$ & - \\
% N=12, Orbit=282429536481, R=316481022121
% 1.9613564293404165774815928191543434877530671656131744384765625
% Exact counts: N=12, total=282429536481, symmetry representatives=11767964475, support=316481022121.
12 & $\mathllap{{\scriptstyle\approx}\,}2.82\times10^{11}(1.18\times10^{10})$ & $316481022121$ & $1.9613564293\,4041657748\,1592819154$ & $3.96\times10^{-12}$ & - \\
\midrule
$\infty$ & & & $1.9613564293\,3645484337\,0860296099\mathrlap{\cdots}$ & & \\
\bottomrule
\end{tabular}
\caption{Fredholm approximations with certified error bounds calculated using the clock inequality. The $\infty$ row records the true value. The dash denotes an uncalculated bound due to limits in compute resources.}
\label{tab: Fredholm determinant method for Sierpinski carpets intersection}
\end{table}

\begin{remark} \label{rmk: visualizer and reproducible materials}
For the random intersection of Sierpi\'nski carpets considered in Appendix~\ref{appendix: Comparison to Fredholm determinant approach}, a visualizer showing how the constellations evolve is available at \url{https://constellations-for-lyapunov-exponents-visualizer.nimaalibabaei.workers.dev/}.

The constellations used to certify the numerical values displayed in Examples~\ref{ex: the random Fibonacci sequence} and~\ref{ex: Sierpinski gaskets}, as well as the Hausdorff dimension of the random intersection of Sierpi\'nski carpets considered in Appendix~\ref{appendix: Comparison to Fredholm determinant approach}, are available at \url{https://github.com/NimaAlibabaei/Constellations-for-Lyapunov-exponents}. Applying the clock inequality (Theorem~\ref{thm: Clock inequality}) to these constellations reproduces the certified digits. These constellations were obtained using an implementation of our construction written by ChatGPT (OpenAI). The author has verified the application of the clock inequality to these constellations and the resulting certified values.
\end{remark}

\section*{Acknowledgements}
I would like to thank my supervisor, Masaki Tsukamoto, for his steady support throughout the research process.

The development of this paper has been a collaborative effort with ChatGPT (OpenAI). The author formulated the central idea leading to the clock inequality and developed the final construction of constellations, following many unsuccessful attempts by both the author and ChatGPT. The main contributions of ChatGPT were the initial ideas underlying the proof of the lower-exponent assertion in Proposition~\ref{prop: optimal cardinality is the constellation dimension}, the proof of Theorem~\ref{thm: invertible subfamily lower bounds constellation dimension}, and Example~\ref{ex: Liouville reset}. All ideas from ChatGPT have been carefully digested, refined, and incorporated by the author.

\vspace{0.5cm}

\address{
Department of Mathematics, Kyoto University, Kyoto 606-8501, Japan}

\textit{E-mail}: \texttt{alibabaei.nima.28c@st.kyoto-u.ac.jp}

This work was supported by JSPS KAKENHI Grant Number 25KJ1473.

\end{document}